\documentclass[10pt]{article}

\usepackage{amsmath,amsthm,amsfonts,amssymb,bm,bbm,enumerate, graphicx, mathtools,color,epstopdf}
\usepackage{tcolorbox}
\usepackage[english]{babel}
\usepackage[utf8]{inputenc}
\usepackage{fancyhdr}
\usepackage{caption}
\usepackage{subcaption}
\usepackage[margin=1in]{geometry}

\newcommand{\pr}[1]{\mathbb{P}\!\left(#1\right)}
\newcommand{\E}[1]{\mathbb{E}\!\left[#1\right]}

\newcommand{\prstart}[2]{\mathbb{P}_{#2}\!\left(#1\right)}

\newcommand{\prcond}[3]{\mathbb{P}_{#3}\!\left(#1\;\middle\vert\;#2\right)}

\newcommand{\prb}[1]{\mathbf{P}\!\left(#1\right)}
\newcommand{\Eb}[1]{\mathbf{E}\!\left[#1\right]}
\newcommand{\estartb}[2]{\mathbf{E}_{#2}\!\left[#1\right]}

\newcommand{\prcondb}[3]{\mathbf{P}_{#3}\!\left(#1\;\middle\vert\;#2\right)}

\newcommand{\dGH}[1]{d_{GH}\!\left(#1\right)}

\newtheorem{theorem}{Theorem}[section]
\newtheorem{lem}[theorem]{Lemma}
\newtheorem{prop}[theorem]{Proposition}
\newtheorem{cor}[theorem]{Corollary}
\newtheorem{defn}[theorem]{Definition}

\newtheorem{rmk}[theorem]{Remark}

\newcommand{\dGHPp}[1]{d_{GHP}\!\left(#1\right)}

\def\N{\mathbb{N}}

\def\L{\mathbb{L}}

\def\R{\mathbb{R}}

\def\F{\mathcal{F}}
\def\EE{\mathcal{E}}

\def\X{X^{\text{exc}}}
\def\Xb{X^{\text{br}}}
\def\HH{H^{\text{exc}}}
\renewcommand{\phi}{\varphi}
\renewcommand{\epsilon}{\varepsilon}

\def\T{\mathcal{T}}
\def\Levy{L\'{e}vy }
\def\Ito{It\^o }
\def\L{\mathcal{L}}
\def\cadlag{c\`{a}dl\`{a}g }
\def\Loop'{\textsf{Loop'}}
\def\Loop{\textsf{Loop}}

\def\exc{\text{exc}}

\def\osc{\textsf{Osc}}
\def\diam{\textsf{Diam}}
\def\lr{(\log r^{-1})}

\def\La{\L_{\alpha}}
\def\Ta{\T_{\alpha}}

\def\Sr2{S^{(\frac{1}{2}r)}_{\sigma}}

\def\Hm{H_{\text{max}}}

\def\Lai{\L^{\infty}_{\alpha}}

\def\PD{\textsf{PD}}

\def\bf{\beta_1}
\def\bE{\beta_3}
\def\bd{\beta_4}
\def\bg{\beta_2}
\def\bPb{\mathbf{P}}
\allowdisplaybreaks

\begin{document}

\title{Brownian motion on stable looptrees}
\author{Eleanor Archer\thanks{Mathematics Institute, University of Warwick, Coventry CV4 7AL, United Kingdom. Email: {E.Archer.1@warwick.ac.uk}.}}
\maketitle

\begin{abstract}
In this article, we introduce Brownian motion on $\alpha$-stable looptrees using resistance techniques, where $\alpha \in (1,2)$. We prove an invariance principle characterising it as the scaling limit of random walks on discrete looptrees, and prove precise local and global bounds on its heat kernel. We also conduct a detailed investigation of the volume growth properties of stable looptrees, and show that the random volume and heat kernel fluctuations are locally log-logarithmic, and globally logarithmic around leading terms of $r^{\alpha}$ and $t^{\frac{-\alpha}{\alpha + 1}}$ respectively. These volume fluctuations are the same order as for the Brownian continuum random tree, but the upper volume fluctuations (and corresponding lower heat kernel fluctuations) are different to those of stable trees.
\end{abstract}
%


\section{Introduction}
Stable looptrees are a class of random fractal objects indexed by a parameter $\alpha \in (1,2)$ and can informally be thought of as the dual graphs of stable trees. Motivated by \cite{LeGMiermontScalingLimitsLargeFaces}, they were originally introduced by Curien and Kortchemski in \cite{RSLTCurKort}, and along with their discrete counterparts have been shown to be of increasing significance in the study of statistical mechanics models on random planar maps. For example, the same authors showed in \cite{CurKortUIPTPerc} that a stable looptree arises as the scaling limit of the boundary of a critical percolation cluster on the UIPT, and Richier showed in \cite{RichierIICUIHPT} that the incipient infinite cluster of the UIHPT has the form of an infinite discrete looptree. Further results along these lines can be found in \cite{CurKortUIPTPerc, CurKortDuqMan, StefStufBolzOuterplanar, BaurRichUIPQSkew, CurRichDualityRPMPerc, KortRichBoundaryRPMLooptrees}, though this is a very non-exhaustive list. More generally, they also arise as the scaling limits of boundaries of stable maps \cite{RichierMapBoundaryLimit}, are closely connected to the shredded stable spheres constructed in \cite{BjornCurStef}, and are emerging as an important tool in the programme to reconcile the theories of random planar maps and Liouville quantum gravity \cite{MillSheff, GwynnePfefferConnectivitySLE, BernardiHoldenSun}.

\begin{figure}[h]
\includegraphics[width=14cm]{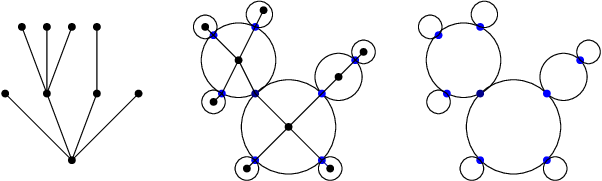}
\centering
\caption{A tree $T$ and the corresponding looptree \Loop($T$).}\label{fig:disc looptree intro}
\end{figure}

Stable looptrees can be formally defined from stable \Levy excursions but a key result of \cite{RSLTCurKort} is an invariance principle characterising them as the scaling limit of discrete looptrees.
More precisely, given a discrete tree $T$, the corresponding discrete looptree ${\Loop}(T)$ as defined in \cite{RSLTCurKort} is constructed by replacing each vertex $u \in T$ with a discrete cycle of length equal to the degree of $u$ in $T$, and then gluing these cycles along the tree structure of $T$. Each edge of $T$ then naturally corresponds to a vertex of $\Loop(T)$ as illustrated in Figure \ref{fig:disc looptree intro}. Two vertices in $\Loop(T)$ are adjacent if and only they correspond to edges in $T$ joining a vertex $v \in T$ to two of its consecutive offspring, or else one vertex corresponds to the edge joining to $v$ to its parent, and the other vertex corresponds to the edge joining $v$ to its first or last child. We give every edge in $\Loop(T)$ unit length, and the root of $\Loop(T)$ corresponds to the edge of $T$ joining its root to its first child. This operation can also be applied in the case where $T$ is an infinite tree.

In this article we will be interested in the case where $T$ has a critical offspring distribution in the domain of attraction of an $\alpha$-stable law, by which we mean that there exists an increasing sequence $a_n \uparrow \infty$ such that, if $(\xi^{(i)})_{i=1}^{\infty}$ are i.i.d. copies of $\xi$, then
\begin{equation}\label{eqn:dom of att def}
\frac{\sum_{i=1}^n \xi^{(i)} - n}{a_n} \overset{(d)}{\rightarrow} Z_{\alpha}
\end{equation}
as $n \rightarrow \infty$, where $Z_{\alpha}$ is an $\alpha$-stable random variable. In the case $\alpha < 2$, this is equivalent to saying that $\xi([n, \infty)) = n^{-\alpha} L(n)$ for some slowly-varying function $L$.

Throughout the article we will make the assumption that $\alpha \in (1,2)$, $\xi$ is an offspring distribution in the domain of attraction of an $\alpha$-stable law, and let $(a_n)_{n=1}^{\infty}$ be the sequence appearing in (\ref{eqn:dom of att def}). In \cite[Theorem 4.1]{RSLTCurKort}, it is shown that if $T_n$ is a Galton Watson tree conditioned to have $n$ vertices with offspring distribution $\xi$, then we can define the $\alpha$-stable looptree (denoted $\L_{\alpha}$) to be the random compact metric space such that 
\[
a_n^{-1} {\Loop}(T_n) \overset{(d)}{\rightarrow} \L_{\alpha}
\]
in the Gromov-Hausdorff topology as $n \rightarrow \infty$, where $c \cdot M$ denotes the metric space obtained from $M$ by multiplying all distances by $c>0$. A simulation is shown in Figure \ref{fig:stable looptree intro}. In the case $\alpha = 2$, the looptrees instead rescale to the Brownian Continuum Random Tree \cite[Theorem 2]{KortRichCondensationCritical}.

The purpose of this article is to introduce and study Brownian motion on stable looptrees, and we start in Section \ref{sctn:limit results} by proving a similar invariance principle that identifies it as the scaling limit of random walks on discrete looptrees, given below. As a consequence, it also follows that the rescaled transition densities and mixing times converge respectively to those of the limiting Brownian motion.

\begin{theorem}\label{thm:main RW inv princ compact quenched}
Let $T_n$ be as above, let $Z^{(n)}$ denote a discrete-time simple random walk on $\Loop (T_n)$, and let $(B_t)_{t \geq 0}$ denote Brownian motion on $\La$. There exists a probability space $(\Omega', \mathcal{F}', \mathbf{P}')$ on which we can (pointwise) define isometric embeddings of $(a_n^{-1}\Loop(T_n))_{n \geq 1}$ and $\La$ into a common metric space $(M, d_M)$ so that
\[
a_n^{-1}\Loop(T_n) \rightarrow \La
\]
almost surely with respect to the Hausdorff metric. In this metric space, we also have that
\[
\Big(a_n^{-1} Z^{(n)}_{\lfloor 4n a_nt \rfloor}\Big)_{t \geq 0} \overset{(d)}{\rightarrow} (B_t)_{t \geq 0}
\]
as $n \rightarrow \infty$, by which we mean that, almost surely on $(\Omega', \mathcal{F}', \mathbf{P}')$, the laws of these processes converge weakly on the space $D([0, \infty), M$) endowed with the uniform topology.
\end{theorem}

In fact, we prove a slightly more general version of the theorem that holds for any sequence of discrete trees satisfying the assumptions of \cite[Theorem 4.1]{RSLTCurKort}, but we are mainly interested in applying it in the stable case. The Brownian motion $(B_t)_{t \geq 0}$ is constructed via the theory of Dirichlet forms and resistance metrics using the now classical theory of \cite{AOF}. The bulk of this article is then devoted to a detailed study of the resistance volume growth of stable looptrees, from which we obtain heat kernel estimates using results of \cite{CroyHKFluctNonUn}. The volume growth results also have implications for the Hausdorff and packing measures of stable looptrees with respect to certain gauge functions, for which we prove results analogous to those proved by Duquesne, Le Gall and Wang for stable trees in \cite{DuqLeGPFALT}, \cite{DuqLeGHausdorffStable}, \cite{DuqWangSmallBallsStable} and \cite{DuqPackingHausdorff}, and by Croydon for the Brownian continuum random tree (CRT) in \cite{DavidCRT}. Additionally, the results imply that the packing dimension of $\La$ is equal to $\alpha$, which is the same as the Hausdorff dimension that was established in \cite{RSLTCurKort}.

\begin{figure}[h]
\includegraphics[width=10cm]{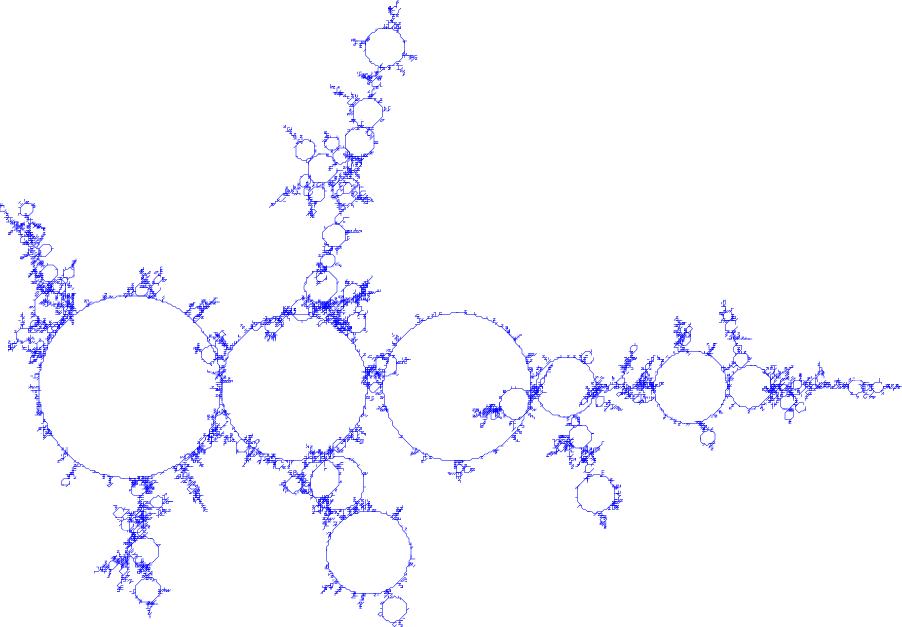}
\centering
\caption{Simulation of a stable looptree, copied from \cite{RSLTCurKort}.}\label{fig:stable looptree intro}
\end{figure}

In the past, resistance growth analysis has mainly been useful in studying random walks on trees (for example \cite{DavidCRT}, \cite{CroyHamSpectralCRT}, \cite{CroyHamSpectralStable}), since in this case the resistance metric and the geodesic metric are the same. In the case of looptrees the two metrics are different, but we will show that they are equivalent, which allows us to use the two metrics interchangeably when proving the volume bounds.

We will use two main approaches to prove the looptree volume bounds. One approach, used to prove most of the volume lower bounds in this article, builds on ideas of \cite{RSLTCurKort} by comparing looptree volume fluctuations with fluctuations in the \Levy excursion that code them. This comparison cannot be used to prove upper bounds however, and a substantial part of this article is devoted to introducing an iterative decomposition of stable looptrees that we use to prove the upper volume bounds. The procedure utilises the Williams' decomposition of stable trees given in \cite{AbDelWilliamsDecomp} to decompose $\La$ along a loopspine, breaking it into smaller fragments which are all smaller rescaled looptrees. We then reapply the decomposition to these resulting fragments, and continue to repeat the decomposition on the fragments we obtain each time. This procedure can be realised as a separate branching process, which we will analyse in Sections \ref{sctn:sup vol UBs} and \ref{sctn:inf UBs} to prove the upper volume bounds.

We now summarise the volume bound results. We will give proper definitions of all the quantities involved in Section \ref{sctn:tree looptree def}, but for now we note that $\nu$ denotes the natural analogue of uniform volume measure on stable looptrees. Due to the equivalence of metrics, these results will hold regardless of whether we define the open ball $B(u,r)$ (and its closure $\bar{B}(u,r)$) using the shortest distance metric or the effective resistance metric. As in \cite{RSLTCurKort}, we denote the $\alpha$-stable looptree by $\La$, and its root by $\rho$. We assume that our looptree is defined on the probability space $(\Omega, \F, \mathbf{P})$, and let $\mathbf{E}$ denote expectation on this space.

We start with the following global (uniform) volume bounds for small balls in $\La$, which demonstrate both upper and lower fluctuations of logarithmic order.

\begin{theorem}\label{thm:main global}
$\mathbf{P}$-almost surely, there exist constants $C_1, C_2 \in (0, \infty)$ such that for all $r \in (0, \diam (\L_{\alpha}))$:

\begin{minipage}{0.5\linewidth} 
\begin{equation}\label{eqn:glob inf LB}  \inf_{u \in \L_{\alpha}} \nu \big( B(u,r) \big) \geq C_1 r^{\alpha} (\log r^{-1})^{-\alpha}  \end{equation} 
\end{minipage} 
\begin{minipage}{0.5\linewidth}  \begin{equation}\label{eqn:glob sup UB}  \sup_{u \in \La} \nu \big(B(u, r)\big) \leq C_2 r^{\alpha} (\log r^{-1})^{\frac{4\alpha - 3}{{\alpha - 1}}} \end{equation}  \end{minipage}
\begin{minipage}{0.5\linewidth} \begin{equation}\label{eqn:glob sup LB}  \limsup_{r \downarrow 0} \Bigg( \frac{ \sup_{u \in \La} \nu (B(u, r))}{r^{\alpha} \log r^{-1}} \Bigg) > 0.  \end{equation}  
\end{minipage} 
\begin{minipage}{0.5\linewidth}  \begin{equation}\label{eqn:glob inf UB}  \liminf_{r \downarrow 0} \Bigg( \frac{\inf_{u \in \La}\nu(B(u, r))}{r^{\alpha} \lr^{-(\alpha - 1)}} \Bigg) < \infty  \end{equation}
\end{minipage}
\end{theorem}

We also have the following local (pointwise) results.

\begin{theorem}\label{thm:main local}
$\mathbf{P}$-almost surely, for $\nu$-almost every $u \in \La$ we have:

\begin{minipage}{0.5\linewidth} 
\begin{equation}\label{eqn:loc inf LB} \liminf_{r \downarrow 0} \Bigg( \frac{\nu(B(u, r))}{r^{\alpha}(\log \log r^{-1})^{-\alpha}} \Bigg) > 0 \end{equation}
\end{minipage} 
\begin{minipage}{0.5\linewidth}  
\begin{equation}\label{eqn:loc sup UB} \limsup_{r \downarrow 0} \Bigg( \frac{\nu(B(u, r))}{r^{\alpha} (\log \log r^{-1})^{\frac{4\alpha - 3}{{\alpha - 1}}}} \Bigg) < \infty  \end{equation}
  \end{minipage}
\begin{minipage}{0.5\linewidth} 
\begin{equation}\label{eqn:loc sup LB}  \limsup_{r \downarrow 0} \Bigg( \frac{\nu (B(u, r))}{r^{\alpha} \log \log r^{-1}} \Bigg) > 0 \end{equation}
\end{minipage} 
\begin{minipage}{0.5\linewidth}  \begin{equation}\label{eqn:loc inf UB}  \liminf_{r \downarrow 0} \Bigg( \frac{\nu (B(u, r))}{r^{\alpha} (\log \log r^{-1})^{-(\alpha - 1)}} \Bigg) < \infty. \end{equation}  \end{minipage}

\end{theorem}

We remark here that the log-logarithmic fluctuations are the same order (up to exponents) as those obtained for a certain class of random recursive fractals in \cite{HamJonesThinThickRandomRecursive}, and specifically the same as those obtained for the Brownian CRT in \cite[Theorem 1.3]{DavidCRT}. However, the upper volume fluctuations contrast with those for stable trees which were shown to be logarithmic in \cite[Theorem 1.4]{DuqLeGHausdorffStable} when $\alpha \in (1,2)$. Intuitively, this is because denser points in stable trees are spread out by larger loops in stable looptrees, creating a more uniform spread of mass. However, the lower fluctuations for stable trees are also log-logarithmic (see \cite[Theorem 1.1]{DuqExactPackingLevy} and \cite[Theorem 1.2]{DuqWangSmallBallsStable}). As in \cite{DuqLeGHausdorffStable} and \cite{DuqExactPackingLevy}, our results can also be interpreted to give precise bounds on possible gauge functions for exact Hausdorff and packing measures.

The results of Theorem \ref{thm:main global} show that stable looptrees almost surely satisfy the assumptions of \cite[Equation 1.2]{CroyHKFluctNonUn}, and so we can apply the results of that article to deduce that the transition density of the Brownian motion $(B_t)_{t \geq 0}$ almost surely exists, and study its properties. More precisely, the transition density $p_t(\cdot, \cdot)$ is a symmetric $\nu \times \nu$-measurable function on $\La \times \La$ such that
\[
\estartb{f(B_t)}{x} = \int_{\La} f(y) p_t(x,y) \nu (dy)
\]
for all bounded, $\nu$-measurable functions $f$ on $\La$ and all $x \in \La$. (For reference, a heat kernel is any similar integral kernel so in general is only defined up to a $\nu$-null set, but in our case, the transition density will be continuous so we will use the two terms interchangeably).

One particular quantity of interest is the spectral dimension of $\La$, defined by
\begin{equation*}
d_S(\La) = \lim_{t \rightarrow \infty} \frac{-2\log( p_t(\rho, \rho))}{\log t}.
\end{equation*}
Using \cite{CroyHKFluctNonUn}, we show that $d_S(\La) = \frac{2\alpha}{\alpha + 1}$ almost surely, and also go a step further to obtain the following quenched bounds on the transition density. Here $\gamma_1$ is a deterministic constant, dependent on $\alpha$, that we will write down explicitly in Section \ref{sctn:HK estimates compact looptrees}.

\begin{theorem}\label{thm:main HK global}
$\mathbf{P}$-almost surely, there exist $t_0, C_3, C_4 \in (0, \infty)$ such that
\begin{align*}
C_3 t^{\frac{-\alpha}{\alpha + 1}} (\log t^{-1})^{-\gamma_1} \leq p_t(x,x) \leq C_4 t^{\frac{-\alpha}{\alpha + 1}} (\log t^{-1})^{\alpha}
\end{align*}
for all $x \in \La$ and all $t \in (0, t_0)$. Moreover, it holds $\mathbf{P}$-almost surely that
\begin{align*}
\liminf_{t \downarrow 0} \frac{\inf_{x \in \La} p_t(x,x) }{t^{\frac{-\alpha}{\alpha + 1}}(\log t^{-1})^{-1}} < \infty, \ \ \ \ \ \limsup_{t \downarrow 0} \frac{\sup_{x \in \La} p_t(x,x)}{t^{\frac{-\alpha}{\alpha + 1}}(\log t^{-1})^{\alpha - 1}} > 0.
\end{align*}
\end{theorem}

We also use the local volume bounds of Theorem \ref{thm:main local} to deduce pointwise heat kernel estimates. Note however that one of the lower bounds in Theorem \ref{thm:main HK local} is missing. Heat kernel lower bounds are generally more subtle to obtain than upper bounds, and in particular in this case we would need some additional global volume control to apply the chaining arguments of \cite{CroyHKFluctNonUn} that are used to prove the corresponding global bound in Theorem \ref{thm:main HK global}.

\begin{theorem}\label{thm:main HK local}
$\mathbf{P}$-almost surely, for any $\epsilon>0$ we have for $\nu$-almost every $x \in \La$ that
\begin{align*}
\liminf_{t \downarrow 0} \frac{p_t(x,x)}{t^{\frac{-\alpha}{\alpha +1}}(\log \log t^{-1})^{\frac{-1}{\alpha + 1}}} < \infty, \ \ \ \limsup_{t \downarrow 0}  \frac{p_t(x,x)}{t^{\frac{-\alpha}{\alpha +1}}(\log \log t^{-1})^{\frac{\alpha}{\alpha + 1}}} &< \infty, \\
\limsup_{t \downarrow 0}  \frac{p_t(x,x)}{t^{\frac{-\alpha}{\alpha +1}}(\log \log t^{-1})^{\frac{\alpha-1-\epsilon}{\alpha + 1}}} &> 0.
\end{align*}
\end{theorem}

We can similarly apply the results of \cite{DavidCRT} to get off diagonal heat kernel bounds. Once again, $\gamma_2$ and $\gamma_3$ are deterministic constants (dependent on $\alpha$) and we will give their explicit values in Section \ref{sctn:HK estimates compact looptrees}.

\begin{theorem}\label{thm:main HK off diag}
$\mathbf{P}$-almost surely, there exist $t_0', C_5, C_6, C_7, C_8 \in (0, \infty)$ such that for all $x, y \in \La$ and all $t \in (0, t_0')$, we have
\begin{align*}
p_t(x,y) &\leq C_5 t^{\frac{-\alpha}{\alpha + 1}} (\log t^{-1})^{\alpha} \exp \{- C_6 {\tilde{d}}^{1+\frac{1}{\alpha}} t^{\frac{-1}{\alpha}} (\log t^{-1}\tilde{d})^{-\gamma_3} \}, \\
p_t(x,y) &\geq 
C_7 t^{\frac{-\alpha}{\alpha + 1}} (\log t^{-1})^{-\gamma_1} \exp \{ -C_8 {\tilde{d}}^{1+\frac{1}{\alpha}} t^{\frac{-1}{\alpha}} (\log t^{-1}\tilde{d})^{\gamma_2}  \}.
\end{align*}
Here $\tilde{d} = \tilde{d}(x,y)$ can denote the distance between $x$ and $y$ with respect to either the shortest distance metric on $\La$, or the effective resistance metric.
\end{theorem}

A key step in these heat kernel estimates are bounds on the expected exit times from balls, which we will consider in Section \ref{sctn:HK estimates compact looptrees}. Finally, we give an annealed result for the transition density at the root, averaged over the law of $\La$. This also implies that the annealed spectral dimension is $\frac{2\alpha}{\alpha + 1}$.

\begin{theorem}\label{thm:main annealed HK conv}
There exists $C_9 \in (0, \infty)$ such that
\[
t^{\frac{\alpha}{\alpha + 1}} \Eb{p_t(\rho, \rho)} {\rightarrow} C_9
\]
as $t \downarrow 0$.
\end{theorem}

In light of these results, it also natural to investigate the associated eigenvalue counting function of the Laplacian $\Delta$ associated with $(B_t)_{t \geq 0}$. More precisely, let $R$ be the effective resistance metric on $\La$ (we will contruct this properly in Section \ref{sctn:limit results}), and let $(\EE, \F)$ be the Dirichlet form associated with the space $L^2 (\La, \nu)$ through the relation 
\[
R(x,y)^{-1} = \inf \{\EE(f,f): f \in \F, f(x)=0, f(y)=1 \}
\]
(this is $\bPb$-almost surely well-defined: see \cite[Section 1]{CroyHamSpectralStable} for more details on the construction for stable trees; the same principles apply for stable looptrees).
We say that $\lambda > 0$ is an eigenvalue of $(\EE, \F, \nu)$ with eigenfunction $f$ (assumed to be non-trivial) if
\[
\EE(f, g) = \lambda \int_{\La} fg \ d\nu
\]
for all $g \in \F$. The eigenvalue counting function $N(\lambda)$ is then defined as the number of eigenvalues of $(\EE, \F, \nu)$ that are less than or equal to $\lambda$. Due to the representation $\EE(f,g) = - \int_{\La} (\Delta f) g \ d\nu$, any eigenvalue of the operator $\Delta$ is also an eigenvalue of $(\EE, \F, \nu)$. Since $\La$ is compact and $(\EE, \F)$ is consequently regular, the converse also holds. Similar arguments to \cite{CroyHamSpectralStable} then lead to the following result.

\begin{theorem}\label{thm:spec asymp}
\begin{enumerate}[(i)]
\item For any $\epsilon > 0$, as $\lambda \rightarrow \infty$,
\[
\Eb{N(\lambda)} \sim C\lambda^{\frac{\alpha}{\alpha + 1}} + O (\lambda^{\frac{1}{\alpha + 1} + \epsilon}).
\]
\item $\bPb$-almost surely, $N(\lambda) \sim C \lambda^{\frac{\alpha}{\alpha + 1}}$ as $\lambda \rightarrow \infty$. More over, in $\bPb$-probability, the second order estimate of part (i) holds.
\end{enumerate}
\end{theorem}

Richier showed in \cite{RichierIICUIHPT} that the incipient infinite cluster (IIC) of the Uniform Infinite Half-Planar Triangulation (UIHPT) has the structure of a discrete looptree, but where each of the loops are filled with independent critically percolated Boltzmann triangulations. The size of the loops of this looptree are given by a distribution in the domain of attraction of a $\frac{3}{2}$-stable law and the results of our companion paper \cite{ArchInfiniteLooptrees} imply that the boundary of this cluster converges after rescaling to the infinite stable looptree $\L_{3/2}^{\infty}$. The question of the scaling limit of the whole cluster is more subtle but is conjectured to be the $\frac{7}{6}$-stable map \cite[Section 5.4]{BerCurMierPerconTriang}, and we hope the methods used in this article will be a good starting point for studying random walks on the IIC. In particular, we anticipate that such a random walk might fall into a framework similar to the discussions of \cite{AlRuiFreiKigSSG}, in that the looptree forming the boundary of the IIC may play a role analogous to that of the classical Sierpinski gasket in that article. If this is the case, then understanding random walks on looptrees is a crucial first step to understanding a random walk on the IIC.

Random walks on random infinite discrete looptrees were also studied by Bj\"ornberg and Stef\'ansson in \cite{BjornStef} using a generating function approach. As we also prove for $\Lai$ in \cite{ArchInfiniteLooptrees}, they prove that both the annealed and quenched spectral dimensions of a discrete infinite looptree with critical offspring distribution in the domain of attraction of an $\alpha$-stable law are equal to $\frac{2\alpha}{\alpha + 1}$. Their arguments also exploit the link with resistance growth properties of the space and they show that the volume of a typical ball of radius $r$ around the root almost surely undergoes at most logarithmic volume fluctuations around a leading term $r^{\alpha}$ as $r \rightarrow \infty$. This also gives logarithmic upper and lower bounds on the quenched and annealed transition density at the root. The exponential tail bound in equation (3.18) of their paper suggests however that their volume lower bound fluctuations can also be improved to log-logarithmic order, and we envisage that the approaches of this article can also be applied to their discrete case to give log-logarithmic upper and lower bounds on the volume and transition density fluctuations at typical points. This will be shown in the upcoming article \cite{ArchRWDecTrees}.

Finally, in \cite{StefStufBolzOuterplanar}, Stef\'ansson and Stufler show that stable looptrees also arise as scaling limits of outerplanar maps under appropriate conditions on their face weights. Their proof can be adapted to show that this convergence also holds with respect to the resistance metric, with the same scaling exponents, but just with a different scaling constant in front of the metric compared to that of \cite[Theorem 3.2]{StefStufBolzOuterplanar}. Moreover, Proposition \ref{thm:compact disc inv princ res} of this paper allows us to extend this to full Gromov-Hausdorff-Prohorov convergence on endowing the outerplanar maps with a uniform measure on their vertices. As a result, we deduce a similar scaling limit for random walks on these outerplanar maps, analogous to Theorem \ref{thm:main RW inv princ compact quenched} (random walks on outerplanar maps will also be treated more explicitly in \cite{ArchRWDecTrees}).

This paper is organised as follows. In Section \ref{sctn:prelim}, we introduce some of the technical background used throughout the article. In Section \ref{sctn:tree looptree def}, we give formal definitions stable trees and looptrees. In Section \ref{sctn:limit results}, we define a resistance metric on stable looptrees, use this to give a construction of the Brownian motion $(B_t)_{t \geq 0}$ and prove the invariance principle of Theorem \ref{thm:main RW inv princ compact quenched}, along with similar convergence results for associated quantities such as transition densities and mixing times. We then prove Theorems \ref{thm:main vol conv} to \ref{thm:main local} in Section \ref{sctn:vol bounds}. This is the most substantial section of the paper and is also where we introduce the iterative decomposition procedure mentioned above. Finally, we conclude in Section \ref{sctn:HK estimates compact looptrees} by proving the heat kernel estimates of Theorems \ref{thm:main HK global} to \ref{thm:main annealed HK conv}, and the spectral result of Theorem \ref{thm:spec asymp}.

Throughout this paper, $C, C', c$ and $c'$ will denote constants, bounded above and below, that may change on appearance.

\textbf{Acknowledgements.} I would like to thank my supervisor David Croydon for suggesting the problem and for many helpful discussions, as well as the anonymous referees for their detailed and helpful comments on the initial version of this manuscript. I would also like to thank the Great Britain Sasakawa Foundation for supporting a trip to Kyoto during which some of this work was completed, and Kyoto University for their hospitality during this trip.

\section{Preliminaries}\label{sctn:prelim}

\subsection{Gromov-Hausdorff-Prohorov Topologies}
In order to prove convergence results for compact measured metric spaces such as looptrees we will work in the pointed Gromov-Hausdorff-Prohorov topology. Accordingly, let $\mathbb{F}^c$ denote the set of quadruples $(F,R,\mu,\rho)$ such that $(F,R)$ is a compact metric space, $\mu$ is a locally finite Borel measure of full support on $F$, and $\rho$ is a distinguished point of $F$, which we call the root.

Suppose $(F,R,\mu,\rho)$ and $(F',R',\mu',\rho')$ are elements of $\mathbb{F}^c$. Given a metric space $(M, d_M)$, and isometric embeddings $\phi, \phi'$ of $(F,R)$ and $(F', R')$ respectively into $(M, d_M)$, we define $d^{GHP}_{M}\big((F,R,\mu,\rho, \phi), (F',R',\mu',\rho', \phi')\big)$ to be equal to
\begin{align*}\label{eqn:GHP def}
d_M^H(\phi (F), \phi' (F')) + &d_M^P(\mu \circ \phi^{-1}, \mu' \circ {\phi'}^{-1} ) + d_M(\phi (\rho), \phi' (\rho')).
\end{align*}
Here $d_M^H$ denotes the Hausdorff distance between two sets in $M$, and $d_M^P$ the Prohorov distance between two measures, as defined in \cite[Chapter 1]{BillsleyConv}. The pointed Gromov-Hausdorff-Prohorov distance between $(F,R,\mu,\rho)$ and $(F',R',\mu',\rho')$ is given by
\begin{align}
\begin{split}
\dGHPp{(F,R,\mu,\rho), (F',R',\mu',\rho')} = \inf_{\phi, \phi', M} d^{GHP}_{M}\big((F,R,\mu,\rho, \phi), (F',R',\mu',\rho', \phi')\big)\end{split}
\end{align}
where the infimum is taken over all isometric embeddings $\phi, \phi'$ of $(F,R)$ and $(F', R')$ respectively into a common metric space $(M, d_M)$. It is well-known (for example, see \cite[Theorem 2.3]{AbDelHoschNoteGromov}) that this defines a metric on the space of equivalence classes of $\mathbb{F}^c$, where we say that two spaces $(F,R,\mu,\rho)$ and $(F',R',\mu',\rho')$ are equivalent if there is a measure and root preserving isometry between them.
The pointed Gromov-Hausdorff distance $d_{GH}(\cdot, \cdot)$, which is defined by removing the Prohorov term from (\ref{eqn:GHP def}) above, can be helpfully defined in terms of \textit{correspondences}. A correspondence $\mathcal{R}$ between $(F,R,\mu,\rho)$ and $(F',R',\mu',\rho')$ is a subset of $F \times F'$ such that for every $x \in F$, there exists $y \in F'$ with $(x,y) \in \mathcal{R}$, and similarly for every $y \in F'$, there exists $x \in F$ with $(x,y) \in \mathcal{R}$. We define the \textit{distortion} of a correspondence by
\[
\textsf{dis} (\mathcal{R}) = \sup_{(x,x'), (y, y') \in \mathcal{R}} |R(x,y) - R(x',y')|.
\]
It is then straightforward to show that
\begin{align*}
\dGH{(F,R,\mu,\rho), (F',R',\mu',\rho')} = \frac{1}{2} \inf_{\mathcal{R}} \textsf{dis}(\mathcal{R}),
\end{align*}
where the infimum is taken over all correspondences $\mathcal{R}$ between $(F,R,\mu,\rho)$ and $(F',R',\mu',\rho')$ that contain the point $(\rho, \rho')$.

In this article, we will prove pointed Gromov-Hausdorff-Prohorov convergence by first proving  pointed Gromov-Hausdorff convergence using correspondences, and then show Prohorov convergence of the measures on the appropriate metric space.

\subsection{Stochastic Processes Associated with Resistance Metrics}\label{sctn:res forms}
To study Brownian motion and random walks on metric spaces we will be using the theory of resistance forms and resistance metrics, developed by Kigami in \cite{AOF} and  \cite{KigamiResistanceFormsMono}.

Let $G = (V,E)$ be a discrete graph equipped with non-negative symmetric edge conductances $c(x,y)_{(x,y) \in E}$ and a measure $(\mu(x))_{x \in V}$. Given these conductances, effective resistance on $G$ is a function $R$ on $V \times V$ defined by
\begin{equation}\label{eqn:resistance def variational}
R(x,y)^{-1} = \inf \{ \mathcal{E}(f,f)| f: V \rightarrow \R, f(x)=1, f(y)=0 \},
\end{equation}
where we take the convention $\inf \emptyset = \infty$, and $\mathcal{E}$ is an energy functional given by
\begin{equation*}\label{eqn:energy def}
\mathcal{E}(f,g) = \frac{1}{2} \sum_{x, y \in V} c(x,y) (f(y)- f(x))(g(y) - g(x)).
\end{equation*}
$R(x,y)$ corresponds to the usual physical notion of electrical resistance between $x$ and $y$ in $G$, when equipped with the given conductances. It can be shown (e.g. see \cite{Tetali}) that $R$ is a metric on $G$, and that $\mathcal{E}$ is a \textit{Dirichlet form} on $L^2(V, \mu)$.

The notion of a resistance metric can be extended to the continuum as follows.

\begin{defn}\label{def:eff resistance metric}\cite[Definition 2.3.2]{AOF}.
Let $F$ be a set. A function $R : F \times F$ is known as a \textit{resistance metric} on $F$ if and only if for every finite subset $V \subset F$, there exists a weighted graph with vertex set $V$ such that $R|_{V \times V}$ is the effective resistance on $(V,E)$, as defined by (\ref{eqn:resistance def variational}).
\end{defn}

A resistance metric on a set $F$ can be naturally associated with a stochastic process on $F$ via the theory of resistance forms. We do not give details of the theory here, but see \cite{KigamiResistanceFormsMono} for more on resistance forms. In particular, if $(F,R)$ is a compact metric space, then there is a one-to-one correspondence between resistance metrics and resistance forms on $F$ by \cite[Corollary 6.4]{KigamiResistanceFormsMono}. Moreover, if $(F,R)$ is additionally endowed with a finite Borel measure $\mu$ of full support, then by \cite[Theorem 9.4]{KigamiResistanceFormsMono}, the corresponding resistance form is in fact a regular Dirichlet form on $L^2(F, \mu)$, which in turn is naturally associated with a Hunt process on $F$ as a consequence of \cite[Theorem 7.2.1]{FOT}.


This correspondence allows us to use results about scaling limits of measured resistance metric spaces to prove results about scaling limits of stochastic processes as detailed in the following result of \cite{DavidResForms}.

\begin{theorem}\cite[Theorem 1.2, compact case]{DavidResForms}.\label{thm:scaling lim RW resistance}
Suppose that $(F_n, R_n, \mu_n, \rho_n)_{n \geq 0}$ is a sequence in $\mathbb{F}^c$ such that 
\[
(F_n, R_n, \mu_n, \rho_n) \rightarrow (F, R, \mu, \rho)
\]
with respect to the pointed Gromov-Hausdorff-Prohorov topology for some $(F, R, \mu, \rho) \in \mathbb{F}^c$, and $(R_n)_{n \geq 1}, R$ are resistance metrics on the respective spaces.

Let $(Y_t^{(n)})_{t \geq 0}$ and $(Y_t)_{t \geq 0}$ be the stochastic processes respectively associated with $(F_n, R_n, \mu_n, \rho_n)$ and $(F, R, \mu, \rho)$ as outlined above. Then it is possible to isometrically embed $(F_n, R_n)_{n \geq 1}$ and $(F,R)$ into a common metric space $(M, d_M)$ so that
\[
\prstart{(Y_t^{(n)})_{t \geq 0} \in \cdot}{\rho_n}{} \rightarrow \prstart{(Y_t)_{t \geq 0} \in \cdot}{\rho}{}
\]
weakly as probability measures as $n \rightarrow \infty$ on $D(\R_+, M)$ (i.e. on the space of \cadlag functions on $M$ equipped with the Skorohod $J_1$-topology).
\end{theorem}
For more on the Skorohod-$J_1$ topology, see \cite[Chapter 3]{BillsleyConv}. The intuition behind the result above is that the convergence of metrics and measures respectively give the appropriate spatial and temporal convergences of the stochastic processes.

It is also the case that we can analyse the associated stochastic processes by analysing the resistance volume growth of the space. This is part of the motivation for proving Theorems \ref{thm:main global}, \ref{thm:main local} and \ref{thm:main vol conv}, as the heat kernel estimates of Theorems \ref{thm:main HK global}, \ref{thm:main HK local} and \ref{thm:main annealed HK conv} then follow by an application of results from \cite{CroyHKFluctNonUn}.

\subsection{Stable \Levy Excursions}\label{sctn:Levy background}
Following the presentations of \cite{DuqContourLimit} and \cite{RSLTCurKort}, we now introduce stable \Levy excursions, which will be used to code stable trees and looptrees in Section \ref{sctn:tree looptree def}.

Throughout this article, we take $\alpha \in (1,2)$, and $X$ will be an $\alpha$-stable spectrally positive \Levy process (i.e. with only positive jumps) as in \cite[Section 8]{BertoinLevy}, normalised so that
\[
\E{e^{-\lambda X_t}} = e^{{\lambda}^{\alpha}t}
\]
for all $\lambda > 0$. $X$ takes values in the space $D([0, \infty), \R)$ of \cadlag functions, endowed with the Skorohod-$J_1$ topology, and satisfies the scaling property that for any constant $c>0$, $(c^{-\frac{1}{\alpha}} X_{ct})_{t \geq 0}$ has the same law as $(X_t)_{t \geq 0}$. Moreover, $X$ has infinite variation \cite[Lemma 2.12]{KyprianouIntroLectures}, and has \Levy measure
\[
\Pi(dx) = \frac{\alpha(\alpha - 1)}{\Gamma(2-\alpha)} x^{-\alpha - 1} \mathbbm{1}_{(0, \infty)}(x) dx.
\]

To define a normalised excursion of $X$, we follow \cite{Chaumont} and let $\underline{X}_t = \inf_{s \in [0,t]} X_s$ denote its running infimum process, and set
\begin{align*}
g_1 = \sup \{ s \leq 1: X_s = \underline{X}_s \}, \hspace{10mm} d_1 = \inf \{ s > 1: X_s = \underline{X}_s \}.
\end{align*}

Note that, since $X$ has no negative jumps and is of infinite variation, $X_{g_1} = X_{d_1}$ almost surely. Following \cite[Proposition 1]{Chaumont}, we define the normalised excursion $X^{\text{exc}}$ of $X$ above its infimum at time $1$ by
\[
X_s^{\text{exc}} = (d_1 - g_1)^{\frac{-1}{\alpha}} (X_{g_1 + s(d_1 - g_1)} - X_{g_1})
\]
for every $s \in [0,1]$. $X^{\text{exc}}$ is almost surely an $\alpha$-stable \cadlag function on $[0,1]$ with $X^{\text{exc}}(s)>0$ for all $s \in (0,1)$, and $X_0^{\text{exc}}=X_1^{\text{exc}}=0$.

\subsubsection{\Ito excursion measure}\label{sctn:Ito exc}
We can alternatively define $\X$ using the {\Ito excursion measure}. For full details, see \cite[Chapter IV]{BertoinLevy}, but the measure is defined by applying excursion theory to the process $X - \underline{X}$, which is strongly Markov and for which the point $0$ is regular for itself. We normalise local time so that $-\underline{X}$ denotes the local time of $X - \underline{X}$ at its infimum, and let $(g_j, d_j)_{j \in \mathcal{I}}$ denote the excursion intervals of $X - \underline{X}$ away from zero. For each $i \in \mathcal{I}$, the process $(e^i)_{0 \leq s \leq d_i-g_i}$ defined by $e^i(s) = X_{g_i + s} - X_{g_i}$ is an element of the excursion space
\[
E = \{ e \in D([0, \infty), [0, \infty)): e(0)=0, \zeta(e):=\sup\{s>0: e(s)>0\} \in (0, \infty), e(t) > 0 \text{ for all } t \in (0, \zeta (e)) \}.
\]
$\zeta (e)$ is the \textit{lifetime} of the excursion $e$. It was shown in \cite{ItoPP} that the measure
\[
N(dt, de) = \sum_{i \in \mathcal{I}} \delta (-\underline{X}_{g_i}, e^i)
\]
is a Poisson point measure of intensity $dt N(de)$, where $N$ is a $\sigma$-finite measure on the set $E$ known as the \textit{\Ito excursion measure}.

Moreover, the measure $N(\cdot)$ inherits a scaling property from the $\alpha$-stability of $X$. Indeed, for any $\lambda > 0$ we define a mapping  $\Phi_{\lambda}: E \rightarrow E$ by  $\Phi_{\lambda}(e)(t) = \lambda^{\frac{1}{\alpha}} e(\frac{t}{\lambda})$, so that $N \circ \Phi_{\lambda}^{-1} = \lambda^{\frac{1}{\alpha}} N$ (e.g. see \cite{WataIto}). It then follows from the results in \cite[Section IV.4]{BertoinLevy} that we can uniquely define a set of conditional measures $(N_{(s)}, s>0)$ on $E$ such that:
\begin{enumerate}[(i)]
\item For every $s > 0$, $N_{(s)}( \zeta=s)=1$.
\item For every $\lambda > 0$ and every $s>0$, $\Phi_{\lambda}(N_{(s)}) = N_{(\lambda s)}$.
\item For every measurable $A \subset E$
\[
N(A) = \int_0^{\infty} \frac{N_{(s)}(A)}{\alpha \Gamma(1 - \frac{1}{\alpha}) s^{\frac{1}{\alpha}+1}} ds.
\]
\end{enumerate}

$N_{(s)}$ is therefore used to denote the law $N( \cdot | \zeta = s)$. The probability distribution $N_{(1)}$ coincides with the law of $X^{\exc}$ as constructed above.

\subsubsection{Relation between $X$ and $X^{\exc}$}
Throughout this paper we will use the following two tools to compare the probability of an event defined in terms of $\X$ to that of the same event defined in terms of $X$.

\begin{theorem}\cite[Th\'eor\`eme 4]{Chaumont}.\label{thm:Vervaat} Vervaat Transform.
\begin{enumerate}
\item Let $\X$ be as above, and take $U \sim$ \textsf{Uniform}$([0,1])$. Then the process $(\Xb_t)_{0 \leq t \leq 1}$ defined by
\[
\Xb_t = \begin{cases} \X_{U+t} & \text{ if } U+t \leq 1,\\
\X_{U+t-1} & \text{ if } U+t > 1.
\end{cases}
\]
has the law of a spectrally positive stable \Levy bridge on $[0,1]$.
\item Now let $\Xb$ be a spectrally positive stable \Levy bridge on $[0,1]$, and let $m$ be the (almost surely unique) time at which it attains its minimum. Define an excursion $\X$ by
\[
\X_t = \begin{cases} \Xb_{m+t} & \text{ if } m+t \leq 1,\\
\Xb_{m+t-1} & \text{ if } m+t > 1.
\end{cases}
\]
\end{enumerate}
Then $\X$ has the law of a spectrally positive stable \Levy excursion.
\end{theorem}

An event defined for the stable bridge on the interval $[0,T]$ can then be transferred to the unconditioned process using the fact that the law of the bridge is absolutely continuous with respect to the law of the process, with Radon-Nikodym derivative
\begin{equation}\label{eqn:abs cont RN deriv Levy bridge}
\frac{p_{1-{T}}(-X_{T})}{p_{1}(0)}
\end{equation}
for $T \in (0,1)$ (see \cite[Section VIII.3, Equation (8)]{BertoinLevy}).

Here $p_t(x)$ is the transition density of the \Levy process $X$ defined above with respect to Lebesgue measure on $\R$, i.e. a measurable function on $\R$ such that
\[
\estartb{f(X_t)}{0} = \int_{-\infty}^{\infty} f(y) p_t(y) dy
\]
for all bounded, measurable functions $f$ on $\R$.

\subsection{Two Parameter Poisson-Dirichlet Distribution}
We now introduce the two parameter Poisson-Dirichlet distribution, denoted \PD($\beta, \theta$), which arises naturally in the context of decompositions of random trees (amongst other things). It is a law on countable partitions of the interval $[0,1]$. We will denote such a partition by $(M_1, M_2, \ldots)$. Here we outline the GEM (Griffiths, Engen, McCloskey) construction, which gives a size-biased ordering of the \PD($\beta, \theta$) distribution via a residual allocation model. For further background see \cite{PitYor2PD}.

\begin{prop}\cite[Proposition 2]{PitYor2PD}.
For $0 \leq \beta < 1$, and $\theta > -\beta$, let $(Z_n)_{n \geq 1}$ be a sequence of independent random variables with 
\[
Z_n \sim \textsf{Beta} (1 - \beta, n \beta + \theta)
\]
for each $n \geq 1$. Define a sequence of random variables $(M_n)_{n \geq 1}$ by
\begin{align*}
M_1 = Z_1, \ \ \ M_2 = (1-Z_1)Z_2, \ \ \ \ldots, \ \ \ M_n = (1-Z_1)(1-Z_2) \ldots (1-Z_{n-1}) Z_n
\end{align*}
for all $n \geq 1$.  Then $\sum_n M_n = 1$ almost surely, and the random vector $(M_1, M_2, \ldots)$ is distributed as a size-biased ordering of \PD($\beta, \theta$).
\end{prop}

In Section \ref{sctn:vol bounds} we will use the following two results.

\begin{lem}\label{lem:Gem correlations}
Let $(M_1, M_2, \ldots)$ be as above, and let $(g(n))_{n \geq 1}$ be any sequence of numbers taking values in $[0,1]^{\N}$. Then 
\begin{align*}
\prcond{ M_n \geq g(n)}{M_l < g(l) \forall l < n}{} \geq \pr{M_n \geq g(n)}.
\end{align*}
\begin{proof}
This is immediate on noting that $M_n = (1 - \sum_{i=1}^{n-1} M_i) Z_n$.
\end{proof}
\end{lem}

\begin{lem}\label{lem:Gem paley zig}
Let $(M_1, M_2, \ldots)$ be \textsf{PD}($\alpha^{-1}, 1 - \alpha^{-1}$), where $\alpha \in (1,2)$. Then there exists a constant $c > 0$ such that for any $c' \in (0,1)$ we have:
\begin{align*}
\pr{M_{k} \geq c' k^{-\alpha}} \geq c(1-c')^2.
\end{align*}
\begin{proof}
The proof is an application of the Paley-Zigmund inequality, which states that for any non-negative random variable $X$ with finite variance, and any $\theta \in [0,1]$,
\[
\pr{X \geq \theta \E{X}} \geq (1-\theta)^2 \frac{\E{X}^2}{\E{X^2}}.
\]
By taking $X = M_k$, and using the independence of the $(Z_n)_{n \geq 1}$ we have that there exists $c, k_0 < \infty$ such that, whenever $k \geq k_0$,
\begin{align*}
\E{M_{k}} &=\E{Z_k} \Big( \prod_{i=1}^{k-1} \E{1-Z_i} \Big) \geq \frac{1 - \alpha^{-1}}{2+(k-2)\alpha^{-1}} \prod_{i=1}^{k-1} \frac{1 + (i-1) \alpha^{-1}}{2+(i-2)\alpha^{-1}} \geq \big(\frac{3}{2}\big)^{\alpha} k^{-\alpha}, \\
\E{M_{k}^2} &=\E{Z_k^2} \Big( \prod_{i=1}^{k-1} \E{(1-Z_i)^2} \Big) = \frac{\alpha - 1}{(3\alpha+k-2)(2 \alpha + k-2)} \prod_{i=1}^{k-1} \frac{2 \alpha + i - 1}{3 \alpha + i-2} \frac{\alpha + i -1}{2\alpha + i - 2} \leq ck^{-2\alpha}.
\end{align*}
The result follows.
\end{proof}
\end{lem}

\section{Background on Stable Trees and Looptrees}\label{sctn:tree looptree def}
\subsection{Discrete Trees}\label{sctn:trees background discrete}
Before defining stable trees and looptrees, we briefly recap the Ulam-Harris labelling notation for discrete trees, following the formalism of \cite{Neveu}. Firstly, introduce the set
\[
\mathcal{U}=\cup_{n=0}^{\infty} {\N}^n.
\]
By convention, ${\N}^0=\{ \emptyset \}$. If $u=(u_1, \ldots, u_n)$ and $v=(v_1, \ldots, v_m) \in \mathcal{U}$, we let $uv= (u_1, \ldots, u_n, v_1, \ldots, v_m)$ be the concatenation of $u$ and $v$.

\begin{defn} A plane tree $T$ is a finite subset of $\mathcal{U}$ such that
\begin{enumerate}[(i)]
\item $\emptyset \in T$,
\item If $v \in T$ and $v=uj$ for some $j \in \N$, then $u \in T$,
\item For every $u \in T$, there exists a number $k_u(T) \geq 0$ such that $uj \in T$ if and only if $1 \leq j \leq k_u(T)$.
\end{enumerate}
\end{defn}

We let $\mathbb{T}$ denote the set of all plane trees. A plane tree $T \in \mathbb{T}$ with $n+1$ vertices labelled according to the lexicographical order as $u_0, u_1, \ldots, u_n$ can be coded by its \textit{height function}, \textit{contour function}, or \textit{Lukasiewicz path}, defined as follows.
\begin{itemize}
\item The height function $(H^{T}_m)_{0 \leq m \leq n}$ is defined by considering the vertices $u_0, u_1, \ldots, u_n$ in lexicographical order, and then setting $H^{T}_i$ to be the generation of vertex $u_i$.
\item The contour function $(C^{T}_t)_{0 \leq t \leq 2n}$ is defined by considering a particle that starts at the root $\emptyset$ at time zero, and then continuously traverses the boundary of ${T}$ at speed one, respecting the lexicographical order where possible, until returning to the root. $C^{T}(t)$ is equal to the height of the particle at time $t$.
\item The Lukasiewicz path $(W^{T}_m)_{0 \leq m \leq n}$ is defined by setting $W^{T}_0 = 0$, then by considering the vertices $u_0, u_1, \ldots, u_n$ in lexicographical order and setting $W^T_{m+1} = W^T_m + k_{u_m}(T)-1$.
\end{itemize}

These are illustrated in Figure \ref{fig:contourheightfns}, together with points corresponding to specific vertices in the tree, and the part of each excursion coding the subtree rooted at the red vertex, which we denote by $\tau_1(T)$. For further details, see \cite[Section 0.1]{LeGDuqMono}.

\begin{figure}[h]
\includegraphics[width=15cm, height=5.6cm]{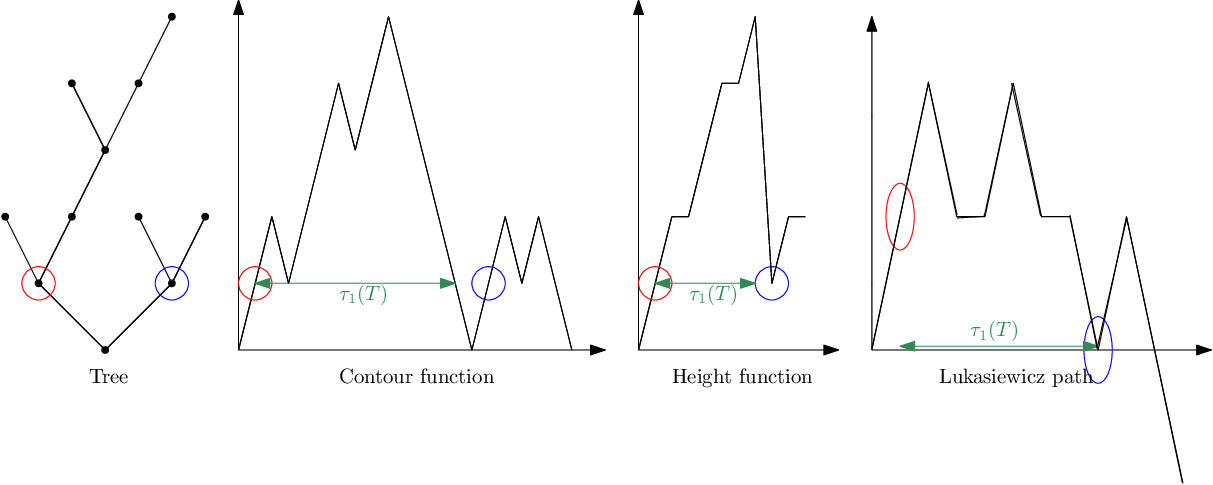}
\centering
\caption{Example of contour function, height function and Lukasiewicz path for the given tree.}\label{fig:contourheightfns}
\end{figure}

These functions all uniquely define the tree $T$. This can be written particularly conveniently in the case of the contour function, since for any $s, t \in \{0, \ldots, 2(n-1)\}$, we can write the tree distance as a function on $\{0, \ldots, 2(n-1)\} \times \{0, \ldots, 2(n-1)\}$ by setting
\[
d^T(s,t) = C^T(s) + C^T(t) - 2\inf_{s \leq r \leq t} C^T(r).
\]

We will work mainly with the Lukasiewicz path $(W^{T}_m)_{0 \leq m \leq n}$ in this paper. It is not too hard to see that $W^{T}_m \geq 0$ for all $0 \leq m \leq n-1$, and $W^T_n = -1$. Moreover, the height function can be defined as a function of the Lukasiewicz path (see \cite[Equation (1)]{LeGDuqMono}) by setting
\begin{equation}\label{eqn:height Luk def}
H^T(m) = \Big|\Big\{ k \in \{0, 1, \ldots, m-1\}: W^T_k = \inf_{k \leq l \leq m} W^T_l \Big\}\Big|.
\end{equation}

\subsection{Stable trees}
We now introduce stable trees. These are closely related to stable looptrees, and were introduced by Le Gall and Le Jan in \cite{LeGLeJanExploration} then further developed by Duquesne and Le Gall in \cite{LeGDuqMono,DuqLeGPFALT}. For $\alpha \in (1,2)$ we define the stable tree $\Ta$ from a spectrally positive $\alpha$-stable \Levy excursion, which plays the role of the Lukasiewicz path introduced above. By analogy with (\ref{eqn:height Luk def}), given such an excursion $\X$, we define the height function $\HH$ to be the continuous modification of the process satisfying
\begin{equation*}\label{eqn:height def}
\HH(t) = \lim_{\epsilon \rightarrow 0} \frac{1}{\epsilon} \int_0^t \mathbbm{1} \{X^{\exc}_s < I_s^t + \epsilon \} ds,
\end{equation*}
where $I_s^t = \inf_{r \in [s,t]} \X_r$ for $s \leq t$, and the limit exists in probability (e.g. see \cite[Lemma 1.1.3]{LeGDuqMono}). We define a distance function on $[0,1]$ by
\[
d(s,t) = \HH(s) + \HH(t) - 2 \inf_{s \leq r \leq t} \HH(r),
\]
and an equivalence relation on $[0,1]$ by setting $s \sim t$ if and only if $d(s,t) = 0$. $\T_{\alpha}$ is the quotient space $([0,1]/ \sim, d)$, and we let $\pi$ denote the canonical projection from $[0,1]$ to $\Ta$. If $u, v \in \Ta$, we let $[[u,v]]$ denote the unique geodesic between $u$ and $v$ in $\Ta$.

This construction also provides a natural way to define a measure $\mu$ on $\T_{\alpha}$ as the image of Lebesgue measure on $[0,1]$ under the quotient operation.

Stable trees arise naturally as scaling limits of discrete plane trees with appropriate offspring distributions. More specifically, let $T_n$ be a discrete tree conditioned to have $n$ vertices and with critical offspring distribution $\xi$ in the domain of attraction of an $\alpha$-stable law, and such that $\xi$ is aperiodic. It is shown in \cite[Theorem 3.1]{DuqContourLimit} that
\begin{equation}\label{eqn:stable tree scaling limit def}
a_n n^{-1} T_n \rightarrow \T_{\alpha}
\end{equation}
in the Gromov-Hausdorff topology as $n \rightarrow \infty$, where $a_n$ is as defined in (\ref{eqn:dom of att def}).

\subsection{Random Looptrees}\label{sctn:looptree def}
Discrete looptrees are best described by Figure \ref{fig:disc looptree intro} in the introduction. Moreover, as outlined there, stable looptrees can be defined as scaling limits of discrete looptrees: $T_n$ is a Galton Watson tree conditioned to have $n$ vertices with critical offspring distribution $\xi$ in the domain of attraction of an $\alpha$-stable law, then
\[
a_n^{-1} {\Loop}(T_n) \overset{(d)}{\rightarrow} \L_{\alpha}
\]
with respect to the Gromov-Hausdorff topology as $n \rightarrow \infty$ \cite[Theorem 4.1]{RSLTCurKort}. By comparison with (\ref{eqn:stable tree scaling limit def}), $\L_{\alpha}$ can therefore be thought of as the looptree version of the \Levy tree $\T_{\alpha}$. We now explain how this intuition can be used to code $\La$ from a stable \Levy excursion, in such a way that $\La$ can be heuristically obtained from the corresponding stable tree $\T_{\alpha}$ by replacing each branch point by a loop with length proportional to the size of the branch point, and gluing these loops together along the tree structure of $\T_{\alpha}$.

We first define some notation. As is standard, for every $s, t \in [0,1]$ we write $s \preceq t$ if and only if $s \leq t$ and $\X_{s^-} \leq \inf_{[s,t]}\X$, and in this case we set
\begin{align*}
x_s^t(\X) = \inf_{[s,t]}\X - \X_{s^-}, \text{     and      } u_s^t(\X) = \frac{x_s^t(\X)}{\Delta \X_s}.
\end{align*}
In what follows we will drop the dependence on $\X$ and merely write $x_s^t, u_s^t$ etc.

The following construction was introduced in \cite[Section 2.3]{RSLTCurKort}. The \Levy excursion itself plays the role of a continuum Lukasiewicz path. It was shown in \cite[Proposition 2]{MiermontSplittingNodes} that if we define the width of a branch point in $\T_{\alpha}$, coded by a jump at $t \in [0,1]$, by 
\[
\lim_{\epsilon \downarrow 0} \frac{1}{\epsilon} \mu (\{v \in \Ta, d(\pi(t), v) \leq \epsilon\}),
\]
then the limit almost surely exists and is equal to $\Delta_t$. It is therefore natural that a jump of size $\Delta$ in $\X$ should code a ``loop" of length $\Delta$ in $\La$.

Accordingly, for every $t \in [0,1]$ with $\Delta_t > 0$, the authors in \cite[Section 2.3]{RSLTCurKort} equip the segment $[0, \Delta_t]$ with the pseudodistance
\begin{equation}\label{eqn:delta def}
\delta_t(a,b) = \min \{|a-b|,(\Delta_t - |a-b|) \}, \ \ \ \ \ \ \ \ \ \text{for} \ a, b \in [0, \Delta_t],
\end{equation}
and define a distance function on $[0,1]$ by first setting 
\begin{equation*}
d_0(s,t) = \sum_{s \prec u \preceq t} \delta_u(0, x_u^t)
\end{equation*}
whenever $s \preceq t$, and
\begin{equation} \label{eqn:d}
d(s,t) = \delta_{s \wedge t}(x_{s \wedge t}^s,x_{s \wedge t}^t) + d_0(s \wedge t, s) + d_0(s \wedge t, t)
\end{equation}
for arbitrary $s, t \in [0,1]$. They show that $d$ as defined above is almost surely a continuous pseudodistance on $[0,1]$, and then define an equivalence relation $\sim$ on $[0, 1]$ by setting $s \sim t$ if $d(s,t)=0$, and in \cite[Definition 2.3]{RSLTCurKort} define the stable looptree $\La$ as the quotient space
\[
\mathcal{L}_{\alpha} = ([0,1]/ \sim, d)
\]
We let $p:[0,1] \rightarrow \La$ denote the canonical projection under the quotient operation, and let $\nu$ denote the image of Lebesgue measure on $[0,1]$ under $p$, so that $\nu$ is the natural analogue of uniform measure on $\La$.

At various points in this paper, we will refer to the ``corresponding" or ``underlying" stable tree of $\La$, by which we mean the stable tree $\Ta$ coded by the same excursion that codes $\La$. We let $\La$ denote a compact stable looptree conditioned on $\nu(\La)=1$, but at various points we will let $\tilde{\La}$ denote a generic stable looptree coded by an excursion under the \Ito measure but without any conditioning on its total mass. We will also let $\La^{1}$ denote a stable looptree but conditioned so that its underlying tree has height $1$. We will however make any conditioning explicit at the time of writing.

\subsubsection{Re-rooting Invariance for Stable Trees and Looptrees}

In \cite{DuqLeGPFALT}, Duquesne and Le Gall prove that stable \Levy trees are invariant under uniform rerooting. More formally, if $U$ is a uniform point in $[0,1]$, and we define a new height function $H^{[U]}:[0,1] \rightarrow \R$ from the original height function $\HH$ by
\[
H^{[U]}(x) = \begin{cases} \HH(U) + \HH(U+x) - 2 \min_{U \leq s \leq U+x} \HH(s) & \text{ if } U+x \leq 1 \\
\HH(U) + \HH(U+x-1) - 2 \min_{U+x-1 \leq s \leq U} \HH(s) & \text{ if } U+x > 1 ,
\end{cases}
\]
then $H^{[U]} \overset{(d)}{=} \HH$. This property is just saying that if we pick a uniform point $U \in [0,1]$, and reroot the tree $\T_{\alpha}$ at $\pi(U)$, then the resulting tree has the same distribution as the original one.

We will prove most of our looptree volume results by considering the volume of a ball at a uniform point in $\La$, and then extending to almost all of $\La$ by Fubini's theorem. To prove the upper bounds, we will apply some spinal decomposition results for stable trees that we outline in the next section. The uniform rerooting invariance result means that we can equivalently consider our uniform point to correspond to the root of the stable tree.

Note that the problem of uniform rerooting invariance of continuum fragmentation trees was also considered in the paper \cite{HPWSpinPart}, where the authors additionally show that stable trees are the only fragmentation trees for which this property holds. Duquesne and Le Gall also prove a similar result for rerooting at a deterministic point $u \in [0,1]$ in the paper \cite{DuqLeGRerooting}, and \cite[Remark 4.6]{RSLTCurKort} directly addresses the question of uniform rerooting invariance for stable looptrees.

\section{Resistance and Random Walk Scaling Limit Results}\label{sctn:limit results}

In this section we construct a resistance metric on $\La$, define Brownian motion, and prove Theorem \ref{thm:main RW inv princ compact quenched}.

\subsection{Construction of a Resistance Metric on Stable Looptrees}\label{sctn:res RWs looptrees}
The metric is similar in spirit to the metric of \cite{RSLTCurKort} that we introduced in Section \ref{sctn:looptree def}, but we will sum the effective resistance across loops rather than the shortest-path distance. It turns out that these resistance looptrees are homeomorphic to the original ones, which means that the shortest distance metric can equivalently be used to prove the volume bounds of Theorems \ref{thm:main global} and \ref{thm:main local}, making the problem more tractable. Additionally this means that part of the invariance principle of Theorem \ref{thm:compact disc inv princ res} arises as a direct consequence of \cite[Theorem 4.1]{RSLTCurKort}.

In the continuum, again let $\X$ be as in Section \ref{sctn:Levy background}. This time, if $\X$ has a jump of size $\Delta_t>0$ at point $t$, equip the segment $[0, \Delta_t]$ with the pseudodistance
\begin{equation}\label{eqn:r def}
r_t(a,b) = {\Big( \frac{1}{|a-b|}+\frac{1}{\Delta_t - |a-b|} \Big) }^{-1}= \frac{|a-b|(\Delta_t - |a-b|)}{\Delta_t}, \ \ \ \ \ \ \ \ \ \text{for} \ a, b \in [0, \Delta_t].
\end{equation}
The quantity $r_t$ gives the resistance across the loop associated to the branch point at $t$. Note that $r_t(a,b)$ corresponds to the effective resistance of two parallel edges of resistance $|a-b|$ and $\Delta_t - |a-b|$, and by Rayleigh's Monotonicity Principle it follows that $r_t(a,b) \leq \delta_t(a,b)$ (also see Lemma \ref{lem:dR compare}).

By analogy with expression (\ref{eqn:delta def}) in Section \ref{sctn:looptree def}, for $s, t \in [0,1]$ with $s \preceq t$ we set
\begin{equation}\label{eqn:R0}
R_0(s,t) = \sum_{s \prec u \preceq t} r_u(0, x_u^t).
\end{equation}
For arbitrary $s, t \in [0,1]$, we set 
\begin{equation} \label{eqn:R}
R(s,t) = r_{s \wedge t}(x_{s \wedge t}^s,x_{s \wedge t}^t) + R_0(s \wedge t, s) + R_0(s \wedge t, t).
\end{equation}

We give a comparison between $d$ and $R$ in the following lemma.

\begin{lem}\label{lem:dR compare}
For any $s,t \in [0,1]$, we have
\[
\frac{1}{2}d(s,t) \leq R(s,t) \leq d(s,t).
\]
\begin{proof}
Note that, trivially, for any $x, y \in [0,1]$:
\[
{\Bigg(\frac{2}{\text{min} \{x, y \}} \Bigg)}^{-1} \leq {\Bigg( \frac{1}{x} + \frac{1}{y} \Bigg) }^{-1} \leq {\Bigg(\frac{1}{\text{min} \{x, y \}} \Bigg)}^{-1}.
\]
Taking $x = |a-b|, y = \Delta_t - |a-b|$ we obtain $\frac{1}{2} \delta_t(a,b) \leq r_t(a,b) \leq \delta_t(a,b)$ for all $t \in [0,1]$, $a, b \in [0, \Delta_t]$.
\end{proof}
\end{lem}

It therefore follows from the corresponding result for $d$ given in \cite[Proposition 2.2]{RSLTCurKort} that almost surely, the function $R(\cdot, \cdot): [0,1]^2 \rightarrow {\R}_+$ is a continuous pseudodistance, and so we can make the following definition. Note in particular that $d(s,t) = 0$ if and only if $R(s,t) = 0$.

\begin{defn}
Let $X$ be an $\alpha$-stable \Levy excursion. The corresponding $\alpha$-stable resistance looptree is defined to be the quotient metric space
\[
\mathcal{L}_{\alpha}^R = ([0,1]/ \sim, R).
\]
\end{defn}

At several points we will write $\La^R$ as $(\La, R)$, to emphasise how it fits into the framework of \cite{DavidResForms} and various other articles. The next corollary follows directly from Lemma \ref{lem:dR compare}.

\begin{cor}\label{cor:homeo}
The looptrees $\mathcal{L}_\alpha$ and $\mathcal{L}^R_{\alpha}$ are homeomorphic.
\end{cor}

\begin{prop}\label{prop:res metric}
$R$ is a resistance metric in the sense of Definition \ref{def:eff resistance metric}.
\end{prop}

\begin{figure}[h]
\begin{subfigure}{.5\textwidth}
\includegraphics[height=6.5cm]{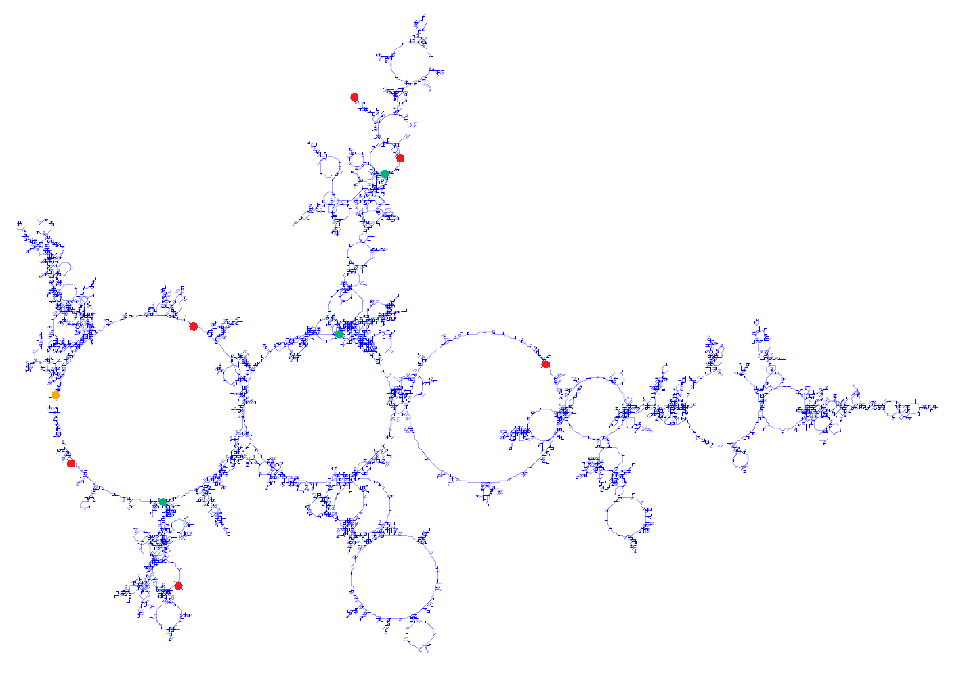}
\centering
\subcaption{Selection of points in $\La$}
\end{subfigure}
\begin{subfigure}{.5\textwidth}
\includegraphics[height=6.5cm]{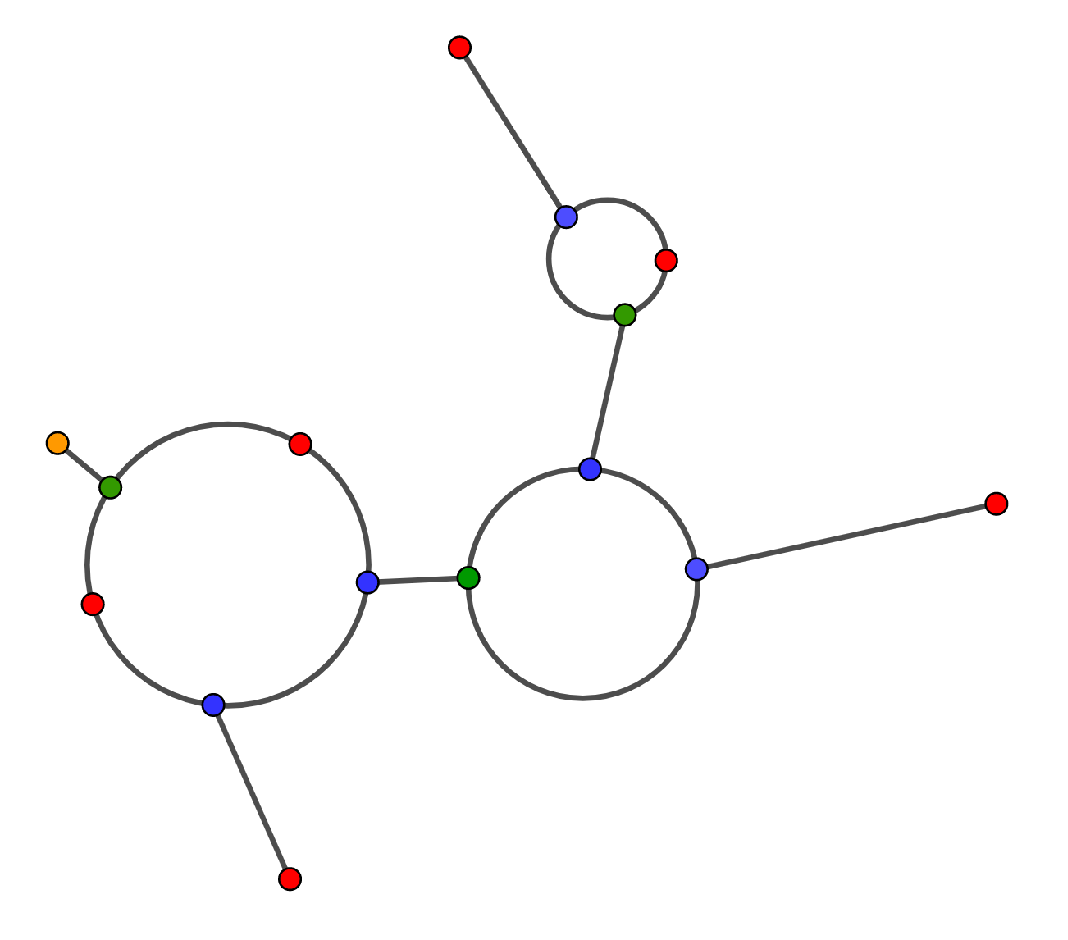}
\centering
\subcaption{$G'$}
\end{subfigure}
\caption{Resistance metric illustration.}\label{fig:resistance metric}
\end{figure}

To prove Proposition \ref{prop:res metric} we will take a finite set $V \subset \La$, define a larger graph $G' = (V', E')$ with $V \subset V'$, and show that $R=R|_{V'}$. An illustration is provided in Figure \ref{fig:resistance metric}. For notational convenience, we will represent $V$ by points in $[0,1]$ that project onto $\La$, so suppose $V = \{ v_1, v_2, \ldots, v_n \}$, where each $v_i \in [0,1]$, and $p(v_i) \neq p(v_j)$ for $i \neq j$. We will also assume that these are the minimal representatives in $[0,1]$: that is, if $v_i \in V$, then $\nexists v \in [0,1]$ with $v < v_i$ and $p(v_i) = p(v)$. We will let $p'$ denote the projection from $[0,1]$ into $V'$. It then follows from results of \cite[Section 2]{AOF} that we can reduce $V'$ to an appropriate network on $V$.

Informally, we do this in the natural way: by drawing loops corresponding to points in $V$, and joining these appropriately. In Figure \ref{fig:resistance metric}, we take $V$ to be the set of red points, and form $V^{\wedge}$ by adding the set of green points that correspond to the most recent common ancestors of pairs of points in $V$. The gold point corresponds to the root of the whole looptree, and the blue points correspond to extra vertices that we denote by expressions of the form $\rho_{u,v}$ below. Formally, we can construct our discrete picture as as follows:
\begin{enumerate}
\item \textit{Extend $V$ to include all the recent common ancestors of points in $V$.} Set
\[
V^{\wedge} = V \cup \{v_i \wedge v_j: v_i, v_j \in V \}.
\]
\item \textit{Draw loops corresponding to points in $V^{\wedge}$.} Define an equivalence relation $\sim_L$ on $V^{\wedge}$ by setting $v_i \sim_L v_j$ if and only if they have exactly the same set of strict ancestors.

We call each equivalence class a ``loop". Denote the loop corresponding to equivalence class $[v]$ by $L_v$. If the loop contains only one point in $V^{\wedge}$, and if this point is also in $V$, but $\Delta_v=0$, then it must be the case that there are no ancestors of $v$ in $V'$, so we leave it as this one point, and set $\rho_v=v$. If $\Delta_v>0$, then we instead draw a loop of length $\Delta_v$, and mark a point on it as $p'(\rho_v)$, which we will consider to be the ``base" of this loop. If the loop instead contains only one point in $V^{\wedge}$ but this point is not in $V$ then this point must correspond to a recent common ancestor of two points in $V$, and hence is a jump point of the \Levy excursion, say of size $\Delta_v$. We draw a loop of length $\Delta_v$ and mark $p'(v)$ as a point on this loop. If $L_v$ contains more than one point, then by definition of $V^{\wedge}$ it must also contain a point $s$ at which $\X$ has a jump. We denote this point by $\rho_v$ (note that $\rho_v$ does not depend on which point $v$ we choose to represent the equivalence class). We draw the loop corresponding to $[v]$ by taking all the elements $v_1, v_2, \ldots v_m \in [v]$ in order so that $v_i < v_{i+1}$ for all $i$. Note it will then be the case that $\rho_v = v_1$. Recall also that we are representing points in $V^{\wedge}$ by their minimal representative in $[0,1]$. We then draw the loop corresponding to $[v]$ by adding an edge joining $p'(v_i)$ to $p'(v_{i+1})$ of length $x_{v_1}^{v_{i}} - x_{v_1}^{v_{i+1}}$ for each $1 \leq i \leq m-1$, and an edge joining $p'(v_m)$ to $p'(v_1)$ of length $x_{v_1}^{v_{m}}$.
\item \textit{Join the loops along the tree structure.} We define a partial order on the set of loops by setting $L_u \prec L_v$ if and only if $\rho_u \prec \rho_v$ using the ancestral definition of $\prec$, where we set $\rho_v = v$ if $L_v$ contains only one point. For any two loops $L_u \prec L_v$ such that there is no $[w]$ with $L_u \prec L_w \prec L_v$, join them as follows:
\begin{enumerate}[(i)]
\item If $L_u$ and $L_v$ are both single points, then join these two points with a single edge of length $R(u,v)$.
\item If $L_u$ is a single point but $L_v$ is a loop with at least two points, then join $p'(u)$ to the point $p'(\rho_v)$ with a single edge of length $R(u, \rho_v)$.
\item If $L_u$ contains more than point, but $L_v$ contains a single point, then letting $\textsf{Anc}(w)$ denote the set of ancestors of a point $w \in [0,1]$, set $\rho_{u,v} = \inf \{ \textsf{Anc}(v) \setminus \textsf{Anc}(u) \}$, and add a point $p'(\rho_{u,v})$ to the loop at a distance $x_u^v$ from $p'(\rho_u)$, where this distance is measured in the direction that respects the lexicographical ordering of points in $L_u$. Then add an edge joining $p'(\rho_{u,v})$ to $p'(v)$ of length $R(\rho_{u,v}, v)$.
\item If both $L_u$ and $L_v$ contain more than one point, then define $\rho_{u,v}$ as above, and join $p'(\rho_{u,v})$ to $p'(\rho_v)$ by an edge of length $R(\rho_{u,v}, \rho_v)$.
\end{enumerate}
\end{enumerate}

It follows by construction that $G'$ is a connected graph and $V \subset V'$. Let $r'$ denote the effective resistance metric on this graph, which can be calculated using the series and parallel laws. We now prove the following.

\begin{lem}\label{lem:res equal r'}
\[
R|_{V'} = r'.
\]
\begin{proof}
The network $G'$ is in the form of several loops which are joined together by extra edges in such a way as to preserve the tree structure of $\L_{\alpha}$. For notational simplicity, we now relabel vertices so that $V^{\wedge} = \{v_1, v_2, \ldots, v_n \}$, $L_i = L_{v_i}$, $\rho_i = \rho_{v_i}$, and $\rho_{i,j} = \rho_{v_i, v_j}$.

Given $i,j$ such that $L_i \npreceq L_j$ and $L_j \npreceq L_i$, let $\rho_i^j = v_i \wedge v_j (= \rho_j^i)$, and let $S_{i,j} = \{ v \in V': v \preceq v_i, \rho_i^j \prec v \}$, and similarly $S_{j,i} = \{v \in V': v \preceq v_j, \rho_j^i \prec v \}$. Note that $\rho_k = k$ for all $k \in S_{i,j} \cup S_{j,i}$. Additionally, write $S_{i,j} = \{i_1, \ldots i_m \}$ and $S_{j,i} = \{j_1, \ldots, j_n \}$ in lexicographical order. We then have by construction (specifically Step 3 above) and the series law that:
\begin{align}\label{eqn:r' breakdown 1}
\begin{split}
r'(v_i, v_j) &= r'(\rho_{i_1, i_2}, \rho_{i_2}) + \Big( \sum_{l=2}^{m-1} \big( r'(\rho_{i_l}, \rho_{i_{l}, i_{l+1}}) + r'(\rho_{i_{l}, i_{l+1}}, \rho_{i_{l+1}}) \big) \Big) + r'(\rho_{i_m}, v_i) \mathbbm{1}\{v_i \neq i_m \} \\
&\ \ \ \ \ + r'(\rho_{j_1, j_2}, \rho_{j_2}) + \Big( \sum_{l=2}^{n-1} \big( r'(\rho_{j_l}, \rho_{j_{l}, j_{l+1}}) + r'(\rho_{j_{l}, j_{l+1}}, \rho_{j_{l+1}}) \big) \Big) + r'(\rho_{j_n}, v_j) \mathbbm{1}\{v_j \neq i_n \} \\
&\ \ \ \ \ + r'(\rho_{i_{1}, i_2}, \rho_{j_1, j_2}).
\end{split}
\end{align}
Then note that by definition (specifically, from Step 3 above), that
\begin{align*}
r'(\rho_{i_l}, \rho_{i_{l}, i_{l+1}}) = R(\rho_{i_l}, \rho_{i_{l}, i_{l+1}}) = r_u(0, x_{\rho_{i_l}}^{\rho_{i_l, i_{l+1}}}) = \sum_{\rho_{i_l} \preceq u \prec \rho_{i_l, i_{l+1}}} r_u(0, x_u^{v_i}),
\end{align*}
and similarly, by the series law for resistance, we have that
\[
r'(\rho_{i_{l}, i_{l+1}}, \rho_{i_{l+1}}) = R(\rho_{i_{l}, i_{l+1}}, \rho_{i_{l+1}}) = \sum_{\rho_{i_l, i_{l+1}} \preceq u \preceq \rho_{i_{l+1}}} r_u(0, x_u^{\rho_{i_{l+1}}}) = \sum_{\rho_{i_l, i_{l+1}} \preceq u \prec \rho_{i_{l+1}}} r_u(0, x_u^{v_i}),
\]
where the final line follows since $u \prec \rho_{i_{l+1}} \preceq v_i$ for all $u$ satisfying $\rho_{i_l, i_{l+1}} \preceq u \prec \rho_{i_{l+1}}$.

Similarly, $r'(\rho_{i_{1}, i_2}, \rho_{j_1, j_2}) = r_{v_i \wedge v_j}(x_{v_i \wedge v_j}^{v_i}, x_{v_i \wedge v_j}^{v_j})$.

It therefore follows from (\ref{eqn:R}), (\ref{eqn:r' breakdown 1}) and the series law that
\begin{align*}
r'(v_i,v_j) &= \sum_{l=2}^{m-1} \sum_{\rho_{i_l} \preceq u \prec \rho_{i_l, i_{l+1}}} r_u(0, x_u^{v_i}) + \sum_{l=1}^{m-1}\sum_{\rho_{i_l, i_{l+1}} \preceq u \prec \rho_{i_{l+1}}} r_u(0, x_u^{v_i}) + r_{\rho_{i_m}}(0, x_{\rho_{i_m}}^{v_i}) \mathbbm{1}\{v_i \neq i_m \} \\
&\ \ \ + \sum_{l=2}^{n-1} \sum_{\rho_{j_l} \preceq u \prec \rho_{j_l, j_{l+1}}} r_u(0, x_u^{v_j}) + \sum_{l=1}^{n-1}\sum_{\rho_{j_l, j_{l+1}} \preceq u \prec \rho_{j_{l+1}}} r_u(0, x_u^{v_j}) + r_{\rho_{j_n}}(0, x_{\rho_{j_n}}^{v_j}) \mathbbm{1}\{v_j \neq j_n \} \\
&\ \ \ +  r_{v_i \wedge v_j}(x_{v_i \wedge v_j}^{v_i}, x_{v_i \wedge v_j}^{v_j}) \\
&=\sum_{v_i \wedge v_j \prec u \preceq v_i} r_u(0, x_u^{v_i}) + \sum_{v_i \wedge v_j \prec u \preceq v_j} r_u(0, x_u^{v_j}) +  r_{v_i \wedge v_j}(x_{v_i \wedge v_j}^{v_i}, x_{v_i \wedge v_j}^{v_j}) \\
&=R(v_i, v_j).
\end{align*}
The cases when $L_i$ and $L_j$ are related differently are dealt with similarly (and more straightforwardly) to show that $r'(v_i,v_j)=R(v_i, v_j)$ for all $v_i, v_j$ in $V^{\wedge}$, thus proving the result.
\end{proof}
\end{lem}

Since $|V'| < \infty$, it then follows from \cite[Proposition 2.1.11 and Theorem 2.1.12]{AOF} that we can reduce $G'$ to a network with vertex set precisely equal to $V$, and define a metric $r$ on $G$ to be $r'|_V$, the projection of $r'$ onto $V$, and $r$ will be a resistance metric agreeing with $r'$ on $V'$. Proposition \ref{prop:res metric} then follows.

We now prove the following extension of \cite[Theorem 4.1]{RSLTCurKort}.

\begin{prop}\label{thm:compact disc inv princ res}
Let $(\tau_n)_{n=1}^{\infty}$ be a sequence of trees with $|\tau_n| \rightarrow \infty$ and corresponding Lukasiewicz paths $(W^n)_{n = 1}^{\infty}$, and let $R_n$ denote the effective resistance metric on $\Loop(\tau_n)$ obtained via (\ref{eqn:resistance def variational}) by giving each edge conductance $1$. Additionally let $\nu_n$ be the uniform measure that gives mass $1$ to each vertex of $\Loop(\tau_n)$, and let $\rho_n$ be the root of ${\Loop}(\tau_n)$. Suppose that $(C_n)_{n=1}^{\infty}$ is a sequence of positive real numbers such that  
\begin{enumerate}[(i)]
\item $\Big( \frac{1}{C_n} W^n_{\lfloor |\tau_n| t \rfloor} (\tau_n) \Big)_{0 \leq t \leq 1} \overset{(d)}{\rightarrow} \X$ as $n \rightarrow \infty$,
\item $\frac{1}{C_n} \textsf{Height}(\tau_n) \ \overset{\mathbb{P}}{\rightarrow} \ 0$ as $n \rightarrow \infty$.
\end{enumerate}
Then
\[
\Big({\Loop}(\tau_n), \frac{1}{C_n}R_n, \frac{1}{|\tau_n|} \nu_n, \rho_n \Big) \overset{(d)}{\rightarrow} \Big( \La, R, \nu, \rho \Big)
\]
as $n \rightarrow \infty$ with respect to the Gromov-Hausdorff-Prohorov topology.
\begin{proof}
As a result of Lemma \ref{lem:dR compare}, Gromov-Hausdorff convergence follows exactly as in the proof of \cite[Theorem 4.1]{RSLTCurKort} by applying the Skorohod Representation Theorem and then defining a correspondence $\mathcal{R}_n$ between $\La$ and $\Loop (\tau_n)$ to consist of all pairs of the form $(t, \lfloor \tau_n (t) \rfloor)$ or $(t, \lceil \tau_n (t) \rceil)$.

To prove that the measures also converge on this space, let $r_n = dis(\mathcal{R}_n)$, and take the Gromov-Hausdorff embedding $F_n = \Loop (\tau_n) \sqcup \La$ endowed with the metric
\[
D_{F_n}(x,y) = \begin{cases}
\frac{1}{C_n} R_n(x,y) & \text{ if } x, y \in \Loop (\tau_n) \\
R(x,y) & \text{ if } x, y \in \La \\
\inf_{u, v \in \mathcal{R}_n} (\frac{1}{C_n} R_n(x,u) + R(y,v) + \frac{1}{2} r_n) & \text{ if } x \in \Loop (\tau_n), y \in \La.
\end{cases}
\]

We claim that $d^{F_n}_P(\nu_n, \nu) \rightarrow 0$ as $n \rightarrow \infty$.
For notational convenience we will assume that $|\tau_n|=n$, and let $I_{n,i} = [\frac{i}{n} - \frac{1}{2n}, \frac{i}{n} + \frac{1}{2n}]$. Let $u_0, u_1, \ldots, u_{n-1}$ denote the lexicographical ordering of these vertices. Take a set $A_n$ of vertices in $\Loop (\tau_n)$, and let
\[
A'_n = \bigcup_{i: u_i \in A_n} I_{n,i}.
\]
Let $A_n'' = p(A'_n)$. We will show that $A_n'' \subset A_n^{r_n}$. For any $v \in A_n'', \exists \ s \in A'_n$ with $v=p(s)$ and $s \in I_{n,i}$ for some $i$ with $u_i \in A_n$. It follows that $i = \lfloor ns \rfloor$ or $\lceil ns \rceil$, and hence $(u_i, v) \in \mathcal{R}_n$ and $D_F(u_i,v) = \frac{1}{2} r_n$. It follows that $v \in A_n^{r_n}$ and $A_n'' \subset A_n^{r_n}$. Also note that $\nu_n(A_n) = \nu(A_n'')$ by construction, and so $\nu_n(A_n) \leq \nu(A_n^{r_n})$.

Similarly, take a set $B \subset \La$. We similarly show that $\nu(B) \leq \nu_n(B^{r_n})$. Let $B' = p^{-1}(B)$, and 
\[
B_n = \{ u_i \in L_n \colon \exists \ s \in B' \text{ with } s \in I_{n,i} \}.
\]
Clearly $B' \subset \bigcup_{u_i \in B_n} I_{n,i}$ and so
\[ 
\nu(B) = Leb(B') \leq \frac{|B_n|}{n} = \nu_n(B_n).
\]
If $u_i \in B_n$, then there exists $s \in B'$ with $s \in I_{n,i}$ and so $(u_i, p(s)) \in \mathcal{R}_n$. Hence $B_n \subset B^{r_n}$, so $\nu_n(B_n) \leq \nu_n(B^{r_n})$ and $\nu(B) \leq \nu_n(B^{r_n})$.

It follows that $d^{F_n}_P(\nu_n, \nu) \leq r_n$, and hence converges to zero as $n \rightarrow \infty$.
\end{proof}
\end{prop}

\subsection{Random Walk Scaling Limits}
In light of Proposition \ref{prop:res metric}, we define Brownian motion on $\L_{\alpha}$ to be the diffusion associated with $(\L_{\alpha}, R, \nu, \rho)$ as in Section \ref{sctn:res forms}. We now show that this is the scaling limit of random walks on discrete looptrees.

\begin{proof}[Proof of Theorem \ref{thm:main RW inv princ compact quenched}]
It follows from Proposition \ref{thm:compact disc inv princ res}, separability and the Skorohod Representation Theorem that there exists a probability space on which the convergence of Proposition \ref{thm:compact disc inv princ res} holds almost surely. We will show that on this space, the laws of the given stochastic processes converge weakly, almost surely.

The stochastic process $Y^{(n)}$ naturally associated with the quadruplet $(\Loop (T_n), R_n, \nu_n, \rho_n)$ in the sense of Section \ref{sctn:res forms} is a continuous time random walk that jumps from its present state to each of its neighbouring vertices at rate $1$. Since every vertex of these discrete looptrees has degree $4$ (we consider self-loops as undirected), this amounts to an $\textsf{exp}(4)$ waiting time at every vertex.

Now define processes $(\tilde{Z}^{(n)}_t)_{t \geq 0}$ and $(\tilde{Y}^{(n)}_t)_{t \geq 0}$ by $\tilde{Z}^{(n)}_t = a_n^{-1} Z^{(n)}_{\lfloor 4na_nt \rfloor}$, and $\tilde{Y}^{(n)}_t = a_n^{-1} Y^{(n)}_{na_nt}$. It follows from Theorem \ref{thm:scaling lim RW resistance} that almost surely as $n \rightarrow \infty$, we have the weak convergence
\begin{equation}\label{eqn:RW conv VSRW}
( \tilde{Y}^{(n)}_{t})_{t \geq 0} \rightarrow (B_t)_{t \geq 0}.
\end{equation}

To deduce the result for $\tilde{Z}^{(n)}$ in place of $\tilde{Y}^{(n)}$, we will show that we can couple the processes $Y^{(n)}$ and $Z^{(n)}$ so that they almost surely have the same limit. To do this, note that we can obtain $Y^{(n)}$ from $Z^{(n)}$ by sampling a sequence of independent \textsf{exponential}($4$) random variables $(w^{(n)}_i)_{i=1}^{\infty}$, letting $S^{(n)}_m = \sum_{i=1}^m w^{(n)}_i$ for all $m \in \N$, and setting $Y^{(n)}_{t} = Z^{(n)}_{m}$ for all $t \in [S^{(n)}_m, S^{(n)}_{m+1})$. In particular, $Y^{(n)}_{S^{(n)}_m} = Z^{(n)}_{m}$ for all $m$.

Fix some $T< \infty$. Since the limit process $(B_t)_{t \geq 0}$ is almost surely continuous, the convergence of (\ref{eqn:RW conv VSRW}) actually holds with respect to the uniform topology. By again appealing to the Skorohod representation theorem along with a functional law of large numbers, we can therefore restrict to a probability space where $\big((\tilde{Y}_t)_{t \in [0,T]}, \big((na_n)^{-1}S^{(n)}_{\lfloor 4C_{\alpha} n^{1 + \frac{1}{\alpha}} t \rfloor}\big)_{t \in [0.T]}\big) \rightarrow ((B_t)_{t \in [0,T]}, t)$ jointly almost surely.

By composing these continuous limits, we therefore deduce that
\begin{align*}
(\tilde{Z}_t)_{t \in [0,T]} = \Big(\tilde{Y}_{na_n S^{(n)}_{\lfloor 4na_n \rfloor}}\Big)_{t \in [0,T]} \rightarrow (B_t)_{t \in [0,T]},
\end{align*}
uniformly almost surely. This proves that the distributional result holds for arbitrary $T < \infty$, and we extend to all time by applying \cite[Lemma 16.3]{BillsleyConv}.
\end{proof}

\begin{rmk}
It is also possible to deduce an annealed convergence result by embedding into the universal Urysohn space. We will not pursue this further here, but refer to \cite[Section 2.2]{ArchInfiniteLooptrees} for full details.
\end{rmk}

\begin{rmk}
It also follows from \cite[Theorem 1 and Proposition 14]{CroyHamLLT} that the transition densities of the discrete time random walks on any compact time interval will converge to those of $(B_t)_{t \geq 0}$ under the same rescaling when we isometrically embed in the space $(M,d_M)$ as described above. This can be metrized using the spectral Gromov-Hausdorff distance, introduced in \cite[Section 2]{CroyHamKumMixing}. It also follows by an application of \cite[Theorem 1.4]{CroyHamKumMixing} that for any $p \in [1, \infty)$, the rescaled $L^p$-mixing times for $\Loop (\tau_n)$ will converge to those of $\La$. We expect that we can prove similar convergence results for blanket times using ideas of \cite{AndrioBlanket}, and that the sequence of cover times will be Type 2 in the sense of \cite[Definition 1.1]{AbeCover}.
\end{rmk}

\section{Volume Bounds for Compact Stable Looptrees}\label{sctn:vol bounds}
In this section we prove Theorems \ref{thm:main global} and \ref{thm:main local}. Recall that $\mathbf{P}$ denotes the law of $\La$, and we let $U$ be $\textsf{Uniform} ([0,1])$. For ease of intuition, we define the open ball $B(u,r)$ using the metric $d$ rather than $R$.

\subsection{Infimal Lower Bounds}\label{sctn:inf LBs}
We prove the lower volume bounds of Theorems \ref{thm:main global} and \ref{thm:main local} via the following proposition.

\begin{prop}\label{prop:inf LB vol loc prob}
There exist constants $c, C, r_0 \in (0, \infty)$ such that for all $r \in (0, r_0)$ and all $\lambda \in (0, \frac{1}{2}r^{-\alpha})$,
\begin{align*}
\prb{ \nu (B(p(U), r)) < r^{\alpha} \lambda^{-1}} &\leq  C\exp \{-c\lambda^{\frac{1}{\alpha}}\}.
\end{align*}
\end{prop}

The proof of Proposition \ref{prop:inf LB vol loc prob} uses ideas from the proof of the upper bound on the Hausdorff dimension of $\La$ that was given in \cite[Section 3.3.1]{RSLTCurKort}. It relies on the fact that for any $s,t \in [0,1]$ with $s \leq t$,
\begin{equation}\label{eqn:d osc compare}
d(p(s), p(t)) \leq \X_s + \X_t - 2\inf_{s \leq r \leq t} \X_r.
\end{equation}
This result appears as \cite[Lemma 2.1]{RSLTCurKort}. Consequently, we can lower bound the volume of small balls in $\La$ by upper bounding the oscillations of $\X$. We use the notation $\diam_f (p([a,b])$ to denote the diameter of the set $p(a,b)$ defined from $f$ using the distance function of (\ref{eqn:d}), but with $f$ in place of $\X$.

We first give a technical lemma which appeared previously in \cite[Section 3.3.1]{RSLTCurKort} and uses an argument from \cite{BertoinLevy}. The final claim follows by bounded convergence.

First recall that for a function $f: [0, \infty) \rightarrow \R$ and $[a,b] \subset [0, \infty)$, we define
\[
\osc_{[a,b]} f := \sup_{s, t \in [a,b]} |f(t) - f(s)|.
\]

\begin{lem}\label{lem:osc}
Let $\mathcal{E}$ be an exponential random variable with parameter $1$, and let $X$ be a spectrally positive $\alpha$-stable \Levy process conditioned to have no jumps of size greater than $1$ on $[0, \mathcal{E}]$. Let $\tilde{\osc} = \osc_{[0, \mathcal{E}]} X$. Then there exists $\theta > 0$ such that $\E{e^{\theta \tilde{\osc}}} < \infty$. Moreover, $\E{e^{\theta \tilde{\osc}}} \downarrow 1$ as $\theta \downarrow 0$.
\end{lem}

\begin{rmk}\label{rmk:osc deterministic}
The same results holds if $\EE$ is set to be deterministically equal to $1$ rather than an exponential random variable. The proof is almost identical to the one above, with one minor modification.
\end{rmk}

\begin{proof}[Proof of Proposition \ref{prop:inf LB vol loc prob}]
First, note the inclusion 
\begin{align*}
\{ \nu(B(p(U), r)) < r^{\alpha} \lambda^{-1}\} &\subset \Big\{ p([U, U+r^{\alpha} \lambda^{-1}]) \cap B^c(p(U), r) \neq \emptyset \Big\} \subset \Big\{ \diam_{\X} ( p {[U, U+r^{\alpha} \lambda^{-1}]}) > r \Big\}.
\end{align*}
By applying the Vervaat transform, the absolute continuity relation (\ref{eqn:abs cont RN deriv Levy bridge}) and scaling invariance, we get that 
\begin{align*}
\pr{ \diam_{\X} ( p {[U, U+r^{\alpha} \lambda^{-1}]}) > r } \leq \frac{(1 - r^{\alpha} \lambda^{-1})^{\frac{-1}{\alpha}}||p_1||_{\infty}}{p_1(0)} \pr{ \diam_X ( p {[0,1]}) > \lambda^{\frac{1}{\alpha}}}.
\end{align*}
To bound the latter quantity, let $N$ be the cardinality of the set $\{ t \in [0,1]: \Delta_t > 1 \}$, where $\Delta_t = X_t - X_{t^-}$ now denotes the jump size of $X$ rather than $\X$, and let $t_1, \ldots, t_{N}$ be its members in increasing order of size. Additionally let $t_0=0$ and $t_{N+1} = 1$, and $\tilde{C_{\alpha}} = \frac{\alpha - 1}{\Gamma(2 - \alpha)}$, so that $N \sim \textsf{Poi}(\tilde{C_{\alpha}})$. We then have:
\begin{align*}
\pr{\diam_X (p[0,1]) > \lambda^{\frac{1}{\alpha}}} 
&\leq \sum_{n=1}^{\infty} \frac{e^{-\tilde{C_{\alpha}}} (\tilde{C_{\alpha}})^n}{n !}
\prcond{\sum_{i=1}^N \osc_{[t_i, t_{i+1}]} X > \lambda^{\frac{1}{\alpha}}}{N =n}{} \\
&\leq \sum_{n=1}^{\infty} \frac{e^{-\tilde{C_{\alpha}}} (\tilde{C_{\alpha}})^n}{n !} \E{e^{\theta \tilde{\osc}}}^{n} \exp \{- \theta \lambda^{\frac{1}{\alpha}}\},
\end{align*}
where $\tilde{Osc}$ is as in Remark \ref{rmk:osc deterministic}. Note that $N$ and $(\tilde{\osc}_{[t_i, t_{i+1}]})_{i \leq N}$ are not independent, but we certainly have $t_{i+1} - t_i \leq 1$ for all $i$, and hence by Lemma \ref{lem:osc} and Remark \ref{rmk:osc deterministic} we can choose $\theta$ small enough that $C_{\theta} := \E{e^{\theta \tilde{\osc}}} < \infty$. The result follows from noting that
\begin{align*}
\pr{\diam_X (p[0,1]) > \lambda^{\frac{1}{\alpha}}} &\leq e^{(C_{\theta} - 1)\tilde{C_{\alpha}}} e^{- \theta \lambda^{\frac{1}{\alpha}}}.
\end{align*}
\end{proof}

By taking a union bound, the same argument can be used to give a bound on the global infimum.

\begin{prop}\label{prop:prob glob vol LB}
There exist constants $c, C, r_0 \in (0, \infty)$ such that for all $r \in (0, r_0)$ and all $\lambda \in (0, \frac{1}{2}r^{-\alpha})$,
\[
\prb{ \inf_{u \in \L_{\alpha}} \nu \Big( B (u,r) \Big) < r^{\alpha} \lambda^{-1} } \leq C r^{- \alpha} \lambda  \exp \{ -c \lambda^{\frac{1}{\alpha}} \}.
\]
\begin{proof}
By the same reasoning as in the proof of Proposition \ref{prop:inf LB vol loc prob}, we have:
\begin{small}
\begin{align*}
\{ \inf_{u \in \L} \nu(B(u,r)) < r^{\alpha} \lambda^{-1}\} &\subset \Big\{ \diam_{\Xb} (p{[kr^{\alpha} \lambda^{-1}, (k+1) r^{\alpha} \lambda^{-1} \wedge 1]})> \frac{1}{2} r \text{ for some } k=0, \ldots, \lfloor r^{-\alpha} \lambda \rfloor \Big\},
\end{align*}
\end{small}
and hence
\begin{small}
\begin{align*}
\prb{ \inf_{u \in \L} \nu(B(u,r)) < r^{\alpha} \lambda^{-1}} &\leq \prb{ \diam_{\Xb} (p{[kr^{\alpha} \lambda^{-1}, (k+1) r^{\alpha} \lambda^{-1} \wedge \frac{1}{2}]})> \frac{1}{2} r \text{ for some } k=0, \ldots, \lfloor \frac{1}{2} r^{-\alpha} \lambda \rfloor } \\
&\ + \prb{ \diam_{\Xb} (p{[\frac{1}{2} \vee kr^{\alpha} \lambda^{-1}, (k+1) r^{\alpha} \lambda^{-1} \wedge 1]})> \frac{1}{2} r \text{ for some } k= \lfloor \frac{1}{2} r^{-\alpha} \lambda \rfloor, \ldots, \lfloor r^{-\alpha} \lambda \rfloor } \\
 \\
&\leq C_{\theta} r^{-\alpha} \lambda \frac{||p_{\frac{1}{2}}||_{\infty}}{p_1(0)} e^{-\theta \lambda^{\frac{1}{\alpha}}},
\end{align*}
\end{small}
where the final line follows by Proposition \ref{prop:inf LB vol loc prob}.
\end{proof}
\end{prop}

\begin{proof}[Proof of infimal lower bounds in Theorems \ref{thm:main global} and \ref{thm:main local}]
Take $c$ as in Proposition \ref{prop:inf LB vol loc prob}, and $M>c^{-1}$. Set
\begin{align*}
g(r) &= Mr^{\alpha}(\log \log r^{-1})^{-\alpha}, \hspace{5mm} \text{ and } \hspace{5mm} J_r = \{ \nu (B(p(U), r)) < g(r) \}.
\end{align*}
Taking $\lambda = M(\log \log r^{-1})^{\alpha}$ in Proposition \ref{prop:inf LB vol loc prob} we see that $\prb{ J_r} \leq C(\log r^{-1})^{-cM}$, and since $M > c^{-1}$ we have by Borel-Cantelli that $\prb{J_{2^{-k}} \text{ i.o.}} = 0$. Hence there almost surely exists $K \in \N$ such that $J^c_{2^{-k}}$ occurs for all $k \geq K$. On this event, $\nu (B(p(U), r)) \geq  2^{-\alpha} g(r)$ for all sufficiently small $r$, or equivalently,
\begin{equation}
\liminf_{r \downarrow 0} \Bigg( \frac{\nu(B(p(U), r))}{r^{\alpha}(\log \log r^{-1})^{-\alpha}} \Bigg) \geq 2^{-\alpha }M.
\end{equation}

To deduce the result for $\nu$-almost every $u \in \La$ we apply Fubini's theorem.  Letting
\[
F(\L_{\alpha},u) = \mathbbm{1} \Big\{ \liminf_{r \downarrow 0} \Bigg( \frac{\nu(B(u, r))}{r^{\alpha}(\log \log r^{-1})^{-\alpha}} \Bigg) \geq 2^{-\alpha }M \Big\},
\]
we have from above that
\[
\int_0^1 \Eb{ F(\L_{\alpha}, u) }du  = \Eb{F(\L_{\alpha}, p(U))} = 1.
\]
By Fubini's theorem, this implies that almost surely, $F(\L_{\alpha}, u) = 1$ for Lebesgue almost every $u \in [0,1]$, and consequently for $\nu$-almost every $u \in \L_{\alpha}$. This proves (\ref{eqn:loc inf LB}).

The proof of the global bound (\ref{eqn:glob inf LB}) is similar. Take $c$ as in Proposition \ref{prop:inf LB vol loc prob}, choose some $A > \alpha c^{-1}$, and set $\epsilon = A - \alpha c^{-1}$. Then, setting $\lambda = (A\log r^{-1})^{\alpha}$ we have by Proposition \ref{prop:prob glob vol LB} that:
\begin{align*}
\prb{ \inf_{u \in \L_{\alpha}} \nu \Big( B(u,r) \Big) < r^{\alpha} (A\log r^{-1})^{-\alpha} } &\leq C r^{\epsilon}(\log r^{-1})^{\alpha}.
\end{align*}
Consequently, letting
\[
K_r = \big\{ \inf_{u \in \L_{\alpha}} \nu \Big( B_r (u) \Big) < r^{\alpha} (A\log r^{-1})^{-\alpha} \big\},
\]
we have by Borel-Cantelli that $\prb{K_{2^{-k}} \text{ i.o.}} = 0$. Hence, there almost surely exists a $K_0 < \infty$ such that for any $r < 2^{-K_0}$ we have 
(\ref{eqn:glob inf LB}), or more precisely that:
\[
\inf_{u \in \L_{\alpha}} \nu \Big( B(u,r) \Big) \geq 2^{-\alpha} r^{\alpha} (A\log r^{-1})^{- \alpha}.
\]
\end{proof}

\subsection{Supremal Upper Bounds}\label{sctn:sup vol UBs}
In this section we prove (\ref{eqn:glob sup UB}) and (\ref{eqn:loc sup UB}) using the following Williams' Decomposition. By appealing to uniform re-rooting invariance, we will treat $p(U)$ as the root of the looptree throughout.

\subsubsection{Williams' Decomposition}\label{sctn:Williams Decomp}
The Williams' Decomposition of \cite{AbDelWilliamsDecomp} gives a decomposition of a stable tree $\tilde{\Ta}$ along its spine of maximal height. In the Brownian case $\alpha = 2$, this corresponds to Williams' decomposition of Brownian motion. Letting $H_{\text{max}} = \sup_{u \in \tilde{\Ta}} d_{\tilde{\Ta}}(\rho, u)$, we see from \cite[Equation (23)]{DuqWangDiameter} (and references therein) that there is almost surely a unique $u_h \in \tilde{\Ta}$ such that $d_{\tilde{\Ta}}(\rho, u_h) = H_{\text{max}}$. We define the Williams' spine (or W-spine) of $\tilde{\Ta}$ to be the segment $[[\rho, u_h]]$, and take the Williams' loopspine (or W-loopspine) in the corresponding looptree $\La$ to be the closure of the set of loops coded by points in $[[\rho, u_h]]$. A main result of \cite{AbDelWilliamsDecomp} is a theorem which firstly gives the distribution of the loop lengths along the W-loopspine, and additionally the distribution of the fragments obtained by decomposing along it.

Given the spine from $\rho$ to $u_h$, and conditional on $\Hm = H$, the loops along the W-loopspine can be represented by a Poisson point measure $\sum_{j \in J} \delta (l_j, t_j, u_j)$ on $\R^+ \times [0, H] \times [0,1]$ with a certain intensity. A point $(l,t,u)$ corresponds to a loop of length $l$ in the W-loopspine, occurring on the W-spine at distance $t$ from the root in the corresponding tree $\tilde{\Ta}$, and such that a proportion $u$ of the loop is on the ``left" of the W-loopspine, and a proportion $1-u$ is on the ``right". In \cite{AbDelWilliamsDecomp}, this is written in terms of the exploration process on $\tilde{\Ta}$, but  we interpret their result below in the context of looptrees.

We note that when stating this result, we are not conditioning on the total mass of $\tilde{\Ta}$: only the maximal height. The mass will depend on its height via the joint laws for these under the \Ito excursion measure.

\begin{theorem}(Follows directly from \cite[Lemma 3.1 and Theorem 3.3]{AbDelWilliamsDecomp}).\label{thm:AbDel Williams Decomp}
\begin{enumerate}[(i)]
\item Conditionally on $H_{\text{max}}=H$, the set of loops in the W-loopspine forms a Poisson point process $\mu_{\textsf{W-loopspine}} = \sum_{j \in \mathcal{J}} \delta (l_j, t_j, u_j)$ on the W-spine in the corresponding tree with intensity
\[
\mathbbm{1}_{\{ [0,1] \}}(u) \mathbbm{1}_{\{[0,H]\}}(t) l \exp \{ -l (H-t)^{\frac{-1}{\alpha - 1}} \} du \ dt \ \Pi (dl),
\]
where $\Pi$ is the underlying \Levy measure, with $\Pi(dl) = \frac{1}{|\Gamma(-\alpha)|} l^{-\alpha - 1} \mathbbm{1}_{(0, \infty)}(l) dl$ in the stable case. We will denote the atom $\delta (l_j, t_j, u_j)$ by $\textsf{Loop}_j$.
\item Let $\delta (l, t, u)$ be an atom of the Poisson process described above. The set of sublooptrees grafted to the W-loopspine at a point in the corresponding loop can be described by a random measure $M^{(l)} = \sum_{i \in I} \delta^{(l)} (\EE_i, D_i)$, where $\EE_i$ is a \Levy excursion that codes a looptree in the usual way, and $D_i$ represents the distance going clockwise around the loop from the point at which this sublooptree is grafted to the loop, to the point in the loop that is closest to $\rho$. This measure has intensity
\begin{align*}
N( \cdot, \Hm \leq H-t) \times \mathbbm{1}_{\{[0,l]\}}(D) dD.
\end{align*}
In particular, the sublooptrees are just rescaled copies of our usual normalised compact stable looptrees, and each of these is grafted to the loop on the W-loopspine at a uniform point around the loop lengths.
\end{enumerate}
\end{theorem}

\begin{rmk}
Point $(ii)$ is a slight extension of the results of \cite{AbDelWilliamsDecomp}, where the authors write that the intensity of subtrees incident to the W-spine at a node of ``degree" $l$ has intensity $l N( \cdot, \Hm \leq m-t)$. However, it follows from \cite[Equation (11)]{DuqLeGPFALT} and the remarks below it that the corresponding sublooptrees are actually distributed uniformly around the corresponding loop.
\end{rmk}

\subsubsection{Encoding the Looptree Structure in a Branching Process}
The Williams' decomposition suggests a natural way to encode the fractal structure of $\La$ in a branching process or cascade. Specifically, we let $\emptyset$ denote the root vertex of our cascade. This represents the whole looptree $\La$ (in particular, $\emptyset$ should not be confused with $\rho$, which is the root of $\La$). By performing the Williams' decomposition on $\La$ and removing the W-loopspine, the fragments obtained are countably many smaller copies of $\La$, which we view as the children of $\emptyset$ in our branching process, and index by $\N$. Moreover, to each edge joining $\emptyset$ to one of its offspring $i$, we associate a random variable $m (\emptyset,i)$ which gives the mass of the sublooptree corresponding to index $i$. The root of a sublooptree is the point at which it is grafted to the W-loopspine of its parent.

We can then perform further Williams' decompositions of these sublooptrees. More precisely, if $i$ is a child of $\emptyset$, we can decompose along its W-loopspine from its root to its point of maximal tree height to obtain a countable collection of offspring of $i$ that correspond to the fragments obtained on removing this W-loopspine, and label the offspring as $(ij)_{j \geq 1}$. By repeating this procedure again and again on the resulting subsublooptrees, we can keep iterating to obtain an infinite branching process.

\begin{rmk}
It may seem more straightforward to use a spinal decomposition to a uniform point (as in \cite{HPWSpinPart}) as the basis of this iteration, rather than the Williams' decomposition. However, this leads to technical difficulties in the case when $V$ is chosen so that $p(V)$ is a point too close to $p(U)$, and it is convenient to avoid this by instead decomposing along the maximal spine in the underlying tree.
\end{rmk}

We index this process using the Ulam-Harris labelling convention defined in Section \ref{sctn:trees background discrete}. Using the notation of \cite{Neveu}, an element of our branching process will be denoted by $u = u_1 u_2 u_3 \ldots u_j$, and corresponds to a smaller sublooptree $L \subset \La$. Its offspring will all be of the form $(u i)_{i \in \N}$, with corresponding roots $(\rho_{u i})_{i \in \N}$, where $ui$ here abbreviates the concatenation $u_1 u_2 u_3 \ldots u_j i$, and each will correspond to one of the further sublooptrees obtained on performing a Williams' decomposition of $L$.

Moreover, to each edge joining $u$ to its child $u i$ we associate a random variable $m(u , u i)$. These give the ratios of the masses of each of the sublooptrees that correspond to the offspring of $u$, so that $\sum_{i=1}^{\infty} m(u, ui) = 1$ for all $u \in \mathcal{U}$. Given a particular element $u = u_1 u_2 \ldots u_j$ of the branching process, the overall mass of the corresponding sublooptree is then given by $M_u = \prod_{i =0}^{j-1} m(u_i, u_{i+1})$, where here we let $u_0$ denote the root $\emptyset$.

\subsubsection{Main Argument for Supremal Upper Bound}\label{sctn:sup UB actual arg}

The simplest way to upper bound the volume is to sum the masses of all the sublooptrees that are incident to the W-loopspine at a point within distance $r$ of $p(U)$, giving
\begin{align}\label{eqn:UB sum crude}
\nu (B(p(U), r)) \leq \sum_{i=1}^{\infty} {M}_i \mathbbm{1} \{\rho_i \in B(p(U), r) \}.
\end{align}

We would like to use this to bound $\prb{\nu(B(p(U), r)) \geq r^{\alpha} \lambda}$. However, this approach is not very sharp since the probability that there is such an incident sublooptree of mass greater than $r^{\alpha} \lambda$ is of order $\lambda^{\frac{-1}{\alpha}}$, and when this happens the bound on the right hand side of (\ref{eqn:UB sum crude}) is immediately too large. However, if this event occurs, it is actually likely that this sublooptree is not completely contained in $B(p(U), r)$, and so we are not really capturing the right asymptotics for the behaviour of $\nu(B(p(U), r))$ by applying (\ref{eqn:UB sum crude}).

To refine the argument we instead repeat the same procedure around the W-loopspine of the larger sublooptree. If there are no larger (sub)sublooptrees incident to the (sub)W-loopspine close to the (sub)root, then we conclude by summing the smaller terms; otherwise, we can keep repeating the same procedure and iterating further until eventually we reach a stage where there are no more ``large" sublooptrees to consider.

This iterative process corresponds to selecting a finite subtree $T$ of $\mathcal{U}$ in such a way that the elements of $T$ correspond to the large sublooptrees around which we perform further iterations. The offspring distribution of $T$ will be sufficiently subcritical that the process will die out fairly quickly. Conditioning on the extinction time and then on the total progeny of $T$, we bound the volume of the ball $B(p(U), r)$ by the sum of the masses of all the small sublooptrees that are grafted to the W-loopspine of each of the large sublooptrees.

Below, we describe how we select $T$ generation by generation as a subtree of $\mathcal{U}$. Throughout, we take:
\begin{align*}
\bf = \frac{\alpha - 1}{4\alpha - 3}, \hspace{5mm} \bg = \frac{\alpha - 1}{4\alpha - 3}, \hspace{5mm} \bE = \frac{2\alpha - 1}{2\alpha (4\alpha - 3)}, \hspace{5mm} \bd = \frac{1}{4\alpha - 3}.
\end{align*}
Note that $2\bE - \frac{1}{\alpha}(1 - \bf - \bg) = 0$. Also fix some $\kappa \in \big( 0, \big( \frac{1}{3e^2} \Gamma (1-\frac{1}{\alpha})\big)^{\alpha} \big)$. We need $\kappa$ to be sufficiently small to ensure that $T$ is sufficiently subcritical, but we will not be taking any kind of limit as $\kappa \downarrow 0$.

\begin{tcolorbox}[colback=white]
\textbf{Iterative Algorithm}\\

Start by taking $\emptyset$ to be the root of $T$. Recall this represents the whole looptree $\La$.
\begin{enumerate}
\item Perform a Williams' decomposition of $\La$ along its W-loopspine.
\item Consider the resulting fragments. To choose the offspring of $\emptyset$, select the fragments that have mass at least $\kappa^{-1} r^{\alpha} \lambda^{1-\bf-\bg}$, and such that the subroots of the corresponding looptrees are within distance $r$ of the root of $\emptyset$.
\item Repeat this process to construct $T$ one generation at a time. Given an element $u = u_1 u_2 \ldots u_j \in T$, there is a corresponding sublooptree $L_u$ in $\La$ with root $\rho_u$ and $M_u := \nu( L_u) \geq \kappa^{-1} r^{\alpha} \lambda^{1-\bf-\bg}$. Consider the fragments obtained in a Williams' decomposition of $L_u$, and select those that correspond to further sublooptrees that are within distance $r$ of $\rho_u$, and also such that they have mass at least $\kappa^{-1} r^{\alpha} \lambda^{1-\bf-\bg}$ (i.e. with $M_{u_1 u_2 \ldots u_j u_{j+1}} = \prod_{k=0}^{j} m(u_k, u_{k+1}) \geq \kappa^{-1} r^{\alpha} \lambda^{1-\bf-\bg}$), to be the offspring of $u$.
\item For each $u = u_1 u_2 \ldots u_j \in T$, set 
\[
S_u = \sum_{i=1}^{\infty} M_{ui} \mathbbm{1} \Big\{ \rho_{ui} \in B (\rho_u, r) \Big\} \mathbbm{1} \Big\{ M_{ui} < \kappa^{-1} r^{\alpha} \lambda^{1-\bf-\bg} \Big\}.
\]
\end{enumerate}
\end{tcolorbox}

As explained above, in the event that $T$ is finite we then have that:
\[
\nu(B(p(U), r)) \leq \sum_{u \in T} S_u.
\]

Since the Williams' decomposition involves conditioning on the height of the corresponding stable tree rather than its mass, we will prove this theorem by rescaling each sublooptree corresponding to an element of $T$ to have underlying tree height $1$, and then using Theorem \ref{thm:AbDel Williams Decomp} to analyse the fragments. Most of the effort in proving the supremal upper bounds is devoted to proving the following proposition.

\begin{prop}\label{prop:prob bound sup UB}
There exist constants $\tilde{c}, \tilde{C} \in (0, \infty)$ such that for all $r<1$ and all $\lambda > 1$, 
\begin{align*}
\prb{ \nu (B(p(U), r)) \geq r^{\alpha} \lambda} \leq \tilde{C} \lambda^{\frac{\alpha - 1}{4\alpha - 3}} e^{-\tilde{c}\lambda^{\frac{\alpha - 1}{4\alpha - 3}}}.
\end{align*}
\end{prop}

The volume results (\ref{eqn:glob sup UB}) and (\ref{eqn:loc sup UB}) follow from Proposition \ref{prop:prob bound sup UB} by Borel-Cantelli, similarly to those in the previous section. We sketch this below, and prove Proposition \ref{prop:prob bound sup UB} afterwards.

\begin{proof}[Proof of supremal upper bounds of Theorems \ref{thm:main global} and \ref{thm:main local}, assuming Proposition \ref{prop:prob bound sup UB}]
Take $\tilde{c}$ as in Proposition \ref{prop:prob bound sup UB}, and choose some $A > \tilde{c}^{-1}$.
Taking $\lambda_r = A (\log \log r^{-1})^{\frac{4\alpha - 3}{\alpha - 1}}$ in Proposition \ref{prop:prob bound sup UB} and applying Borel-Cantelli we deduce that $\prb{I_{2^{-k}} \text{ i.o.}} = 0$, where
\[
I_r = \{ \nu(B(p(U), r)) \geq r^{\alpha} \lambda_r \}.
\]
Similarly to the proof of the infimal bounds, it follows that
\[
\limsup_{r \downarrow 0} \Bigg( \frac{\nu(B(p(U), r))}{r^{\alpha} (\log \log r^{-1})^{\frac{4\alpha - 3}{{\alpha} - 1}}} \Bigg) \leq 2^{\alpha} A
\]
almost surely, and we extend to $\nu$-almost every $u \in \La$ using Fubini's theorem as before. This proves (\ref{eqn:loc sup UB}).

To prove the global bound (\ref{eqn:glob sup UB}), we have to do a bit more work. First take some $\epsilon > 0$, and define $\mathcal{W}$ to be the set of sets
\begin{align*}
\Big\{p([n c^{\alpha}(\alpha + \epsilon)^{-\alpha} r^{\alpha} (\log r^{-1})^{-\alpha(1+ \epsilon)}, (n+1) c^{\alpha}(\alpha + \epsilon)^{-\alpha} r^{\alpha} (\log r^{-1})^{-\alpha(1+ \epsilon)} )) : \vphantom{n \in \{0, 1, \ldots \lceil c^{-\alpha}(\alpha + \epsilon)^{\alpha} r^{-\alpha} (\log r^{-1})^{\alpha(1+ \epsilon)} \rceil \} \big\}}& \\
\vphantom{\big\{p([n c^{\alpha}(\alpha + \epsilon)^{-\alpha} r^{\alpha} (\log r^{-1})^{-\alpha(1+ \epsilon)}, (n+1) c^{\alpha}(\alpha + \epsilon)^{-\alpha} r^{\alpha} (\log r^{-1})^{-\alpha(1+ \epsilon)} )) :} n \in \{0, 1, \ldots \lfloor c^{-\alpha}(\alpha + \epsilon)^{\alpha}& r^{-\alpha} (\log r^{-1})^{\alpha(1+ \epsilon)} \rfloor \} \Big\},
\end{align*}
where $c$ takes the same value as it did in Proposition \ref{prop:prob glob vol LB}. It then follows from Proposition \ref{prop:prob glob vol LB} that
\begin{align}\label{eqn:glob UB 1}
\prb{ \mathcal{W} \text{ is an } r \text{-covering of } \La} \geq 1-  C c^{-\alpha}(\alpha + \epsilon)^{\alpha} r^{\epsilon} (\log r^{-1})^{\alpha(1+ \epsilon)}
\end{align}
for all sufficiently small $r$. Moreover, assuming that $\mathcal{W}$ is indeed an $r$-covering of $ \La$, and letting
\[
W^r =  \{x \in \L_{\alpha}: d(x,y) \leq r \text{ for some } y \in W \}
\]
be the $r$-fattening of $W$ for any set $W \in \mathcal{W}$, say with
\[
W = p([n c^{\alpha}(\alpha + \epsilon)^{-\alpha} r^{\alpha} (\log r^{-1})^{-\alpha(1+ \epsilon)}, (n+1) c^{\alpha}(\alpha + \epsilon)^{-\alpha} r^{\alpha} (\log r^{-1})^{-\alpha(1+ \epsilon)} )),
\]
we have that $W^r \subset B(p(nc^{\alpha}(\alpha + \epsilon)^{-\alpha} r^{\alpha} (\log r^{-1})^{-\alpha(1+ \epsilon)}, 2r)$. It hence follows that
\begin{align*}
\{ \sup_{u \in \La} \nu (B(u, r)) \leq r^{\alpha} \lambda_r \} \
&\subset  \Big\{ \{ \mathcal{W} \text{ is an } r \text{-covering of } \La\} \cap \{ \nu(W^r) \leq r^{\alpha} \lambda_r \forall W \in \mathcal{W} \} \Big\},
\end{align*}
and consequently,
\begin{align}\label{eqn:glob UB 2}
\begin{split}
\prb{\sup_{u \in \La} \nu (B(u, r)) \geq r^{\alpha} \lambda_r} &\leq \prb{ \mathcal{W} \text{ is not an } r \text{-covering of } \La} \\
&\hspace{10mm} + \prb{ \exists n: \nu(B(p(nc^{\alpha}(\alpha + \epsilon)^{-\alpha} r^{\alpha} (\log r^{-1})^{-\alpha(1+ \epsilon)}), 2r)) \geq r^{\alpha} \lambda_r }.
\end{split}
\end{align}
It follows from rerooting invariance at deterministic points that for any $n$,
\begin{align*}
\prb{\nu(B(p(nc^{\alpha}(\alpha + \epsilon)^{-\alpha} r^{\alpha} (\log r^{-1})^{-\alpha(1+ \epsilon)}), 2r)) \geq r^{\alpha} \lambda_r } &= \prb{\nu(B(\rho, 2r)) \geq r^{\alpha} \lambda_r } = \prb{\nu(B(p(U), 2r)) \geq r^{\alpha} \lambda_r },
\end{align*}
and hence by applying a union bound and Proposition \ref{prop:prob bound sup UB}, we see that
\begin{align*}
\prb{ \exists n: \nu(B(p(nc^{\alpha}(\alpha + \epsilon)^{-\alpha} r^{\alpha} (\log r^{-1})^{-\alpha(1+ \epsilon)}), 2r)) \geq r^{\alpha} \lambda_r } \leq C'r^{-\alpha} \lr^{\alpha(1+ \epsilon)} \lambda^{\frac{\alpha - 1}{4\alpha - 3}} e^{-\tilde{c}\lambda^{\frac{\alpha - 1}{4\alpha - 3}}}.
\end{align*}

In particular, taking $\lambda = \lambda_r = ((\alpha + \epsilon)\tilde{c}^{-1} \log r^{-1})^{\frac{4\alpha - 3}{{\alpha - 1}}}$, where $\tilde{c}$ is as it was in Proposition \ref{prop:prob bound sup UB}, we obtain
\begin{align}\label{eqn:glob UB 3}
\prb{ \exists n: \nu(B(p(nc^{\alpha}(\alpha + \epsilon)^{-\alpha} r^{\alpha} (\log r^{-1})^{-\alpha(1+ \epsilon)}), 2r)) \leq r^{\alpha} \lambda_r } \leq C'r^{\epsilon} \lr^{1+\alpha(1+ \epsilon)}.
\end{align}

By combining equations (\ref{eqn:glob UB 1}), (\ref{eqn:glob UB 2}) and (\ref{eqn:glob UB 3}), we therefore see that
\[
\prb{\sup_{u \in \La} \nu (B(u, r)) \geq r^{\alpha} \lambda_r} \leq C'r^{\epsilon} \lr^{1+\alpha(1+ \epsilon)}.
\]
Hence, letting $J_r = \{ \sup_{u \in \La} \nu (B(u, r)) \geq r^{\alpha} \lambda_r \}$, we have as before that $\prb{J_{2^{-k}} \text{ i.o.}} = 0$. This implies (\ref{eqn:glob sup UB}), since similarly to before, we deduce that there exists $r_0 > 0$ such that for all $r \in (0, r_0)$,
\[
\sup_{u \in \La} \nu (B(u, r)) \leq 2^{\alpha} r^{\alpha} (\log r^{-1})^{\frac{4\alpha - 3}{{\alpha - 1}}}.
\]
\end{proof}

For a given looptree $\tilde{L}_{\alpha}$ and a given $R>0$, we let $I_R$ denote the set of points in the W-loopspine that also fall within distance $R$ of the root. Formally,
\[
I_R = \bigcup_{s \preceq u_H} \{ t \geq s: \tilde{X}^{\text{exc}}_t = \inf_{s < r \leq t} \tilde{X}^{\text{exc}}_r, d(\rho, p(t)) < R \},
\]
where $\tilde{X}^{\text{exc}}$ is the \Levy excursion coding $\tilde{L}_{\alpha}$. $I_R$ can be endowed with a natural notion of length, denoted $|I_R|$, which can be thought of as the sum of the lengths of loop fragments contained in $I_R$. Formally, this can be defined as the Lebesgue measure of the closure of the set $\{\tilde{X}^{\exc}_t: t \in I_R\}$.

To bound the progeny of $T$, we can then use the Williams' decomposition to view the sublooptrees grafted to the W-loopspine as a Poisson process on $D([0, \infty), [0, \infty)) \times I_R$. In particular, the number of sublooptrees with mass greater than $m$ will be stochastically dominated by a Poisson with parameter $|I_R| N(\zeta > m)$, where $N$ denotes the \Ito excursion measure and $\zeta$ denotes the length of an excursion under this measure. $|I_R|$ will be roughly of order $R$, but the purpose of the next lemma is to control this more precisely.

\begin{lem}\label{lem:segment length bound}
Let $(\L^1_{\alpha}, \rho^1, d^1, \nu^1)$ be a compact stable looptree conditioned so that its underlying tree has height $1$, but with no conditioning on its mass. Take $R \leq \lambda^{-\beta_4}$, and let $I_R$ and $|I_R|$ be as above.
Then
\[
\prb{|I_R| \geq 3R\lambda^{2\bE}} \leq C(e^{-c\lambda^{\bd(\alpha - 1)}} + e^{- c\lambda^{2\bE}}) \leq Ce^{-c\lambda^{\frac{\alpha - 1}{4\alpha - 3}}}.
\]
\begin{proof}
It is possible that $|I_R|$ may be of order greater than $R$ if, for example, many of the loops close to the root have spinal branch points distributed such that they split the loop into two very unequal segments. We show that this occurs only with very low probability.

First note that, by Theorem \ref{thm:AbDel Williams Decomp}(i), the loops that fall on the first half of the W-spine stochastically dominate a Poisson point measure $\sum_{j \in \mathcal{J}} \delta (l_j, t_j, u_j)$ with intensity
\begin{equation}\label{eqn:height PP first half}
\mathbbm{1}_{\{ [0,1] \}}(u) \mathbbm{1}_{\{[0,\frac{1}{2}]\}}(t) l \exp \{ -l 2^{\frac{1}{\alpha - 1}} \} du \ dt \ \Pi (dl).
\end{equation}
Elements of the set $(t_j)_{j \in \mathcal{J}}$ correspond to distances along the spine in the underlying tree, but we will consider them as time indices throughout the remainder of this proof. We will model the loop lengths using a subordinator, where a jump of the subordinator of size $\Delta$ at time $t$ corresponds to a loop of length $\Delta$ which in turn corresponds to a node at a distance $t$ along the W-spine in the associated stable tree.

To prove the bound, we first condition on existence of a loop in the W-loopspine with length $l$ greater than $4R$ and with $u \in [\frac{1}{4}, \frac{3}{4}]$. We say that such a loop is ``good". We also say that a loop is ``goodish" if it just has length at least $4R$, with no restriction on $u$.  We then select the closest good loop to $\rho$. Given such a loop, the number of goodish loops between $\rho$ and the first good loop is stochastically dominated by a Geometric($\frac{1}{2}$) random variable. Letting this number be $N$, $|I_R|$ can then be upper bounded by the random variable
\[
2R(N+1) + \sum_{i=1}^{N+1} Q^{(i)},
\]
where $Q^{(i)}$ denotes the sum of the lengths of all the smaller loops on the W-loopspine that are between the $(i-1)^{\text{th}}$ and $i^{\text{th}}$ goodish loops, and the term $2R(N+1)$ comes from selecting a segment of length at most $R$ in each direction from the ``base point" around each of the goodish loops. Each $Q^{(i)}$ can be independently approximated by an $(\alpha - 1)$-stable subordinator run up until an exponential time and conditioned not to have any jumps greater than $4R$.

First let the number of good loops on the first half of the W-spine be equal to $M$. From (\ref{eqn:height PP first half}), it follows that $M$ stochastically dominates a Poisson random variable with parameter
\[
\kappa_R = \frac{1}{4} \int_{4R}^{8R} l \exp \{ -l 2^{\frac{1}{\alpha - 1}} \} \ \Pi (dl) \geq \frac{1}{4}\int_{4R}^{8R} l^{-\alpha} \exp \{ -8R 2^{\frac{1}{\alpha - 1}} \} dl \geq \tilde{C} R^{1-\alpha},
\]
where $\tilde{C} = \frac{1}{4(\alpha - 1)} (4^{1-\alpha} - 8^{1-\alpha})\exp \{ -8 \cdot 2^{\frac{1}{\alpha - 1}} \}$ is just a constant. Hence,
\begin{equation}\label{eqn:good loop exist bound}
\prb{M=0} \leq e^{-cR^{1-\alpha}} \leq e^{-c\lambda^{\bd(\alpha - 1)}}.
\end{equation}
We henceforth condition on $M>0$. Next, note that for any loop of length at least $4R$, the probability that it is good is at least $\frac{1}{2}$ (independently of the other loops), and so if we examine all such loops of the W-loopspine in order from $\rho$, as described in the previous paragraph, we have that $N + 1$ is stochastically dominated by a \textsf{Geo}($\frac{1}{2}$) random variable. Hence, for any $\theta > 0$, we have by  a Chernoff bound that
\begin{equation}\label{eqn:N geo bound}
\prb{N+1 \geq \lambda^{2\bE}} \leq \prb{\textsf{Geo}\Big(\frac{1}{2}\Big) \geq \lambda^{2\bE}} \leq C e^{-\lambda^{2\bE}}.
\end{equation}

To bound $\sum_{i=1}^{N+1} Q^{(i)}$, we again use (\ref{eqn:height PP first half}). Conditionally on $M>0$, (\ref{eqn:height PP first half}) implies that the times between each successive pair of goodish loops in the W-loopspine will each be independently stochastically dominated by an \textsf{exp}($2\kappa_R$) random variable, which we denote by $\mathcal{E}_R$. Hence, the sum of the smaller jumps between each pair can be stochastically dominated by $\textsf{Sub}_{\mathcal{E}_R}$, where $\textsf{Sub}$ is a subordinator with \Levy measure 
\[
cl^{-\alpha} \mathbbm{1}_{\{l \leq 4R\}} dl,
\]
Also let $\EE$ be an \textsf{exp}($2\tilde{C}$) random variable (recall that $\kappa_R = \tilde{C}R^{1-\alpha}$). It further follows by rescaling that 
\begin{align*}
\prb{\sum_{i=1}^{N+1} Q^{(i)} \geq R \lambda^{2\bE}} \leq \prb{\sum_{i=1}^{N+1} \textsf{Sub}^{(i)}_{\mathcal{E}_R} \geq R \lambda^{2\bE}} \leq \prb{\sum_{i=1}^{N+1} \textsf{Sub}^{(i)'}_{\mathcal{E}} \geq \lambda^{2\bE}},
\end{align*}
where $\textsf{Sub}^{(i)}$ are independent copies of $\textsf{Sub}$, and $\textsf{Sub}^{(i)'}$ are independent copies of a subordinator similar to $\text{Sub}$ but with \Levy measure
\[
cl^{-\alpha} \mathbbm{1}_{\{l \leq 4\}} dl.
\]
It then follows by Lemma \ref{lem:osc} that there exists $\theta > 0$ such that $C_{\theta} := \Eb{e^{\theta \text{Sub}'_{\mathcal{E}}}} < \frac{3}{2}$. For such $\theta$, we hence have
\begin{align}\label{eqn:Qi bound}
\begin{split}
\prb{\sum_{i=1}^{N+1} Q^{(i)} \geq R \lambda^{2\bE}} = \sum_{n=1}^{\infty}\prcondb{\sum_{i=1}^{N+1} \textsf{Sub}^{(i)'}_{\mathcal{E}} \geq R \lambda^{2\bE}}{N+1=n}{} \prb{N+1=n} &\leq \sum_{n=1}^{\infty} \Big(\frac{3}{2}\Big)^n e^{-\theta \lambda^{2\bE}} \Big(\frac{1}{2}\Big)^n \\
&= C_{\theta}'' e^{-\theta \lambda^{2\bE}}.
\end{split}
\end{align}
To conclude, we combine the results of (\ref{eqn:good loop exist bound}), (\ref{eqn:N geo bound}) and (\ref{eqn:Qi bound}) by writing
\begin{align*}
\prb{|I_R| \geq 3R\lambda^{2\bE}} &\leq \prb{M = 0} + \prcondb{N+1 \geq \lambda^{2\bE}}{M>0}{} + \prcondb{\sum_{i=1}^{N+1} Q^{(i)} \geq R \lambda^{2\bE}}{M>0}{} \\
&\leq C \big( e^{-c\lambda^{\bd(\alpha - 1)}} + C'_{\theta} e^{-c \lambda^{2\bE}} \big).
\end{align*}
\end{proof}
\end{lem}

The second technical lemma will allow us to bound the total progeny of $T$ by comparing it to a subcritical Galton-Watson tree with Poisson offspring distribution.

\begin{lem}\label{lem:iterative progeny Poisson compare}
Let $\tilde{\Ta}$ be a compact stable tree, and $\tilde{\La}$ be its corresponding compact stable looptree, both coded by the same excursion $\mathcal{E}$ under the \Ito measure $N(\cdot)$ but conditioned to have lifetime $\zeta$ at least $\kappa^{-1} r^{\alpha} \lambda^{1-\bf-\bg}$. Let $\rho$ be the root of $\tilde{\La}$, and perform a Williams' decomposition of $\tilde{\La}$ along its W-loopspine. Let $N$ denote the number of resulting sublooptrees obtained that are of mass at least $\kappa^{-1} r^{\alpha} \lambda^{1-\bf-\bg}$ and are also grafted to the W-loopspine within distance $r$ of the root of $\tilde{\La}$. Then
\begin{align*}
\prb{N \geq n} \leq  Ce^{-c \lambda^{1 - \bf - \bg - \alpha \bd}} + C(e^{-c\lambda^{\bd(\alpha - 1)}} + e^{- c\lambda^{2\bE}}) + \pr{ \textsf{Poisson}(K_{\alpha}) \geq n},
\end{align*}
where $K_{\alpha} = 3\big(\Gamma \big(1-\frac{1}{\alpha}\big)\big)^{-1} \kappa^{\frac{1}{\alpha}}$. The constants $c$ and $C$ also depend on $\kappa$, but $\kappa$ is fixed and the precise dependence will not be important, so we suppress this notationally.
\begin{proof}
Let $H$ be the height of $\tilde{\Ta}$, and let $\EE^{(H)}$ be the rescaled excursion given by
\[
\EE^{(H)} = \big(H^{\frac{-1}{\alpha - 1}} \EE_{H^{\frac{\alpha}{\alpha - 1}}t}\big)_{0 \leq t \leq H^{\frac{-\alpha}{\alpha - 1}} \zeta}.
\]
The excursion $\EE^{(H)}$ codes a tree conditioned to have height $1$ (this can be seen from combining \cite[Lemma 5.8, Part 3]{GoldHaasExtinctionStable} with \cite[Equation (26)]{DuqWangDiameter}, for example). Moreover, in the corresponding looptree, $N$ now denotes the number of sublooptrees of mass at least $H^{\frac{-\alpha}{\alpha - 1}} \kappa^{-1} r^{\alpha} \lambda^{1-\bf-\bg}$ that are grafted to the W-loopspine within distance $R := H^{\frac{-1}{\alpha - 1}}r$ of $\rho$.

We wish to bound $R$ so that we can apply Lemma \ref{lem:segment length bound}. To do this, note by monotonicity and scaling invariance that 
\begin{align*}
\prcondb{R \geq \lambda^{-\bd}}{\zeta \geq \kappa^{-1} r^{\alpha} \lambda^{1-\bf-\bg}}{}&\leq \prcondb{H \leq \kappa^{\frac{\alpha - 1}{\alpha}} \lambda^{\frac{-(\alpha - 1)({1-\bf-\bg})}{\alpha}} \lambda^{\bd(\alpha - 1)}}{\zeta = 1}{} \leq Ce^{-c \lambda^{1 - \bf - \bg -\alpha \bd}},
\end{align*}
where the final line holds by \cite[Theorem 1.8]{DuqWangDiameter}. Then, conditioning on $R \leq \lambda^{-\bd}$ (i.e. $H \geq r^{\alpha - 1} \lambda^{\bd(\alpha - 1)}$), we have by Lemma \ref{lem:segment length bound} that
\begin{align*}
\prcondb{|I_R| \geq 3R\lambda^{2\bE}}{R \leq \lambda^{-\bd}}{} = \prcondb{|I_R| \geq 3H^{\frac{-1}{\alpha - 1}}r\lambda^{2\bE}}{H \geq r^{\alpha - 1} \lambda^{-\bd(\alpha - 1)}}{} \leq C(e^{-c\lambda^{\bd(\alpha - 1)}} + e^{- c\lambda^{2\bE}}).
\end{align*}
By Theorem \ref{thm:AbDel Williams Decomp}(ii), the sublooptrees grafted to the W-loopspine at points in $I_R$ form a Poisson process of sublooptrees coded by the \Ito excursion measure, but thinned so that none have height large enough to violate the condition that the end of the W-spine corresponds to the point of maximal height in the tree. We can therefore stochastically dominate this by the unthinned (classical) version of the \Ito excursion measure of Section \ref{sctn:Ito exc}. Since $N(\zeta \geq t) = \hat{C}_{\alpha} t^{\frac{-1}{\alpha}}$, where $\hat{C}_{\alpha} = (\Gamma (1-\frac{1}{\alpha}))^{-1}$ (e.g. see \cite[Proposition 5.6]{GoldHaasExtinctionStable}), it follows that conditionally on $|I_R| \leq 3R \lambda^{2\bE} = 3H^{\frac{-1}{\alpha - 1}}r\lambda^{2\bE}$, $N$ is stochastically dominated by a Poisson random variable with parameter:
\begin{align*}
3\hat{C}_{\alpha} (\kappa^{-1} H^{\frac{-\alpha}{\alpha - 1}}r^{\alpha} \lambda^{1-\bf-\bg})^{\frac{-1}{\alpha}} H^{\frac{-1}{\alpha - 1}}r\lambda^{2\bE} = 3\hat{C}_{\alpha} \kappa^{\frac{1}{\alpha}}.
\end{align*}
To conclude, we write:
\begin{align*}
\prb{N \geq n} &\leq \prcondb{H \leq r^{\alpha - 1} \lambda^{\bd(\alpha - 1)}}{\zeta \geq \kappa^{-1} r^{\alpha} \lambda^{1-\bf-\bg}}{} + \prcondb{|I_R| \geq 3H^{\frac{-1}{\alpha - 1}}r\lambda^{2\bE}}{H \geq r^{\alpha - 1} \lambda^{-\bd(\alpha - 1)}}{} \\
&\ \ \ \ \ \ \ \ \ \ + \pr{ \textsf{Poisson}(3\hat{C}_{\alpha} \kappa^{\frac{1}{\alpha}}\lambda^{2\bE - \frac{1}{\alpha}({1-\bf-\bg})}) \geq n} \\
&\leq Ce^{-c \lambda^{1 - \bf - \bg - \alpha \bd}} + C(e^{-c\lambda^{\bd(\alpha - 1)}} + e^{- c\lambda^{2\bE}}) + \pr{ \textsf{Poisson}(3\hat{C}_{\alpha} \kappa^{\frac{1}{\alpha}}) \geq n}.
\end{align*}
\end{proof}
\end{lem}

Armed with these lemmas, there are now two key steps to the main argument. One of these is to bound the number of times we need to reiterate around larger sublooptrees as described by the algorithm, and the other is to bound the contributions of smaller terms from each of these iterations.

As is usual, we will let $|T|$ denote the total progeny of the tree $T$. The first main result is the following.

\begin{prop}\label{prop:prog bound}
There exist constants $c, C \in (0, \infty)$ such that
\begin{align*}
\prb{|T| \geq \lambda^{\bf}} \leq \lambda^{\bf} \Big[ Ce^{-c \lambda^{1 - \bf - \bg \alpha \bd}} + C(e^{-c\lambda^{\bd(\alpha - 1)}} + e^{- c\lambda^{2\bE}})\Big] + Ce^{-\lambda^{\bf}} \leq C\lambda^{\frac{\alpha - 1}{4\alpha - 3}} e^{-c\lambda^{\frac{\alpha - 1}{4\alpha - 3}}}.
\end{align*}
\begin{proof}
The main ingredient in this proof is the main theorem of Dwass from \cite{DwassProg}, that for a Galton-Watson tree with total progeny \textsf{Prog} and offspring distribution $\xi$, it holds that
\[
\pr{ \textsf{Prog} = k} = \frac{1}{k} \pr{\sum_{i=1}^k \xi^{(i)} = k-1},
\]
where the $\xi^{(i)}$ are i.i.d. copies of $\xi$. In particular, if $\xi \sim \textsf{Poisson}(\theta)$ for some $\theta < \frac{1}{e^2}$ we see by writing the sum explicitly and applying Stirling's formula that
\begin{equation}\label{eqn:Dwass bound geq}
\pr{ \textsf{Prog} \geq k} = \sum_{j \geq k} \frac{1}{j} \pr{ \textsf{Poisson}(j \theta) = j-1} \leq \frac{c}{\theta} \sum_{j \geq k} j^{\frac{-3}{2}} (e \theta)^j \leq \frac{c}{\theta} k^{\frac{-3}{2}} (e \theta)^k.
\end{equation}

This isn't a priori applicable since in our case $T$ is not quite a Galton-Watson tree. However, it follows from Lemma \ref{lem:segment length bound} that for any $k>0$, we have
\begin{align*}
\prb{|T| \geq k} \leq k \Big[ Ce^{-c \lambda^{1-\bf-\bg- \alpha \bd}} + C(e^{-c\lambda^{\bd(\alpha - 1)}} + e^{- c\lambda^{2\bE}})\Big] + \pr{ |T'| \geq k},
\end{align*}
where $T'$ is a Galton-Watson tree with \textsf{Poisson}($K_{\alpha}$) offspring distribution. Accordingly, setting $\theta =  K_{\alpha}$ (which is less than $\frac{1}{e^2}$ by our choice of $\kappa$) and $k = \lambda^{\bf}$ we see that
\begin{align*}
\pr{ |T'| \geq k} \leq C e^{-\lambda^{\bf}}
\end{align*}
so combining with the above we deduce that
\begin{align*}
\pr{|T| \geq \lambda^{\bf}} &\leq \lambda^{\bf} \Big[ Ce^{-c \lambda^{1 - \bf - \bg - \alpha \bd}} + C(e^{-c\lambda^{\bd(\alpha - 1)}} + e^{- c\lambda^{2\bE}})\Big] + Ce^{ - \lambda^{\bf}},
\end{align*}
as claimed.
\end{proof}
\end{prop}

\begin{prop}\label{prop:iteration small terms prob bound}
Conditional on $|T| \leq \lambda^{\bf}$, we have that
\[
\prb{\exists u \in T: S_u \geq r^{\alpha} \lambda^{1-\bf}} \leq C\lambda^{\bf} \Big[ e^{-c \lambda^{1-\bf-\bg-\alpha \bd}} + (e^{-c\lambda^{\bd(\alpha - 1)}} + e^{- c\lambda^{2\bE}}) + e^{-c\lambda^{1-\bf-2\bE\alpha}} \Big] \leq C\lambda^{\frac{\alpha - 1}{4\alpha - 3}} e^{-c\lambda^{\frac{\alpha - 1}{4\alpha - 3}}}.
\]
\begin{proof}
Take $u \in T$, and let $L_u$ be the corresponding (sub)looptree that forms part of $\La$. By the same arguments used in Proposition \ref{prop:prog bound}, we can use Lemma \ref{lem:segment length bound} to show that, letting $R = H^{\frac{-\alpha}{\alpha - 1}}r$, we have
\begin{align*}
\prb{|I_R| \geq 3R \lambda^{2\bE}} &\leq \prb{R \geq \lambda^{-\bd}} + \prcondb{|I_R| \geq 3R \lambda^{2\bE}}{R \geq \lambda^{-\bd}}{} \\
&\leq Ce^{-c \lambda^{1-\bf-\bg-\alpha \bd}} + C(e^{-c\lambda^{\bd(\alpha - 1)}} + e^{- c\lambda^{2\bE}})
\end{align*}
We now condition on $\{ |I_R| \geq 3R \lambda^{2\bE}\}$. Again dominating the thinned \Ito excursion measure by the classical \Ito excursion measure as we did in the proof of Proposition \ref{prop:iteration small terms prob bound}, we have that $H^{\frac{-\alpha}{\alpha - 1}} S_u$ is stochastically dominated by a subordinator with \Levy measure $C_{\alpha} x^{\frac{-1}{\alpha}-1} \mathbbm{1}{\{x \leq H^{\frac{-\alpha}{\alpha - 1}} \kappa^{-1} r^{\alpha} \lambda^{1-\bf-\bg} \}} dx$, run up until the time $3H^{\frac{-1}{\alpha - 1}}r\lambda^{2\bE}$. Note that the \Levy measure coincides with that of an $\alpha^{-1}$-stable subordinator, conditioned to have no jumps greater than $\kappa^{-1} H^{\frac{-\alpha}{\alpha - 1}} r^{\alpha} \lambda^{1-\bf-\bg}$.

Hence, letting \textsf{Subord} be an $\alpha^{-1}$-stable subordinator, and conditioning on $|I_R| \leq 3R\lambda^{2\bE}$, we have by scaling invariance that:
\begin{align*}
&\prcondb{S_u \geq r^{\alpha} \lambda^{1-\bf}}{|I_R| \leq R\lambda^{2\bE}}{} \\
&\hspace{1cm}= \prcondb{H^{\frac{-\alpha}{\alpha - 1}} S_u \geq H^{\frac{-\alpha}{\alpha - 1}}  r^{\alpha} \lambda^{1-\bf}}{|I_R| \leq R\lambda^{2\bE}}{} \\
&\hspace{1cm}\leq \prcondb{\textsf{Subord}_{H^{\frac{-1}{\alpha - 1}}r\lambda^{2\bE}} \geq H^{\frac{-\alpha}{\alpha - 1}}  r^{\alpha} \lambda^{1-\bf}}{\text{no jumps greater than } \kappa^{-1} H^{\frac{-\alpha}{\alpha - 1}} r^{\alpha} \lambda^{1-\bf-\bg} }{} \\
&\hspace{1cm}\leq \prcondb{\textsf{Subord}_{1} \geq \kappa\lambda^{1-\bf-2\bE\alpha}}{ \text{no jumps greater than } 1}{}.
\end{align*}
By the arguments of Lemma \ref{lem:osc}, it follows that there exists $\theta > 0$ such that $\Eb{e^{\theta \textsf{ Subord}_1}}<\infty$ when conditioned to have no jumps greater than $1$, so, as before, the latter probability can be bounded by $Ce^{-c\lambda^{1-\bf-2\bE\alpha}}$.

Combining these, we see that 
\begin{align*}
\prb{S_u \geq r^{\alpha} \lambda^{1-\bf}} \leq  Ce^{-c \lambda^{1-\bf-\bg-\alpha \bd}} + C(e^{-c\lambda^{\bd(\alpha - 1)}} + e^{- c\lambda^{2\bE}}) + Ce^{-c\lambda^{1-\bf-2\bE\alpha}}.
\end{align*}
The result follows on taking a union bound.
\end{proof}
\end{prop}

We are now able to prove Proposition \ref{prop:prob bound sup UB}.

\begin{proof}[Proof of Proposition \ref{prop:prob bound sup UB}]
Note that, on the events $\{|T| \leq \lambda^{\bf}\}$ and $\{S_u \leq r^{\alpha} \lambda^{1-\bf} \forall u \in T\}$, we have that
\begin{align*}
\nu (B(\rho, r)) \leq \sum_{u \in T} S_u \leq |T| \sup_{u \in T} S_u \leq \lambda^{\bf} r^{\alpha} \lambda^{1-\bf} = r^{\alpha} \lambda.
\end{align*}
Hence, by combining the results of Propositions \ref{prop:prog bound} and \ref{prop:iteration small terms prob bound}, we see that
\begin{align}\label{eqn:overall prob bound}
\begin{split}
\prb{\nu (B(\rho, r))  \geq r^{\alpha} \lambda} &\leq \prb{|T| \geq \lambda^{\bf} \text{ or } S_u \geq r^{\alpha} \lambda^{1-\bf} \text{ for some } u \in T} \\
&\leq C\lambda^{\bf} \Big[ e^{-c \lambda^{1-\bf-\bg-\alpha \bd}} + (e^{-c\lambda^{\bd(\alpha - 1)}} + e^{- c\lambda^{2\bE}}) + e^{-c\lambda^{1-\bf-2\bE\alpha}} \Big] + Ce^{-\lambda^{\bf}} \\
&C\lambda^{\frac{\alpha - 1}{4\alpha - 3}} e^{-c\lambda^{\frac{\alpha - 1}{4\alpha - 3}}}.
\end{split}
\end{align}
\end{proof}

\subsection{Supremal Lower Bounds}\label{sctn:sup vol LBs}
In this section, we prove (\ref{eqn:glob sup LB}) and (\ref{eqn:loc sup LB}). We start by proving a probabilistic bound. The proof relies on using the relation (\ref{eqn:d osc compare}) to compare volume fluctuations with stable \Levy oscillations.

\begin{prop}\label{prop:sup LB prob}
There exist constants $c, C \in (0, \infty)$ such that for all $r<1$ and all $\lambda > 1$, \[
\prb{\nu (B(p(U), \frac{1}{2} r)) \geq r^{\alpha} \lambda} \geq C e^{-c \lambda}.
\]
\begin{proof}
As explained in Section \ref{sctn:inf LBs}, we know that
\[
\{ \osc _{[ p(U), p(U) + r^{\alpha} \lambda]} \X \leq r \} \subset \{\nu (B(p(U), r)) \geq r^{\alpha} \lambda\}.
\]
It follows from the scaling relation $p_t(x) = t^{\frac{-1}{\alpha}}p_1(xt^{\frac{-1}{\alpha}})$ that $\frac{p_{1 - r^{\alpha} \lambda_r}(r)}{p_1(0)} \wedge \frac{p_{1 - r^{\alpha} \lambda_r}(-r)}{p_1(0)} \rightarrow 1$ as $r \downarrow 0$ whenever $\lambda_r = o(r^{-\alpha})$. Consequently, by applying the Vervaat transform and the absolute continuity relation, we have for all sufficiently small values of $r$ that
\begin{align*}
\prb{\nu (B(p(U), r)) \geq r^{\alpha} \lambda} \geq \prb{ \osc _{[ p(U), p(U) + r^{\alpha}\lambda]} \X \leq r} &\geq \Big\{ \frac{p_{1 - r^{\alpha} \lambda}(r)}{p_1(0)} \wedge \frac{p_{1 - r^{\alpha} \lambda}(-r)}{p_1(0)} \Big\} \prb{ \osc _{[ 0, r^{\alpha}\lambda]} X \leq r} \\
&\geq \frac{1}{2} \prb{ T^0_{[-1,1]} > 2^{\alpha} \lambda},
\end{align*}
where $T^x_I$ denotes the exit time of $X$ from the interval $I$, conditioned on $X_0 = x$.

It follows from the discussion below Theorem 2 of \cite{Bertoin2Side} that $\prb{ T^0_{[-1,1]} > 2^{\alpha} \lambda} \sim c_1 e^{-c_2 \lambda}$, for some deterministic constants $c_1, c_2$. The proposition follows.
\end{proof}
\end{prop}

We cannot directly use Proposition \ref{prop:sup LB prob} to prove the lower supremal bounds since we do not have the necessary independence to immediately apply the second Borel-Cantelli lemma. However, we can achieve this by performing a spinal decomposition, and considering volumes in different fragments, which are independent of each other. To do this, we will use a spinal decomposition to a uniform point, detailed below. The advantage of this over the Williams' decomposition in this case is that it allows us to control the masses of individual fragments more explicitly.

\subsubsection{Spinal Decomposition from the Root to a Uniform Point}\label{sctn:spinal decomp}

In \cite{HPWSpinPart}, it was shown that if we define the spine of a stable \Levy tree $\Ta$ to be the unique path from the root to a uniform point, then $\Ta$ can be broken along this spine and that the resulting fragments form a collection of smaller \Levy trees. This gives a similar decomposition result for looptrees.

We define the decomposition formally as follows. Let $U \sim \textsf{Uniform}([0,1])$, so that $p(U)$ is a uniformly chosen vertex in $\L_{\alpha}$, and let $\rho$ be its root. We say that the \textit{loopspine} from $\rho$ to $p(U)$, denoted $S_{U}$, is the closure of the set of loops corresponding to ancestors of $U$. To form the \textit{fine spinal decomposition}, first let $(L^o_i)_{i =1}^{\infty}$ be the connected components of $\L_{\alpha} \setminus S_{U}$, and then for each $i \in \N$ let $L_i$ be the closure of $L^o_i$ in $\L_{\alpha}$. Then almost surely, each $L_i$ can be written in the form $L^o_i \overset{.}{\cup} \rho_i$ for some $\rho_i \in \L_{\alpha} \setminus L^o_i$. Note that by uniform rerooting invariance, we can also replace the root with an independent uniform point in $\La$.

If the fragment $L_i$ has mass $\alpha_i$, define a metric $d_i$ and a measure $\nu_i$ on $L_i$ by
\[
d_i = {\alpha}_i^{\frac{-1}{\alpha}} d|_{L_i}, \ \ \ \ \ \ \nu_i = \frac{\nu( \cdot \cap L_i)}{\alpha_i}.
\]
Additionally let $p(U_i)$ be a vertex in $L_i$ chosen uniformly according to $\nu_i$. We then have the following result, which is a consequence of \cite[Corollary 10]{HPWSpinPart}, which gives the corresponding result for \Levy trees.

\begin{theorem}\label{thm:spinal decomp}
$\{ (L_i, d_{i}, \nu_i, \rho_i, p(U_i))\}_{i \in \N}$ is a collection of independent copies of $(\L_{\alpha}, d, \rho, \nu, p(U))$. Moreover, the entire family is independent of $(\alpha_i)_{i \in \N}$, which has a Poisson-Dirichlet $(\alpha^{-1}, 1 - \alpha^{-1})$ distribution.
\end{theorem}

\subsubsection{Main Argument for Supremal Lower Bound}

Using Theorem \ref{thm:spinal decomp}, we can construct an argument as follows. First take some $\epsilon>0$ with $0 < \epsilon \ll 1$. Given $r \in (0,1)$, and $\lambda_r$ a decreasing function of $r$ such that $\frac{\lambda_{r}}{\lambda_{2r}} \rightarrow 1$ as $r \downarrow 0$, define the interval $J_r = [r^{-1} \lambda_r^{\frac{-(1+\epsilon)}{\alpha}}, \frac{3}{2} r^{-1} \lambda_r^{\frac{-(1+\epsilon)}{\alpha}}]$. It is easy to verify that for all sufficiently small $r$, the intervals $J_r$ and $J_{2r}$ are disjoint.

Our strategy is as follows. We use the spinal decomposition of Section \ref{sctn:spinal decomp}, between $p(U)$ and an independent uniform point $p(V)$. Recall the GEM distribution introduced there, that gives a size biased representation $(M_1, M_2, \ldots )$ of the Poisson-Dirichlet distribution. Letting $I'_r$ denote the segment of loopspine that intersects $B(p(U), r)$ (analagously to $I_R$ defined in Section \ref{sctn:sup UB actual arg}), there is probability of order at least $\lambda^{-(\frac{1}{\alpha} + \epsilon)}$ that there is a $n \in J_r$ such that the sublooptree with Poisson-Dirichlet mass given by the GEM random variable $M_n$ is grafted to the loopspine at a point in $I'_{\frac{r}{2}}$. Say this sublooptree is $L_{i,r}$, with root $\rho_i$ being the point at which it is grafted to the loopspine. The mass of the ball $B(p(U), r)$ is then lower bounded by the mass of $B(\rho_i, \frac{1}{2}r) \cap L_{i,r}$. We can then rescale the looptree $L_{i,r}$, and the corresponding unit ball, to compute that this mass is at least $r^{\alpha} \lambda$ with at least polynomial probability. We repeat this argument along the sequence $r_n = 2^{-n}$. Since the corresponding intervals $J_{r_n}$ are disjoint (provided we start at a sufficiently large value of $n$), and the rescaled looptrees from the spinal decomposition of Section \ref{sctn:spinal decomp} are independent, we obtain the necessary independence to apply the second Borel-Cantelli Lemma.

\begin{proof}[Proof of supremal lower bound in Theorem \ref{thm:main local}]
Let $L= \sum_{U \wedge V \prec t \preceq U} \Delta_t + \sum_{U \wedge V \prec t \preceq V} \Delta_t + \delta_{U \wedge V}(x_{U \wedge V}^U, x_{U \wedge V}^V)$ be the length of the loopspine, and let $N_r$ be the total number of sublooptrees in the spinal decomposition that are incident to the loopspine at a point in $I'_{\frac{r}{2}}$ and have mass corresponding to a GEM index in $J_r$. Then, conditional on $L=l$, $N_r$ stochastically dominates a random variable that is Binomial($\lfloor \frac{1}{2} r^{-1} \lambda_r^{\frac{-(1+\epsilon)}{\alpha}} \rfloor, r l^{-1}$). Hence, the probability that this number is non-zero is at least of order $l^{-1} \lambda_r^{\frac{-(1+\epsilon)}{\alpha}}$.

Conditional on $\{N_r \geq 1\}$, let $n_r$ be an index in $J_r$ with corresponding sublooptree $L_r$ that is incident to the loopspine at a point in $I_{(\frac{r}{2})}'$. Note that $\nu (L_r)$ stochastically dominates the Poisson-Dirichlet GEM weight $M_{k_r}$, where $k_r = \frac{3}{2} r^{-1} \lambda_r^{\frac{-(1+\epsilon)}{\alpha}}$, and hence we have by Lemma \ref{lem:Gem paley zig} that there exists $c_p > 0$ such that
\begin{align*}
\prb{\nu (L_r) \geq \frac{1}{2} r^{\alpha} \lambda_r^{1+\epsilon}} \geq \prb{M_{k_r} \geq \frac{1}{2} r^{\alpha} \lambda_r^{1+\epsilon}} \geq c_p.
\end{align*}

Conditional on there being such a sublooptree $L_{r}$, say of mass $m \geq \frac{1}{2} r^{\alpha} \lambda_r^{1+\epsilon}$, we know that
\[
\prb{\nu (B(\rho_i, \frac{1}{2}r) \cap L_{i,r}) \geq r^{\alpha} \lambda_r} = \prb{\nu (B(\rho, \frac{1}{2}rm^{\frac{-1}{\alpha}})) \geq m^{-1} r^{\alpha} \lambda_r} \geq C e^{-c \lambda_{r_n}}
\]
by Proposition \ref{prop:sup LB prob} (note in particular that $m^{-1}r^{\alpha} \lambda_r \leq 2 \lambda_r^{ - \epsilon} \rightarrow 0$ as $r \downarrow 0$ so it is fine to apply the result here).

Hence, letting $A_r$ be the event that there exists a sublooptree incident to the loopspine at a point in $I'_{\frac{r}{2}}$ with GEM index $n_r \in J_r$, and such that the ball of radius $\frac{1}{2}r$ in this sublooptree has mass at least $r^{\alpha}\lambda_r$, we deduce that $\prb{A_r} \geq C l^{-1} \lambda_{r}^{\frac{-(1 + \epsilon)}{\alpha}} e^{-c \lambda_r}$.

Now, letting $r_n = 2^{-n}$, we have that there exists a finite $N$ such that the intervals $J_{r_n}$ and $J_{r_m}$ are disjoint whenever $m, n \geq N$, and hence since each sublooptree is distributed uniformly around the perimeter of the loopspine independently of the others, then the events that there exist sublooptrees with GEM index in $J_{r_n}$ (respectively $J_{r_m}$) within distance $\frac{1}{2}r_n$ (respectively $\frac{1}{2}r_m$) from the root are independent events. Moreover, if the sublooptrees described in the events $A_{r_n}$ and $A_{r_m}$ exist, then they are independent once rescaled by Proposition \ref{thm:spinal decomp}. Thus the only dependence between the events $A_{r_n}$ and $A_{r_m}$ is in whether these sublooptrees have masses greater than $c'r_n^{\alpha} \lambda_{r_n}^{1 + \epsilon}$ (respectively $c'r_m^{\alpha} \lambda_{r_m}^{1 + \epsilon}$), but here the only dependence is that all the Poisson-Dirichlet masses must sum to $1$, and hence in actual fact $\prcondb{A_{r_n}}{A_{r_m}^c \text{ for all } N \leq m < n}{} \geq \prb{A_{r_n}}$ by an application of Lemma \ref{lem:Gem correlations}.

Hence,
\[
\sum_{n=N}^{\infty} \prcondb{A_{r_n}}{A_{r_m}^c \text{ for all } N \leq m < n}{} \geq \sum_{n=N}^{\infty} \prb{A_{r_n}} \geq \sum_{n=N}^{\infty} C l^{-1} \lambda_{r_n}^{\frac{-(1 + \epsilon)}{\alpha}} e^{-c \lambda},
\]
so setting $\lambda_r = \frac{C_2^{-1}}{2}(\log \log r^{-1})$ we see that $\prcondb{A_{r_n} \text{ i.o.}}{\text{Length}(S_{\sigma}) = l}{} = 1$. Since $L$ is almost surely finite, we can integrate over possible values of $l$ to deduce that $\prb{A_{r_n} \text{ i.o.}} = 1$.

The result at (\ref{eqn:loc sup LB}) follows by applying Fubini's theorem similarly to the previous extremal volume bounds.
\end{proof}

\begin{figure}[h]
\centering
\includegraphics[width=16cm]{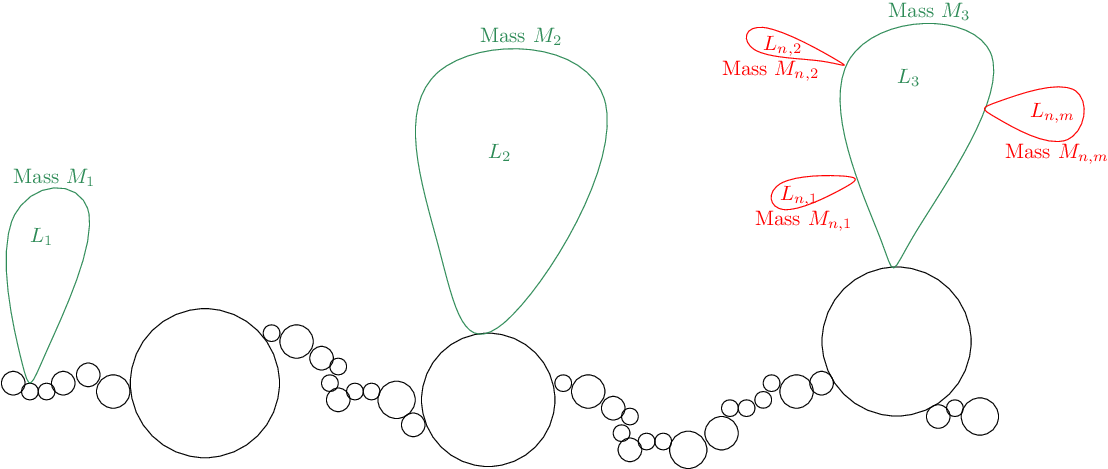}
\caption{Illustration of the double decomposition (simplified picture).}\label{fig:doubledecomp1}
\end{figure}

The proof of the global bound given in (\ref{eqn:glob sup LB}) uses a similar decomposition approach, but this time we perform two subsequent spinal decompositions. This is illustrated in Figure \ref{fig:doubledecomp1}. Firstly, let $(M_1, M_2, \ldots)$ denote the GEM masses obtained on performing a first spinal decomposition of $\La$. Then, for each of the resulting fragments $(L_1, L_2, \ldots)$, rescale to obtain a sequence of independent stable looptrees $(\La^1, \La^2, \ldots)$, each with mass $1$. For each $n \in \N$, we perform a further spinal decomposition of $\La^n$ and denote the resulting GEM masses by $\{M_{n,1}, M_{n,2}, \ldots \}$, and corresponding looptrees by $\{L_{n,1}, L_{n,2}, \ldots \}$, and by $\{\L_{\alpha}^{n,1}, \La^{n,2}, \ldots \}$ after rescaling again to have mass $1$. We take $r_n = 2^{-n}, R_n = M_n^{\frac{-1}{\alpha}}r_n$, $\lambda_n = C^*\log r_n^{-1}$ and $\lambda_n' = C^*\log R_n^{-1}$, where $C^*$ is a specific constant to be specified later. We also define the events:
\begin{align*}
B_n &= \{r_n^{\alpha} \leq M_n^{2} \},
&&C_{n,m} = \{ M_{n,m} \geq R_n^{\alpha} \lambda_n \}, \\
D_{n,m} &= \{ \nu (B(\rho_{n,m}, R_n) \cap L_{n,m}) \geq R_n^{\alpha}\lambda_n \}, 
&&A_{n,m} = C_{n,m} \cap D_{n,m}.
\end{align*}
Also set $N_n = 2^{\frac{-1}{\alpha}} r_n^{-1} (\log r_n^{-1})^{\frac{-1}{\alpha}}$, and define the event
\[
A_n = B_n \cap \Big( \bigcup_{m=1}^{N_n} A_{n,m} \Big).
\]
The key point is to observe that $A_n \subset \{ \sup_{u \in \La} \nu (B(u,r_n)) \geq r_n^{\alpha} \lambda_n \}$, so it is sufficient to show that $\prb{A_n \text{ i.o.}}=1$. The next lemma gives a means to overcome the dependencies between the GEM masses and apply the second Borel-Cantelli Lemma. It should be intuitively clear, but we give a proof for completeness.

\begin{lem}\label{lem:prob indep global sup LB}
Let $A_n, B_n, A_{n,m}, C_{n,m}, D_{n,m}, N_n$ be as above. Then 
\begin{enumerate}[(i)]
\item $ \prcondb{A_n}{A_m^c \forall m < n}{} \geq \prb{A_n}$,
\item $\prcondb{A_{n,m}}{A_{n,l}^c \forall l < m}{} \geq \prb{A_{n,m}}$.
\end{enumerate}
\begin{proof}
First, note that since the individual looptrees in the spinal decomposition are independent of each other and of their original masses once rescaled, we can make the following observations:
\begin{itemize}
\item $A_{n,m}$ is independent of $B_l$ for all $m$ and all $l \leq n$,
\item $B_n$ is independent of $A_{l,m}$ for all $m$ and all $l \leq n$,
\item Conditional on $C_{n,l}$, $D_{n,l}$ is independent of $D_{n,k}$ for all $k<l$.
\end{itemize}
Figure \ref{fig:doubledecomp1} may be helpful to keep track of the dependencies. In fact, the only dependence between these events is of the form described by Lemma \ref{lem:Gem correlations}.


We start by proving $(i)$. First note that by the first independence stated above, we have that
\begin{align}\label{eqn:indep correlation split}
\begin{split}
\prcondb{A_n}{A_m^c \forall m < n}{} &= \prcondb{B_n \cap \Big( \bigcup_{m=1}^{N_n} A_{n,m} \Big)}{A_m^c \forall m < n}{} \\
&= \prcondb{B_n}{A_m^c \forall m < n}{}\prcondb{\bigcup_{m=1}^{N_n} A_{n,m}}{A_m^c \forall m < n}{} \\
&= \prcondb{B_n}{A_m^c \forall m < n}{}\prcondb{B_n}{A_m^c \forall m < n}{}\prcondb{\bigcup_{m=1}^{N_n} A_{n,m}}{A_m^c \forall m < n}{}.
\end{split}
\end{align}
We focus on the first term in the final line above. By the second independence stated above, we have that
\begin{align*}
\prcondb{B_n}{A_l^c \forall l < n}{} &= \prcondb{B_n}{( \cup_m A_{l,m})^c \sqcup (( \cup_m A_{l,m}) \cap B_l^c) \ \forall \ l < n}{} \\
&= \sum_{\omega \in \{0,1\}^{n-1}}\prcondb{B_n}{E(\omega)}{}\prcondb{E(\omega)}{( \cup_m A_{l,m})^c \sqcup (( \cup_m A_{l,m}) \cap B_l^c) \ \forall \ l < n}{} \\
&= \sum_{\omega \in \{0,1\}^{n-1}}\prcondb{B_n}{E'(\omega)}{}\prcondb{E(\omega)}{( \cup_m A_{l,m})^c \sqcup (( \cup_m A_{l,m}) \cap B_l^c) \ \forall \ l < n}{},
\end{align*}
where for $\omega \in \{0,1\}^{n-1}$:
\begin{align*}
E(\omega) &= \Big( \bigcup_{l: \omega_l = 1} (\cup_m A_{l,m})^c \Big) \cap \Big( \bigcup_{l: \omega_l = 0} (( \cup_m A_{l,m}) \cap B_l^c) \Big), \hspace{5mm} E'(\omega) = \Big( \bigcup_{l: \omega_l = 0} (( \cup_m A_{l,m}) \cap B_l^c) \Big).
\end{align*}
Since $B_n$ is independent of $\cup_m A_{l,m}$, we can apply Lemma \ref{lem:Gem correlations} to deduce that $\prb{B_n} \geq \prcondb{B_n}{E'(\omega)}{}$ for all $\omega$. Substituting this into the final line, we obtain
\begin{align*}
\prcondb{B_n}{A_l^c \forall l < n}{} &\geq \sum_{\omega \in \{0,1\}^{n-1}}\prb{B_n}\prcondb{E(\omega)}{( \cup_m A_{l,m})^c \sqcup (( \cup_m A_{l,m}) \cap B_l^c) \ \forall \ l < n}{} = \prb{B_n}.
\end{align*}
We can use the same kind of expansion and apply Lemma \ref{lem:Gem correlations} to show that
\[
\prcondb{\bigcup_{m=1}^{N_n} A_{n,m}}{A_m^c \forall m < n}{} \geq \prb{\bigcup_{m=1}^{N_n} A_{n,m}}.
\]
Point $(i)$ then follows from the final line of (\ref{eqn:indep correlation split}). The proof of the point $(ii)$ is almost identical, so we omit it.
\end{proof}

\end{lem}

Armed with the lemma, we prove the global infimum upper bound as follows.

\begin{proof}[Proof of supremal lower bound in Theorem \ref{thm:main global}]
Recall from Theorem \ref{thm:spinal decomp} that the rescaled looptree $\La^n$ is independent of $M_n$. It follows that the event $B_n$ is independent of $\cup_{m=1}^{N_n} A_{n,m}$, and hence
\begin{align}\label{eqn:glob sup LB indep prob}
\prb{A_n} = \prb{B_n} \prb{\cup_{m=1}^{N_n} A_{n,m}}.
\end{align}
We bound each of these terms separately. Firstly, by Lemma \ref{lem:Gem paley zig}, we have that there exists $\tilde{c}_p > 0$ such that
\begin{align*}
\prb{M_{k} \geq \frac{1}{2} k^{-\alpha}} \geq \tilde{c}_p
\end{align*}
for all $k \geq 1$.
Recalling that $r_n = 2^{-n}$, we see that
\begin{align}\label{eqn:PZ bound}
\prb{B_n} = \prb{M_n \geq r_n^{\frac{\alpha}{2}}} \geq \prb{M_n \geq \frac{1}{2}n^{-\alpha}} \geq \tilde{c}_p.
\end{align}
To bound the second term in (\ref{eqn:glob sup LB indep prob}), we apply point $(ii)$ of Lemma \ref{lem:Gem correlations}, which implies that 
\begin{align*}
\prb{\bigcup_{m=1}^{N_n} A_{n,m}} \geq 1 - \prod_{m=1}^{N_n} (1 - \prb{A_{n,m})}.
\end{align*}
Recalling that $N_n = \lfloor 2^{\frac{-(2\alpha +1)}{\alpha}} r_n^{-1} (\log r_n^{-1})^{\frac{-1}{\alpha}} \rfloor \leq 2^{\frac{-(\alpha +1)}{\alpha}} r_n^{-1} (\log r_n^{-1})^{\frac{-1}{\alpha}}$, we again apply (\ref{eqn:PZ bound}) to deduce that 
\begin{align*}
\prcondb{C_{n,m}}{B_n}{} = \prb{C_{n,m}} = \prb{M_{n,m} \geq R_n^{\alpha} (\log r_n^{-1})} \geq \prb{M_{n,m} \geq \frac{1}{2} m^{-\alpha} } > c_p
\end{align*}
whenever $m < N_n$. To conclude, note that conditional on $C_{n,m}$, we have that $M_n^{-1} R^{\alpha}_n(\log r_n^{-1}) \leq 1$ and hence we can apply Proposition \ref{prop:sup LB prob} to deduce that 
\begin{align*}
\prcondb{D_{n,m}}{C_{n,m}, B_n}{} \geq \prb{\nu(B(\rho, M_n^{\frac{-1}{\alpha}}R_n)) \geq M_n^{-1} R_n(\log r_n^{-1})} \geq C e^{-\hat{c} \lambda_n}.
\end{align*}
Here we are specifically taking $\hat{c}$ to be the constant in the exponent of Proposition \ref{prop:sup LB prob}. Combining, we see that
\begin{align*}
\prb{\bigcup_{m=1}^{N_n} A_{n,m}} \geq 1 - \prod_{m=1}^{N_n} (1 - \prb{A_{n,m})} &\geq 1 - (1-C e^{-2\hat{c} \lambda_n'})^{N_n} 
\geq 1- \exp\{ 2^{\frac{-(\alpha +1)}{\alpha}} r_n^{-1} (\log r_n^{-1})^{\frac{-1}{\alpha}} C r_n^{2\hat{c} C^*} \}
\end{align*}
Hence, by choosing $C^* > (2 \hat{c})^{-1}$, we see that $\prb{\cup_{m=1}^{N_n} A_{n,m}} \rightarrow 1$ as $n \rightarrow \infty$, and in particular that we can lower bound it by a non-negative constant uniformly in $n$. Combining this with (\ref{eqn:glob sup LB indep prob}) and (\ref{eqn:PZ bound}), we see that there exists a constant $c>0$ such that $\prb{A_n} \geq c$ for all $n \geq 1$. It then follows from Lemma \ref{lem:prob indep global sup LB} and Borel-Cantelli that $\prb{A_n \text{ i.o.}} = 1$.

The conclusion follows since on the event $D_{n,m}$, we can rescale the ball $B(\rho_{n,m}, R_n) \cap L_{n,m}$ back to its original size in the original looptree to obtain a ball of radius $r_n$ with volume at least $r_n^{\alpha} 2\lambda'_n$. Moreover, on the event $B_n$ we also have that $\lambda_n \leq 2\lambda_n'$, so this volume is actually lower bounded by $r_n^{\alpha} \lambda_n = r_n^{\alpha} \log r_n^{-1}$.
\end{proof}

\subsection{Infimal Upper Bounds}\label{sctn:inf UBs}

We now prove (\ref{eqn:glob inf UB}) and (\ref{eqn:loc inf UB}). The method we use to prove upper bounds on infimal extrema is a simpler version of that used in Section \ref{sctn:sup vol UBs} based on the Williams' decomposition. We can again control the masses of fragments in the decomposition by comparison with an $\alpha^{-1}$-stable subordinator. In this case however, we do not need to worry about reiterating around larger fragments since the presence of such fragments is a rare event and thus should not affect the infimal behaviour of the subordinator.

Let $H$ be the height of the spine in the corresponding tree $\Ta$. As in Section \ref{sctn:sup vol UBs}, we start by rescaling $\La$ by $H$ to form the looptree $(\La^1, d^1, \rho^1, \nu^1)$, which now has mass $H^{\frac{-\alpha}{\alpha - 1}}$ and has a corresponding underlying stable tree that has height $1$. Note that
\[
\{ \nu (B(\rho, r)) \leq r^{\alpha} \lambda^{-1} \} = \{\nu^1 (B^1(\rho^1, rH^{\frac{-1}{\alpha - 1}})) \leq R^{\alpha} \lambda^{-1}\}.
\]
where again $R = rH^{\frac{-1}{\alpha - 1}}$. As explained in the Lemma \ref{lem:iterative progeny Poisson compare}, and using the notation we introduced there, it follows from properties of the \Ito excursion measure that $\nu^1 (B^1(\rho^1, R))$ is stochastically dominated by $Y(|I_{R}|)$, where $Y$ is an $\alpha^{-1}$-stable subordinator, and $I_{R}$ denotes the length of W-loopspine that intersects $B^1(\rho^1, R)$. A jump of $Y$ of size $\Delta$ at a time $t$ corresponds to a sublooptree coded by an \Ito excursion of lifetime equal to $\Delta$, and grafted to the W-loopspine at a point that informally is at a clockwise distance $t$ ``through" $I_R$. Moreover, since we have rescaled the looptree to have tree height $1$, there is no constraint on its total mass, and therefore no dependence between different jumps of $Y$.

For technical reasons we will in fact model this by two independent $\alpha^{-1}$-stable subordinators, $Y^{(l)}$ and $Y^{(r)}$, corresponding to the left and right sides of the W-loopspine respectively. We set $Y = Y^{(l)} + Y^{(r)}$.

The comparison relies on the following result, which gives the limiting behaviour of the infimum of an $\alpha^{-1}$-stable \Levy subordinator.

\begin{theorem}\cite[Section III.4, Theorem 11]{BertoinLevy}.\label{thm:Bert subord liminf}
Let $(W_t)_{t \geq 0}$ be an $\alpha^{-1}$-stable \Levy subordinator. Then, almost surely,
\[
\liminf_{t \downarrow 0^+} \frac{W_t}{t^{\alpha}(\log \log t^{-1})^{- (\alpha - 1)}} = \alpha^{-1} (1 - \alpha^{-1})^{\alpha - 1}.
\]
\end{theorem}
To deduce a similar result for $(Y_t)_{t \geq 0}$ in place of $(W_t)_{t \geq 0}$, note that the only difference between the two subordinators is the constant in the \Levy measure. Hence we have the same result for $(Y_t)_{t \geq 0}$, but just with a different constant on the right hand side. We will denote this constant by $c_{\alpha}$.

\begin{proof}[Proof of local infimal upper bound in Theorem \ref{thm:main local}]
Set $f(t) = t^{\alpha}(\log \log t^{-1})^{- (\alpha - 1)}$ for $t > 0$. By Theorem \ref{thm:Bert subord liminf}, there almost surely exists a sequence $(r_n)_{n \geq 1}$ with $r_n \downarrow 0$ such that 
\[
Y(3r_nH^{\frac{-1}{\alpha - 1}}) \leq (c_{\alpha} + 1)f(3r_n H^{\frac{-1}{\alpha - 1}})
\]
for all $n$. Since $f(3r_n H^{\frac{-1}{\alpha - 1}}) \leq 2 \cdot 3^{\alpha} r_n^{\alpha} H^{\frac{-\alpha}{\alpha - 1}} (\log \log r_n^{-1})^{-(\alpha - 1)}$ whenever $r_n \leq H^{\frac{-1}{\alpha - 1}}$, we can extract a subsequence if necessary so that 
\[
Y(3r_nH^{\frac{-1}{\alpha - 1}}) \leq 2 \cdot 3^{\alpha} (c_{\alpha} + 1)r_n^{\alpha} H^{\frac{-\alpha}{\alpha - 1}} (\log \log r_n^{-1})^{-(\alpha - 1)}
\]
and also $r_{n+1} < \frac{1}{2}r_n$ for all $n \geq 1$. Set $R_n = r_n H^{\frac{-1}{\alpha - 1}}$.

Note that since the process $Y$ depends only on the total length of the W-loopspine, and not on its microscopic structure, it follows from Lemma \ref{lem:segment length bound} that there exists a constant $C_p > 0$ such that
\begin{equation*}\label{eqn:I rnh bound}
\prb{|I_{R_n}| \leq 3R_n} \geq C_p
\end{equation*}

for all $n$. More specifically, we let $A_n$ be the event described by taking $\lambda = 1$ in the proof of Lemma \ref{lem:segment length bound} that ensures that $|I_{R_n}| \leq 3R_n$, consisting of the three subevents: \\
${(i)}_n$ There exists a good loop in the W-loopspine with total length in $[4R_n, 8R_n]$. \\
${(ii)}_n$ There are no goodish loops in the W-loopspine occurring between the root and the first good one. \\
${(iii)}_n$ The sum of the lengths of the smaller loops up until the first good loop is upper bounded by $R_n$.

The proof of Lemma \ref{lem:segment length bound} ensures that $\prb{A_n} \geq C_p$ for all $n$, but to apply the second Borel-Cantelli Lemma we need to lower bound $\prcondb{A_n}{A_m^c \forall m < n}{}$ instead. To do this, note that conditional on $A_m^c \forall m < n$:
\begin{itemize}
\item The probability of the event described in $(i)_n$ is unaffected by the events of $A_m$ for $m < n$, since the sets $[4R_n, 8R_n]$ are disjoint for different $n$ and therefore can be viewed as independent thinned Poisson processes along the W-spine of the tree.
\item Conditional on $(i)^c_{m}$ occurring for all $m<n$, the probability that there is only one goodish loop before the first good one at level $n-1$ is lower bounded by $\prcond{\textsf{Geo}(\frac{1}{2}) = 1}{ \textsf{Geo}(\frac{1}{2}) \neq 0}{} = \frac{1}{2}$.
\item Conditional on there only being one such goodish loop at level $n-1$, the probability that the good loop at level $n$ occurs before the goodish loop at level $n-1$ is at least $\frac{1}{2}$. If this occurs, then the probability of the events in $(ii)_n$ and $(iii)_n$ is unaffected.
\end{itemize}
It follows that 
\[
\prcondb{A_n}{A_m^c \forall m < n}{} \geq \frac{1}{4}C_p
\]
for all $n$, and therefore $\prb{A_n \text{ i.o.}}=1$ by the second Borel-Cantelli Lemma.

To conclude, note that on the event $A_n$ we have
\begin{align*}
\nu^1 (B^1(\rho^1, R_n)) \leq Y(3R_n) \leq 2 \cdot 3^{\alpha} (c_{\alpha} + 1)R_n^{\alpha} (\log \log r_n^{-1})^{-(\alpha - 1)},
\end{align*}
and hence scaling back to the original looptree we see that
\[
\nu (B(\rho, r_n)) \leq 3^{\alpha}(c_{\alpha} + 1)r_n^{\alpha} (\log \log r_n^{-1})^{-(\alpha - 1)}.
\]
for all sufficiently large $n$. This proves the local result (\ref{eqn:loc inf UB}).
\end{proof}
%

To prove the global bound, we perform two subsequent spinal decompositions of $\La$, exactly as illustrated in Figure \ref{fig:doubledecomp1} in the previous section. Recall from there that we let $(M_1, M_2, \ldots)$ denote the GEM masses obtained on performing a first spinal decomposition of $\La$, as described in Section \ref{sctn:spinal decomp}. Then, for each of the resulting fragments $(L_1, L_2, \ldots)$, rescale to obtain a sequence of independent stable looptrees $(\La^1, \La^2, \ldots)$, each with mass $1$. For each $n \in \N$, we perform a further spinal decomposition of $\La^n$ and denote the resulting GEM masses by $\{M_{n,1}, M_{n,2}, \ldots \}$, and corresponding looptrees by $\{L_{n,1}, L_{n,2}, \ldots \}$, and by $\{\L_{\alpha}^{n,1}, \La^{n,2}, \ldots \}$ after rescaling. We also let $U_{n,m}$ denote a point chosen uniformly in $L_{n,m}$ according to the natural volume measure. We take $r_n = 2^{-n}, R_n = M_n^{\frac{-1}{\alpha}}r_n$, $\lambda_n = (C^*\log r_n^{-1})^{\alpha - 1}$ and $\lambda_n' = (C^*\log R_n^{-1})^{\alpha - 1}$, where $C^*$ is a constant to be specified later. We also define the events:
\begin{align*}
B_n &= \{r_n^{\alpha} \leq M_n^{2} \}, &&C_{n,m} = \{ d_{\La^{n}}(\rho_{m,n}, U_{m,n}) \geq R_n \}, \\
D_{n,m} &= \{ \nu_{\La^{n}} (B(U_{n,m}, R_n) \cap L_{n,m}) \leq R_n^{\alpha} {\lambda_n}^{-1} \}, &&A_{n,m} = C_{n,m} \cap D_{n,m}.
\end{align*}
We also set $N_n = r_n^{\frac{-1}{2}}$. We then define the event
\[
A_n = B_n \cap \Big( \bigcup_{m=1}^{N_n} A_{n,m} \Big).
\]
The key point is to observe that $A_n \subset \{ \sup_{u \in \La} \nu (B(u,r_n)) \leq r_n^{\alpha} {\lambda_n}^{-1} \}$, and hence it is sufficient to only show that $\prb{A_n \text{ i.o.}}=1$. Similarly to the previous section, the next lemma gives us a means to overcome the dependencies between the GEM masses and apply the second Borel-Cantelli Lemma. Its proof is almost identical to that of Lemma \ref{lem:prob indep global sup LB}, so is omitted.

\begin{lem}\label{lem:prob indep global inf UB}
Let $A_n, B_n, A_{n,m}, C_{n,m}, D_{n,m}, N_n$ be as above. Then 
\begin{enumerate}[(i)]
\item $ \prcondb{A_n}{A_m^c \forall m < n}{} \geq \prb{A_n}$,
\item $\prcondb{A_{n,m}}{A_{n,l}^c \forall l < m}{} \geq \prb{A_{n,m}}$.
\end{enumerate}
\end{lem}

\begin{proof}[Proof of global infimal upper bound in Theorem \ref{thm:main global}.]
Now, note that it follows from \cite[Section III.4, Theorem 12]{BertoinLevy} and the local argument given above that
\begin{equation}\label{eqn:inf UB prob}
\prb{\nu (B(p(U), r) \leq r^{\alpha} \lambda^{-1}} \geq Ce^{-c\lambda^{\frac{1}{\alpha - 1}}}.
\end{equation}
We will apply this to prove that $\prb{A_n} \geq Ce^{-c\lambda^{\frac{1}{\alpha - 1}}}$ as well. Firstly, note that by Lemma \ref{lem:Gem paley zig} there exists a constant $c>0$ such that $\prb{B_n} > c$ for all $n$. Then, since the looptrees in the spinal decomposition are independent of their original masses after rescaling (see Theorem \ref{thm:spinal decomp}), it follows that $\bigcup_{m=1}^{N_n} A_{n,m}$ is independent of $B_n$. Next, we note that:
\begin{align*}
\prcondb{C_{n,m}}{B_n, m \leq r_n^{\frac{-1}{2}}}{} &= \prcondb{d_{\La^{n}}(\rho_{m,n}, U_{m,n}) \geq R_n}{B_n, m \leq r_n^{\frac{-1}{2}}}{} \\
&\geq \prcondb{d_{\La^{n}}(\rho_{m,n}, U_{m,n}) \geq r_n^{\frac{1}{2}}}{m \leq r_n^{\frac{-1}{2}}}{} \\
&\geq \prcondb{\nu_{\La}(L_{m}) \geq \frac{1}{2}r_n^{\frac{\alpha}{2}}}{m = r_n^{\frac{-1}{2}}}{} \prcondb{d_{\La}(\rho_{m}, U_{m}) \geq r_n^{\frac{1}{2}}}{\nu_{\La}(L_{m}) = \frac{1}{2}r_n^{\frac{\alpha}{2}}}{} \\
&\geq C,
\end{align*}
where $C>0$. The final line follows since by Lemma \ref{lem:Gem paley zig} the first term in the penultimate line above can be uniformly lower bounded by a constant, and the second term can also be uniformly lower bounded by a constant by scaling invariance.

To conclude, we note from (\ref{eqn:inf UB prob}) that $\prcondb{D_{n,m}}{C_{n,m}, B_n}{} \geq Ce^{-c{\lambda_n}^{\frac{1}{\alpha - 1}}}$ for all $n$, and all $m \leq N_n$. Combining these, we see that $\prb{A_{m,n}} \geq Ce^{-c{\lambda_n}^{\frac{1}{\alpha - 1}}}$. We therefore deduce from Lemma \ref{lem:prob indep global inf UB}(ii) that 
\begin{align*}
\prb{A_n} \geq \prb{B_n} \Big( 1 - \big( 1 - \prcondb{A_{n,m}}{B_n}{}\big)^{N_n}\Big) &\geq C' \Big( 1 - \big( 1 - Ce^{-c{\lambda_n}^{\frac{1}{\alpha - 1}}}\big)^{N_n} \Big) \\
&\geq C' \Big( 1 - \exp \{- N_n Ce^{-c{\lambda_n}^{\frac{1}{\alpha - 1}}} \} \Big) \\
&\geq C' \Big( 1 - \exp \{- r_n^{\frac{-1}{2}} Ce^{-cC^* \log r_n^{-1}} \} \Big).
\end{align*}
Choosing $C^*$ so that $C^* < \frac{1}{4} c^{-1}$, we obtain that
\begin{align*}
\prb{A_n} &\geq C' \Big( 1 - \exp \{- r_n^{\frac{-1}{4}} C \} \Big) \geq \frac{1}{2}C'
\end{align*}
for all sufficiently large $n$. Applying Lemma \ref{lem:prob indep global inf UB}(i) and the second Borel-Cantelli Lemma, we deduce that $\prb{A_n \text{ i.o.}}=1$, which implies (\ref{eqn:glob inf UB}).
\end{proof}

\subsection{Volume Convergence Results}\label{sctn:vol conv}
Here we briefly note a convergence result for $\nu (B(\rho, r))$. In a companion paper \cite{ArchInfiniteLooptrees}, we introduce the infinite stable looptree $\Lai$, which is defined from two stable \Levy processes rather than a \Levy excursion, arises as the local distributional limit of compact stable looptrees as their mass goes to infinity \cite[Theorem 1.1]{ArchInfiniteLooptrees}, and provides the machinery to prove the following result. For details see \cite[Section 6.3]{ArchInfiniteLooptrees}.

\begin{theorem}\cite[Theorem 6.6]{ArchInfiniteLooptrees}.\label{thm:main vol conv}
There exists a random process $(V_t)_{t \geq 0}: \Omega \rightarrow D([0, \infty), [0, \infty))$ such that the finite dimensional distributions of the process
\begin{align*}
\big( r^{-\alpha} \nu(\bar{B}(\rho, rt)) \big)_{t \geq 0}
\end{align*}
converge to those of $\big( V_t \big)_{t \geq 0}$ as $r \downarrow 0$, and $V_t$ denotes the volume of a closed ball of radius $t$ around the root in $\Lai$. Moreover, for any $p \in [1,\infty)$, we have that
\[
r^{-\alpha p} \Eb{\nu(\bar{B}(\rho, r))^p} {\rightarrow} \Eb{V_1^p}
\]
as $r \downarrow 0$, and $V_1$ is a $(0, \infty)$-valued random variable with all moments finite.
\end{theorem}

\section{Heat Kernel Estimates}\label{sctn:HK estimates compact looptrees}
Although we used the shortest distance metric to prove the volume results of Theorems \ref{thm:main local} and \ref{thm:main global}, the result of Lemma \ref{lem:dR compare} ensures that they also hold true with respect to the resistance metric $R$. This allows us to apply results of \cite{CroyHKFluctNonUn} to deduce the heat kernel bounds of Theorems \ref{thm:main HK global} and \ref{thm:main HK local}. Most of our results follow from a direct application of those of \cite{CroyHKFluctNonUn}, so we refer the reader there for further background.

To get some off-diagonal results, we need to verify the Chaining Condition (CC) of \cite[Section 4.2]{CroyHKFluctNonUn}.

\begin{defn}(Chaining Condition (CC), \cite[Section 4.2]{CroyHKFluctNonUn}).
A metric space $(X,R)$ is said to satisfy the chaining condition if there exists a constant $c$ such that for all $x, y \in X$ and all $n \in \N$, there exists $\{x_0, x_1, \ldots, x_n \} \subset X$ with $x_0 = x$ and $x_n = y$ such that
\[
R(x_i, x_{i+1}) \leq c \frac{R(x,y)}{n}.
\]
\end{defn}

It is easy to verify that CC holds for $(\La, R, \rho, \nu)$. Recall from \cite[Corollary 4.4]{RSLTCurKort} that $\La$ is almost surely a length space when endowed with the shortest distance metric $d$. The chaining condition  for $(\La, d, \rho, \nu)$ therefore holds as a straightforward extension of the midpoint condition for length spaces, with $c = 1 + \epsilon$ for any $\epsilon > 0$ (though it actually holds with $c=1$). It hence follows from Corollary \ref{lem:dR compare} that $\La$ endowed with the resistance metric $R$ also satisfies the condition, with $c=2(1+ \epsilon)$ (in fact $c=2$ works) instead.

In the notation of \cite{CroyHKFluctNonUn}, we can take any $\epsilon > 0$ to satisfy point $(i)$ of the conditions given in Section 2 of that paper, and take $b=\epsilon$ to satisfy point $(iii)$. We also let $f_l(r) = C\lr^{-\alpha}$, $f_u(r) = C\lr^{\frac{4\alpha-3}{\alpha -1}}$, and $\beta_l = \beta_u = \alpha$, and $\theta_1 = (3+2\alpha)(2+\alpha)$. The first part of Theorem \ref{thm:main HK global} then follows by a direct application of \cite[Theorem 1]{CroyHKFluctNonUn}, with $\gamma_1 = \theta_1 (\alpha + \frac{4\alpha - 3}{\alpha - 1})$.

We can similarly apply the results to get off diagonal heat kernel bounds. Again in the notation of \cite{CroyHKFluctNonUn}, take $\theta_2$ and $\theta_3$ satisfying
\begin{align*}
\theta_2 &> \theta_1 (1+ \alpha), \hspace{10mm} \theta_3 > (3 + 2\alpha)(1 + 2 \alpha^{-1}),
\end{align*}
and let $\gamma_i = \theta_i (\alpha + \frac{4\alpha - 3}{\alpha - 1})$ for $i=2,3$. Theorem \ref{thm:main HK off diag} then follows by a direct application of \cite[Theorem 3]{CroyHKFluctNonUn}.

The results of \cite[Proposition 11]{CroyHKFluctNonUn} can also be applied to give bounds on expected exit times from a ball of radius $r$. Indeed, letting $\tau_A = \inf \{t \geq 0: B_t \notin A \}$ for any $A \subset \La$, we deduce the following.

\begin{prop}
\begin{align*}
\estartb{\tau_{B(x, r)}}{x} &\geq c r^{\alpha + 1} \lr^{-2(\alpha + \frac{4\alpha - 3}{\alpha - 1})(\alpha + 1)}(\log (r^{-1} (\log r^{-1})^{2(\alpha + \frac{4\alpha - 3}{\alpha - 1})}))^{-\alpha} \\
\estartb{\tau_{B(x, r)}}{x} &\leq C r^{\alpha + 1} \lr^{\frac{4\alpha - 3}{\alpha - 1}}.
\end{align*}
\end{prop}

The results of the propositions above all follow from the fact that the global volume fluctuations are at most logarithmic. We can also use the fact that these logarithmic fluctuations are indeed attained infinitely often as $r \downarrow 0$ to deduce that the heat kernel will indeed experience similar fluctuations.

The volume results as stated in Theorem \ref{thm:main local} do not quite fall into the framework of \cite[Theorem 2]{CroyHKFluctNonUn}, since we have only shown that the infimal and supremal volumes achieve extremal logarithmic fluctuations values infinitely often as $r \downarrow 0$, rather than eventually, which is what is required to apply the theorem. However, by repeating the proof given there with our weaker volume assumptions instead we are able to deduce the (weaker) results that make up the second part of Theorem \ref{thm:main HK global}.

Again using \cite{CroyHKFluctNonUn}, the local volume fluctuation results of Theorem \ref{thm:main local} can also be used to bound pointwise fluctuations for the transition density $p_t(x,x)$. However, the conclusions of \cite[Theorem 20]{CroyHKFluctNonUn} also require the condition 
\[
\liminf_{r \downarrow 0} \frac{R (x, B(x,r)^c)}{r} > 0
\]
to hold for $\nu$-almost every $x \in \La$ in order to get lower bounds on the heat kernel. This does not quite hold in our case but from the proof of \cite{CroyHKFluctNonUn}, we see that the following proposition is sufficient. For clarity in the next proof, we let $B_R (x, r)$ (respectively $B_d(x,r)$) denote the open ball of radius $r$ at $x$ defined with respect to the resistance (respectively geodesic) metric.

\begin{prop}\label{prop:HK subsqnce}
Almost surely, taking $c_{\alpha}$ as in Section \ref{sctn:inf UBs}, we have that for $\nu$-almost every $x \in \La$, there exists a sequence $r_n \downarrow 0$ such that both of the following conditions hold:
\begin{enumerate}[(i)]
\item $\nu (B_R(x, r_n)) \leq 2(c_{\alpha} + 1) r_n^{\alpha} (\log \log r_n^{-1})^{-(\alpha - 1)}$ for all $n$,
\item $R_{\text{eff}}(x, B_R(x, r_n)^c) \geq \frac{1}{64} r_n$.
\end{enumerate}
\begin{proof}
The proof uses a standard technique for lower bounding the effective resistance as given in \cite[Lemma 4.5]{BarKumRWIICTrees}, by defining $M(\rho, r)$ to be the smallest number $m$ such that there exists a set $A_r = \{z_1, z_2, \ldots, z_m \}$ such that $d(\rho, z_i) \in [\frac{r}{4}, \frac{3r}{4}]$ for each $i$, and every path $\gamma$ from $\rho$ to $B_d(\rho, r)^c$ must pass through at least one of the points in $A$. The proof of \cite[Lemma 4.5]{BarKumRWIICTrees} combined with Lemma \ref{lem:dR compare} then entails that
\begin{equation}\label{eqn:BarKum M}
R_{\text{eff}}(\rho, B_R(\rho, r)^c) \geq \frac{r}{16M(\rho, r)}.
\end{equation}
The result exactly as stated in \cite{BarKumRWIICTrees} is written for discrete trees. However, by combining with Lemma \ref{lem:dR compare}, the same proof shows that (\ref{eqn:BarKum M}) holds for $\La$, just with an extra factor of 2.

In what follows, we will therefore assume that all distances are defined with respect to the shortest-distance metric $d$. As in earlier sections, we will prove the result at a uniform point $p(U)$, which we can suppose to be the root, and extend to $\nu$-almost every $x \in \La$ by Fubini's theorem. As in Sections \ref{sctn:sup vol UBs} and \ref{sctn:inf UBs}, let $\lambda_r = 2(c_{\alpha} + 1)(\log \log r^{-1})^{\alpha - 1}$, choose $H$ to be the height of the stable tree associated with $\La$, and rescale time by $H^{\frac{-\alpha}{\alpha - 1}}$ and space by $H^{\frac{-1}{\alpha - 1}}$ in the \Levy excursion coding $\La$ to give a new looptree $\La^1$ such that the new underlying tree associated to $\La^1$ has height $1$. From the arguments of Section \ref{sctn:inf UBs}, it follows that almost surely, there exists a sequence $(r_n)_{n \in \N}$ with $r_n \downarrow 0$ such that $|I_{\frac{1}{4}r_nH^{\frac{-1}{\alpha - 1}}}| \leq \frac{3}{4}r_n H^{\frac{-1}{\alpha - 1}}$, and all sublooptrees grafted to the W-loopspine at a point in $I_{\frac{3}{4}R}$ have mass at most $r_n^{\alpha} \lambda_{r_n}^{-1}$ for all $n$. We will show that, with high probability, we also have $R(x, B_d(x, r_n)^c) \geq c r_n$ for each $n \in \N$.

Now let $r = r_n$ for some $n \in \N$, and $R = rH^{\frac{-1}{\alpha - 1}}$. By construction, we then have:
\begin{itemize}
\item $|I_{\frac{1}{4}R}| \leq \frac{3}{4}R$,
\item Any sublooptrees grafted to the W-loopspine at a point in $I_{\frac{3}{4}R}$ have mass at most $r^{\alpha} \lambda_{r}^{-1}$.
\end{itemize}
To bound $M(\rho, r)$, first let $N_r$ denote the number of sublooptrees grafted to the W-loopspine of $\La^1$ at a point in $I_{\frac{3}{4}R}$ and with diameter at least $\frac{3}{4} R$. It follows by construction that any such sublooptrees also have mass at most $R^{\alpha} \lambda_r^{-1}$. Consequently, $N_r$ is stochastically dominated by a Poisson random variable with parameter:
\[
|I_{\frac{1}{4}R}| N\Big(\diam (\tilde{\La}) \geq \frac{3}{4} R, \zeta \leq R^{\alpha} \lambda_r^{-1}\Big),
\]
where $N(\cdot)$ here denotes the \Ito excursion measure, and $\tilde{\La}$ is a looptree coded by an (unconditioned) excursion under $N$. The point is that the two events $\diam (\tilde{\La}) \geq \frac{3}{4} R$ and $\zeta \leq R^{\alpha} \lambda_r^{-1}$ are in conflict with each other and hence the \Ito measure of the given set is small. Indeed, since $N(\cdot)$ codes a Poisson point process, we have the necessary independence from the Poisson thinning property so that:
\begin{align*}
N\Big(\diam (\tilde{\La}) \geq \frac{3}{4} R, \zeta \leq R^{\alpha} \lambda_r^{-1}\Big) \leq N\Big(\diam (\tilde{\La}) \geq \frac{3}{4} R\Big) \prcondb{\nu (B(\rho', R)) \leq R^{\alpha} \lambda_r^{-1}}{\diam(\tilde{\La}) \geq \frac{3}{4}R}{}.
\end{align*}
To bound each of these terms, note first by the scaling property of looptrees and the \Ito measure that
\[
N\Big(\diam (\tilde{\La}) \geq t\Big) = \hat{C}_{\alpha} t^{-1}
\]
for some constant $\hat{C}_{\alpha} \in (0, \infty)$, and hence $N\Big(\diam (\tilde{\La}) \geq \frac{3}{4} R\Big) = \hat{C}_{\alpha} R^{-1}$. Then, by the same arguments used to prove Proposition \ref{prop:inf LB vol loc prob}, we can bound the second term by $Ce^{-c\lambda_r^{\frac{1}{\alpha}}}$, and therefore obtain that
\begin{align*}
&N\Big(\diam (\tilde{\La}) \geq \frac{3}{4} R, \zeta \leq R^{\alpha} \lambda_r^{-1}\Big) \leq \hat{C}_{\alpha} R^{-1} Ce^{-c\lambda_r^{\frac{1}{\alpha}}}.
\end{align*}
It hence follows that $N_r$ is stochastically dominated by a \textsf{Poisson}($C'e^{-c\lambda_r^{\frac{1}{\alpha}}}$) random variable, so
\[
\prb{N_r > 0} \leq C'e^{-c\lambda_r^{\frac{1}{\alpha}}}.
\]
By restricting to a subsequence $(r_{n_l})_{l \geq 1}$ such that $r_{n_l} \leq e^{-e^l}$ for all $l$, we see by Borel-Cantelli that $\prb{N_{r_{l}} > 0 \text{ i.o.}}=0$.

On the event $N_r = 0$, it follows that any path $\gamma$ from $\rho$ to $B_d(\rho, R)$ must leave the ball $B_d(\rho, \frac{1}{4} R)$ at a point on the W-loopspine. We conclude the argument by showing that we can then take a set $A_r$ (which we denote by $A_R$ in the rescaled looptree) with cardinality $2$.

Recall that, by assumption, we also have that $|I_{\frac{1}{4}R}| \leq \frac{3}{4}R$. In particular, we can assume that the particular event defined in Lemma \ref{lem:segment length bound} and then in Section \ref{sctn:inf UBs} which leads to this length bound occurs. Moreover, taking $\lambda=1$ in that proof and $\frac{1}{4}r$ in place of $r$, and defining ``good" and ``goodish" loops as we did there, the proof ensures that the number of goodish loops encountered before we reach a good one is at most $1$. We claim that this implies that $|A_R| \leq 4$.

To see why, we refer to Figure \ref{fig:loopspine cutset}, which shows a representation of (a discrete approximation of) the W-loopspine. Defining good and goodish loops for the radius $\frac{1}{4}R$ as in Lemma \ref{lem:segment length bound}, we will assume a ``worst-case scenario": that there does indeed exist a goodish loop, and that the smaller of the two segments that it is broken into along the W-loopspine is less than $\frac{1}{4}R$ in length. Since $|I_{\frac{1}{4}R}| \leq \frac{3}{4}R$, it follows that all of the loops that fall between the root and this goodish loop, and also between this goodish loop and the good loop, are completely contained within $B_d(\rho, \frac{3}{4}R)$, and hence we cannot exit $B_d(\rho, \frac{1}{4}R)$ at a point within these sequences of smaller loops. We can therefore only exit at points on either the goodish loop or the good loop pictured, so we can add two points in $A_R$ in each of these loops to cover all possible exit routes, as shown. We rescale back to the original looptree to get $A_r$. Note that for any $\epsilon > 0$, it also follows that we can choose these points to be within distance $\frac{1}{4} + \epsilon$ of $\rho$.

In the case that the smaller of the two segments of the goodish loop actually has length larger than $\frac{1}{4}R$, we can repeat the argument by treating the goodish loop as the good loop, and the same result holds.

\begin{figure}[h]
\includegraphics[width=16cm]{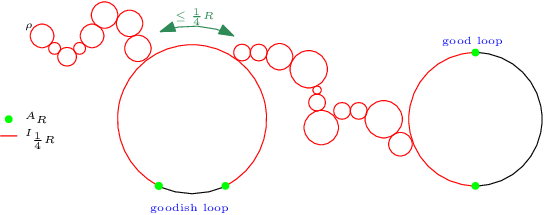}
\centering
\caption{How to select $A_R$. The red segment is a strict subset of $B(\rho, \frac{3}{4}R)$ and contains $B(\rho, \frac{1}{4}R)$. (This is a simplified picture).}\label{fig:loopspine cutset}
\end{figure}

This proves \eqref{eqn:BarKum M}, and we deduce the result as claimed.
\end{proof}
\end{prop}

\begin{rmk}
In \cite[Theorem 6.2]{ArchInfiniteLooptrees}, we prove for infinite stable looptrees that there almost surely exists a constant $c>0$ such that, for all $r > 0$:
\[
cr (\log \log r^{-1})^{\frac{-(3\alpha - 2)}{\alpha - 1}} \leq R(\rho, B(\rho, r)^c).
\]
The argument given there also applies in the compact case, so we deduce the same result for $\La$.
\end{rmk}

Repeating the proof of \cite[Theorem 20]{CroyHKFluctNonUn} along the subsequence of Proposition \ref{prop:HK subsqnce} gives Theorem \ref{thm:main HK local}. Finally, we refer to \cite[Section 7.4]{ArchInfiniteLooptrees} for the details of the proof of Theorem \ref{thm:main annealed HK conv}.

\subsection{Spectral Asymptotics}

Again, we will not give the details of the proof of Theorem \ref{thm:spec asymp}, since it is very similar to the proof of analogous results for stable trees that were the subject of the paper \cite{CroyHamSpectralStable}. The proofs for stable trees involve applying the spinal decomposition that we introduced in Section \ref{sctn:spinal decomp}, and then repeating this decomposition ad infinitum on the resulting fragments to get finer and finer decompositions of the original tree. The eigenvalue counting function $N(\lambda)$ can be estimated by controlling the diameters of these fragments, which is achieved in \cite{CroyHamSpectralStable} by comparison with a Crump-Mode-Jagers process.

The decomposition is the same in the looptree case, but the key difference in the argument arises from the fact that the diameter of a stable tree has finite moments of all orders, whereas the diameter of a stable looptree does not. This means that some propositions from \cite{CroyHamSpectralStable} to do not transfer verbatim to the looptree case; however it turns out that by fine-tuning some of the proofs of \cite{CroyHamSpectralStable}, the analogous statements are all still true in the looptree case (though we take $\gamma = \frac{\alpha}{\alpha + 1}$ instead). For a full understanding of the proof strategy, we refer the reader to \cite{CroyHamSpectralStable}, since it is too long to summarise succintly here. The only difference in the looptree proof is that we must sharpen the upper bound given in the analogous result to \cite[Lemma 4.2]{CroyHamSpectralStable}: in part (i) we do not have sufficiently high moments on the diameter to apply Cauchy Schwarz twice, but this can be rectified by noting that the rescaled diameters appearing in the expectation there are in fact independent of the other terms in the expectation, so can be factorised out directly; in part (ii) the rescaled diameters are not independent of each other, but are still independent of the \textit{other} terms in the expectation, so can be factored out once as a pair and then we only have to apply Cauchy-Schwarz once (rather than appling H\"older twice as in \cite{CroyHamSpectralStable}). This gives a slightly stronger version of Lemma 4.2 for looptrees, and it is possible to chase the remaining arguments of \cite{CroyHamSpectralStable} through in the same way but using this stronger upper bound whenever Lemma 4.2 is applied.

\begin{footnotesize}
\bibliographystyle{alpha}
\bibliography{AIHP1103Archer.bbl}

\end{footnotesize}

\end{document}